\documentclass[11pt,a4paper]{article}
\usepackage{theorem,enumerate}
\usepackage{amsmath,latexsym,amssymb,amsfonts}

\usepackage{eucal}
\usepackage{eufrak}
\usepackage{graphicx}
\usepackage{comment}
\usepackage{pdfpages}
\usepackage{here}
\usepackage{lineno}
\newcommand*\patchAmsMathEnvironmentForLineno[1]{
  \expandafter\let\csname old#1\expandafter\endcsname\csname #1\endcsname
  \expandafter\let\csname oldend#1\expandafter\endcsname\csname end#1\endcsname
  \renewenvironment{#1}
     {\linenomath\csname old#1\endcsname}
     {\csname oldend#1\endcsname\endlinenomath}}
\newcommand*\patchBothAmsMathEnvironmentsForLineno[1]{
  \patchAmsMathEnvironmentForLineno{#1}
  \patchAmsMathEnvironmentForLineno{#1*}}
\AtBeginDocument{
\patchBothAmsMathEnvironmentsForLineno{equation}
\patchBothAmsMathEnvironmentsForLineno{align}
\patchBothAmsMathEnvironmentsForLineno{flalign}
\patchBothAmsMathEnvironmentsForLineno{alignat}
\patchBothAmsMathEnvironmentsForLineno{gather}
\patchBothAmsMathEnvironmentsForLineno{multline}
}
\newtheorem{df}{Definition}
\newtheorem{prop}{Proposition}

\newtheorem{thm}{Theorem}

\newcommand{\qed}{{$\quad\square$\vs{3.6}}}

\newcommand{\vs}[1]{\vspace*{#1 mm}}

\begin{document}

\title{ Ricci curvature of Cayley graphs for dihedral groups, generalized quaternion groups, and cyclic groups} 
\author{
Iwao Mizukai%
\thanks{Graduate school of Science, Mathematical and Science Education, Tokyo University of Science,  E-mail:~1718709@ed.tus.ac.jp}
and Akifumi Sako%
\thanks{Department of Mathematics, Tokyo University of Science, , E-mail:~sako@rs.tus.ac.jp}~
 }

\maketitle
\begin{abstract}
Lin, Lu, and Yau formulated the Ricci curvature
 of edges in simple undirected graphs\cite{Lin}.
Using their formulation, 
we calculate the Ricci curvatures of Cayley graphs for the dihedral groups,
 the general quaternion groups, and the cyclic groups 
with some generating sets chosen so that
their cardinal numbers are less than or equal to four. 
For the dihedral group and the general quaternion group, 
we obtained the Ricci curvatures of all edges of the Cayley graph 
with generator sets consisting of the four elements that 
are the two generators defining each group and 
their inverses elements.
For the cyclic group $({\mathbb Z}/ n{\mathbb Z} , +)$, we have the Ricci curvatures 
of edges of the Cayley graph generating by $S_{1, k}= \{+1, -1, +k, -k \}$.

\end{abstract}

\section{Introduction}
Yann~Ollivier defined a new way to obtain the Ricci curvature 
on the metric space in his paper\cite{olli}.
The Ollivier Ricci curvature method was also applied to graph theory 
by Lin, Lu, and Yau \cite{Lin}. 
 Lin, Lu, and Yau convert the Ollivier Ricci curvatre into the Ricci curvatre.
In the papers\cite{Lin} \cite{smith}, 
they provided the Ricci curvature of graphs such as a path, a cycle, and a complete graph.\\
\bigskip

In this paper, 
we calculate the Ricci curvatures of Cayley graphs for the dihedral groups,
 the general quaternion groups, and the cyclic groups 
with some generating sets chosen so that
their cardinal numbers are less than or equal to four.
The results are shown in the following table \ref{result1} - \ref{result10}.

The dihedral groups are definded by $D_n= <\sigma, \tau | \sigma^{n}=\tau^{2}=e, \tau \sigma
 = \sigma^{n-1}\tau >$ , where $n \geq 3$.
We investigate the Ricci curvatures of Cayley graphs for the dihedral group of 
the minimum  generator set
$S=\{ \tau, \tau^{-1}, \sigma, \sigma^{-1} \} $.
$S$ makes two types of edges in the Cayley graphs.
One type of the edge set is $A= \{ (g, g\sigma)~ |~ g \in D_n \}$, 
and the other is the edge set $B=\{(g, g\tau)~ |~ g \in D_n \}$. 
The results of the Ricci curvatures are given in Table \ref{result1}.
\begin{table}[H]\label{result1}
\centering
  \caption{The Ricci curvature of the Cayley graphs of the dihedral group 
generated set $S$. }
  \begin{tabular}{|l||c|c|c|c|c|c|c|c|c|c|c|c|c|c|c|c|c|c|c|c|c|c|c|c|c|c|c|c|c|c|c|c|}  \hline
  \ \ Cayley graph 
		& $\Gamma(D_{3}, S)$ &$\Gamma(D_{4}, S)$& $\Gamma(D_{5}, S)$ & $\Gamma(D_{n}, S) (n\ge6)$ \\  \hline 
    Ricci curvature (type A) 
		& 1& $\frac{2}{3}$ & $\frac{1}{3}$ & $0$ \\ \hline
     Ricci curvature (type B) 
		& $\frac{2}{3}$ & $\frac{2}{3}$ & $\frac{2}{3}$ &$\frac{2}{3}$  \\ \hline
\end{tabular}
\end{table}

The generalized quaternion groups $Q_{4m}$ are defined by
$Q_{4m}= \langle \sigma, \tau ~|~ \sigma^{2m}= e, \tau^{2}= \sigma^{m},
 \tau^{-1}\sigma \tau = \sigma^{-1} \rangle, where  \  m \geq 2$.
We consider the Cayley graphs for the generalized quaternion groups 
with the generating set $S=\{ \sigma, \tau, \sigma^{-1}, \tau^{-1} \}$.
We distinguish between the two sets of edges.
One is the edge set $A= \{ (g, g \sigma)~ |~ g \in Q_{4m} \}$, 
and the other is the edge set $B=\{(g, g \tau)~ |~ g \in Q_{4m} \}$. 
The results are given in Table \ref{result2}.

\begin{table}[h]
\centering
  \caption{The Ricci curvature of the Cayley graphs of 
$Q_{4m}$ generated set $S$.}
\begin{tabular}{|c|c|c|c|c|c|c|c|c|c|c|c|c|c|c|c|c|c|c|c|c|c|c|c|c|c|c|c|c|c|c|c|c|}  \hline
\ \ Cayley graph 
		& $\Gamma(Q_{8}, S)$ & $\Gamma(Q_{12}, S)$ & $\Gamma(Q_{4m}, S) (m \geq 4)$ \\  \hline 
   Ricci curvature (type A) 
		& $\frac{1}{2}$& $\frac{1}{4}$ & $0$ \\ \hline
   Ricci curvature (type B) 
		& $\frac{1}{2}$& $\frac{1}{2}$ & $\frac{1}{2}$  \\  \hline 
    \end{tabular}
\label{result2}
\end{table}

We consider the Cayley graph for the cyclic group $(\mathbb{Z}/n\mathbb{Z}, +)$
with a generating set $S_{1, k} =\{ +1, +k, -1, -k \}$, 
where $k$ is a positive integer not equal to $1$.
We distinguish between the two sets of edges.
One is the edges $A= \{ (g, g+1)~ |~ g \in \mathbb{Z}/n\mathbb{Z} \}$, 
the other is edges $B=\{(g, g+k)~ |~ g \in \mathbb{Z}/n\mathbb{Z} \}$. 
The Ricci curvatures of the Cayley graphs for the cyclic group 
$\Gamma (\mathbb{Z}/n\mathbb{Z}, S_{1, 2})$ ( $ 6 \leq n \leq 10$)
 with the generating set  $S_{1, 2}=\{ +1,+2, -1, -2 \} $
are given in Table \ref{result3}.

\begin{table}[H]
\centering
  \caption{The Ricci curvature of the Cayley graphs of $\mathbb{Z}/n\mathbb{Z} $ with $S_{1, 2}$.}
  \begin{tabular}{|c|c|c|c|c|c|c|c|c|c|c|c|c|c|c|c|c|c|c|c|}  \hline
   \ \ $n$ & 6&7&8&9&10 \\ \hline
    Ricci curvature (type A)   &1&1 &$\frac{2}{3}$& $\frac{3}{4}$& $\frac{1}{2}$  \\ \hline
    Ricci curvature (type B)  & 1 & $\frac{3}{4}$ &  $\frac{1}{2}$ & $\frac{1}{4}$& $\frac{1}{4}$    \\ \hline
    \end{tabular}
\label{result3}
\end{table}

%Let $\Gamma (\mathbb{Z}/n\mathbb{Z}, S_{1, 2})$ be the Cayley graph 
%of $\mathbb{Z}/n\mathbb{Z}$  ($n \geq 11$) with the generating set $S_{1, 2}$.
The Ricci curvatures of the Cayley graphs $\Gamma (\mathbb{Z}/n\mathbb{Z}, S_{1, 2})$  for $n \geq 11$ 
are given in Table \ref{result4}.

\begin{table}[h]
\centering
  \caption{The Ricci curvature of the Cayley graph of $\mathbb{Z}/n\mathbb{Z}$  ($n \geq 11$)  with  $S_{1, 2}$.}
  \begin{tabular}{|c|c|c|c|c|c|c|c|c|c|c|c|c|c|c|c|c|c|c|c|c|c|c|c|c|c|c|c|c|c|c|c|c|}  \hline
  \ \ Cayley graph 
		& $\Gamma (\mathbb{Z}/n\mathbb{Z}, S_{1, 2})$($n \geq 11$)  \\  \hline 
   Ricci curvature (type A)  
		& $\frac{1}{2}$  \\ \hline
    Ricci curvature (type B)  & $0$ \\ \hline
 \end{tabular}
\label{result4}
\end{table}

The Ricci curvatures of the Cayley graphs 
$\Gamma (\mathbb{Z}/n\mathbb{Z}, S_{1, 3})$
 with the generating set  $S_{1, 3}=\{ +1,+3, -1, -3 \} $ ($6 \leq n \leq 15$)
are given in Table \ref{result5}.

\begin{table}[H]
\centering
  \caption{The Ricci curvature of the Cayley graph of $\mathbb{Z}/n\mathbb{Z}$  ($6 \leq n \leq 15$) 
 generated by $S_{1, 3}$.}
  \begin{tabular}{|c|c|c|c|c|c|c|c|c|c|c|c|c|c|c|c|c|c|c|c|c|c|c|c|c|c|c|c|c|c|c|c|c|}  \hline
   \ \ $n$ &6&7&8&9&10&11&12&13&14&15 \\ \hline 
    Ricci curvature (type A)  & $\frac{2}{3}$ & $\frac{3}{4}$ & $\frac{1}{2}$ & $\frac{1}{2}$ &
   $\frac{1}{2}$ &$\frac{1}{2}$&$\frac{1}{2}$&$\frac{1}{2}$&$\frac{1}{2}$&$\frac{1}{2}$  \\ \hline
    Ricci curvature (type B)  &  $\frac{2}{3}$ & 1 &$\frac{1}{2}$ &$\frac{3}{4}$&
						$\frac{1}{2}$ &$\frac{3}{4}$ &$\frac{1}{2}$&$\frac{1}{4}$&
   					0&$\frac{1}{4}$  \\ \hline
    \end{tabular}
\label{result5}
\end{table}
%Let $\Gamma (\mathbb{Z}/n\mathbb{Z}, S_{1, 3})  (n \geq 16)$ 
%be the Cayley graph with the generating set $S_{1, 3}$.
The Ricci curvatures of the Cayley graphs 
$\Gamma (\mathbb{Z}/n\mathbb{Z}, S_{1, 3})$
 with the generating set  $S_{1, 3}=\{ +1,+3, -1, -3 \} $ ($ n \geq 16$)
are given in Table \ref{result6}.

\begin{table}[h]
\centering
  \caption{The Ricci curvature of the Cayley graph of $\mathbb{Z}/n\mathbb{Z}$
   ($ n \geq 16$) generated set $S_{1, 3}$.}
  \begin{tabular}{|l||c|r|c|c|c|c|c|c|c|c|c|c|c|c|c|c|c|c|c|c|c|c|c|c|c|c|c|c|c|c|c|c|}  \hline
  \ \ Cayley graph 
		& $\Gamma (\mathbb{Z}/n\mathbb{Z}, S_{1, 3}) $ ($ n \geq 16$) \\  \hline 
   Ricci curvature (type A)   
		& $\frac{1}{2}$   \\ \hline
    Ricci curvature (type B)  &  $0$  \\ \hline
 \end{tabular}
 \label{result6}
\end{table}

The Ricci curvatures of the Cayley graphs 
$\Gamma (\mathbb{Z}/n\mathbb{Z}, S_{1, 4})$
 of generating set  $S_{1, 4}=\{ +1, +4, -1, -4 \} $ ($6 \leq n \leq 22$)
are given in Table \ref{result7}.

\begin{table}[H]
\centering
  \caption{The Ricci curvature of the Cayley graph of  $\mathbb{Z}/n\mathbb{Z}$
generated by $S_{1, 4}$}
  \begin{tabular}{|c|c|r|c|c|c|c|c|c|c|c|c|c|c|c|c|c|c|c|c|c|c|c|c|c|c|c|c|c|c|c|c|c|}  \hline
   \ \  $n$ & 6&7&8&9&10&11&12&13&14&15&16&17&18&19&20&21&22 \\  \hline 
     Ricci curvature (type A)  & $\frac{2}{3}$ & $\frac{3}{4}$& 
			        $\frac{1}{4}$ &$\frac{1}{4}$& 
							$\frac{1}{4}$ & $\frac{1}{4}$ &$\frac{1}{4}$ &$\frac{1}{4}$
	 					&$\frac{1}{4}$&$\frac{1}{4}$&$\frac{1}{4}$ &$\frac{1}{4}$&$\frac{1}{4}$
						&$\frac{1}{4}$&$\frac{1}{4}$&$\frac{1}{4}$&$\frac{1}{4}$ \\ \hline
     Ricci curvature (type B) 
		& $\frac{2}{3}$&  1 & $\frac{1}{4}$&$\frac{3}{4}$ &$\frac{1}{4}$ &
   $\frac{1}{2}$&$\frac{3}{4}$&$\frac{1}{2}$ &$\frac{1}{4}$&$\frac{1}{4}$&$\frac{1}{2}$&$\frac{1}{4}$
    & $\frac{1}{4}$&0&$\frac{1}{4}$&0& $\frac{1}{4}$\\ \hline
    \end{tabular}
\label{result7}
\end{table}

The Ricci curvatures of the Cayley graphs 
$\Gamma (\mathbb{Z}/n\mathbb{Z}, S_{1, 4})$
 of generating set  $S_{1, 4}=\{ +1, +4, -1, -4 \}  (n \geq 23)$
are given in Table \ref{result8}.

\begin{table}[H]
\centering
  \caption{The Ricci curvature of the Cayley graph of $\mathbb{Z}/n\mathbb{Z} (n \geq 23)$
generated set $S_{1, 4}$}
  \begin{tabular}{|c|c|r|c|c|c|c|c|c|c|c|c|c|c|c|c|c|c|c|c|c|c|c|c|c|c|c|c|c|c|c|c|c|}  \hline
  \ \ Cayley graph 
		& $\Gamma (\mathbb{Z}/n\mathbb{Z}, S_{1, 4})(n \geq 23)$ \\  \hline 
   Ricci curvature (type A) 
		& $\frac{1}{4}$ \\ \hline
   Ricci curvature (type B)  & $0$  \\ \hline
 \end{tabular}
\label{result8}
\end{table}
%%%%%%%%%%%%%%%%%%%%%%%%%

The Ricci curvature of the edges of the Cayley graph 
$\Gamma (\mathbb{Z}/n\mathbb{Z}, S_{1, 5})$
 with generating set  $S_{1, 5}=\{ +1,+5, -1, -5 \}$
 for $7 \le n \le 25 $
are given in Table \ref{result9}.  
\begin{table}[h]
\centering
  \caption{The Ricci curvature of the Cayley graph that can be generated by $S_{1, 5}$.}
  \begin{tabular}{|c|c|c|c|c|c|c|c|c|c|c|c|c|c|c|c|c|c|c|c|c|c|c|c|c|c|c|c|c|c|c|c|c|}  \hline
   \ \ $n$ &7&8&9&10&11&12&13&14&15&16&17&18&19&20&21&22&23&24&25 \\  \hline 
    Type A & 1& $\frac{1}{2}$ & $\frac{1}{4}$ & $\frac{1}{4}$ & $\frac{1}{2}$ &$\frac{1}{2}$ &$\frac{1}{4}$ &
	 0&0&0&0&0&0&0&0&0&0&0&0 \\ \hline
    Type B & $\frac{3}{4}$&  $\frac{1}{2}$ & $\frac{3}{4}$ &$\frac{3}{4}$ &0&0&$\frac{1}{4}$ &
   $\frac{1}{2}$&$\frac{3}{4}$&$\frac{1}{2}$ &$\frac{1}{4}$&0&$\frac{1}{4}$&$\frac{1}{2}$
    & $\frac{1}{4}$ &0&0&0& $\frac{1}{4}$ \\ \hline
    \end{tabular}
\label{result9}
\end{table}

 %%%%%%%%%%%%%%%%%%%%%%%%%%%%
The Ricci curvatures of the Cayley graphs 
$\Gamma(\mathbb{Z}/n\mathbb{Z}, S_{1, k})$ 
with conditions written in Theorem \ref{thm_z_m}
are zero in table \ref{result10}.
%%%%%%%%%%%%%%%%%%%%%%%%%%
The conditions for type A are $ k \geq  5$, $n \neq 3k-2$, and $n \geq 2k+4$.
The Ricci curvatures of type B edges vanish in the following 4 cases:
1. $k \geq 5$ and  $3k+3 \leq n \leq 4k-2$, 
2. $ k \geq 3$ and $4k+2 \leq n \leq 5k-1$, 
3. $ k \geq 3$ and $n \geq 5k+1$,
4. $ k \geq 6$ and $2k+4 \leq n \leq 3k-3$.

%%%%%%%%%%%%%%%%%%%%%%%%%
\begin{table}[H]
\centering
  \caption{The Ricci curvature of the Cayley graph of the cyclic group 
generated set $S_{1, k}$}
  \begin{tabular}{|c|c|r|c|c|c|c|c|c|c|c|c|c|c|c|c|c|c|c|c|c|c|c|c|c|c|c|c|c|c|c|c|c|}  \hline
  \ \ Cayley graph 
	& $\Gamma (\mathbb{Z}/n\mathbb{Z}, S_{1, k})$ with the above conditions \\  \hline 
   Ricci curvature (type A) 
		& $0$  \\ \hline
    Ricci curvature (type B) & $0$ \\ \hline
	\end{tabular}
\label{result10}
\end{table}

%%%%%%%%%%%%%%%%% plus comment 24_01_11
There are edges that do not satisfy the above any conditions 1-4
in Theorem \ref{thm_z_m}. Such edges have non-zero Ricci curvature in general and their values
depend on relations between $n$ and $k$.
Therefore, determining the curvature of these edges does not seem to be an easy task, as it requires an in-depth discussion of the detailed structure of the Cayley graph, which is determined by $n$ and $k$.
It remains a future work.\\
\bigskip
%%%%%%%%%%%%%%%%%%%%%%%%%%%%%%%%%%%%%

This paper is organized as follows.
In Section \ref{pre}, 
we prepare some definitions and some theorems
used in this paper.
The Ricci curvatures of Cayley graphs are calculated 
for dihedral groups in Section \ref{dihe}, 
for generalized quaternion groups in Section \ref{quate}, 
and for cyclic groups in Section \ref{cyc}.

\section{Preparations}\label{pre}
We review the definitions of the Ricci curvature and the Cayley graphs and related theorems 
to prepare for the calculations in the following sections.
\subsection{Ricci curvature of graph}
Let $G$ be a simple undirected graph with vertex set $V$ and edge set $E$, 
$N_{G}(x)$ be the set of neighbors of $x$, $deg_{G}(x)=|N_{G}(x) |$  be degree
 of vertex $x$, and $d(x, y)$ be the distance between $x$ and $y \in V$.  
\begin{df}\cite{olli}\cite{Lin}\cite{pey}
Let $G$ be a connected graph. A probability measure $\mu$ on 
$G$ is defined as a function $\mu : V \rightarrow [0, 1]$ with
\begin{equation}
	\sum_{x \in V}\mu(x)=1.
\end{equation}
The set of all probability measure on $G$ is  denoted by $P(V) .$
\end{df}

\begin{df}\cite{olli}\cite{Lin}\cite{pey}
Let $G=(V, E)$ be a connected graph and $\mu, \nu \in P(V) $
be two probability measures. The transport plan from  $\mu$ to $\nu$ is a map
$\pi : V \times V \rightarrow [0, 1]$ such that
\begin{equation}
\mu(x) = \sum_{y \in V} \pi(x, y) \hspace{3cm} \forall x \in V,
\end{equation} 
and
\begin{equation}
 \nu(y) = \sum_{x \in V} \pi(x, y) \hspace{3cm} \forall y \in V.
\end{equation} 
The value $\pi (x, y)$ is the mass transported from $x$ to $y$ 
along a geodesic (at cost $d(x, y)\pi(x, y)$ ), and total cost 
of the transport plan $\pi$ is defined as 
\begin{equation}
 cost(\pi) := \sum_{x, y \in V} d(x, y)\pi(x, y).
\end{equation}

We call $d(x, y)\pi(x, y)$ ``transport cost  from $x$ to $y$", 
and we describe it as 
$x \rightarrow y : d(x, y)\pi(x, y)$ for $x, y$ in $V$.
We denote the set of all transport plans from $\mu$ to $\nu$ by $\Pi(\mu, \nu)$.
We define $1$-Wasserstein distance between $\mu$ to $\nu$ as
\begin{equation}\label{W_1_inf}
 W_{1}(\mu, \nu) := \inf_{\pi \in \Pi(\mu, \nu)}  cost(\pi).
\end{equation}

A transport plan $\pi \in \Pi$ is called optimal if we have 
\begin{equation}
 W_{1}(\mu, \nu) = cost(\pi).
\end{equation}
\label{trans}
\end{df}

\begin{df}\cite{olli}\cite{Lin}\cite{pey}
Let  $G =(V, E)$ be a connected graph.
A function $f$ over the vertex set $V$ of $G$ is said to $1$-Lipschitz  
if 
\begin{equation}
|f(u)-f(v)| \leq d(u, v)  \quad \forall u, v \in V.
\end{equation}
\label{Lip}
\end{df}
The following theorem plays an essential role in this study.
\begin{thm}\cite{olli}\cite{Lin}\label{thm_w_1_lower}
Let $G=(V, E)$ be a connected graph. 
For $\mu, \nu \in P(V)$,
%we have
\begin{equation}
W_{1}(\mu, \nu) %= \inf_{\pi \in \Pi(\mu, \nu)}\sum_{x, y \in V}d(x, y)\pi(x, y)
= \sup _{f \in 1-Lip(G)} \sum_{x \in V} f(x)(\mu(x)-\nu(x)).
\label{1lip}
\end{equation}

\end{thm}

\begin{df}\label{def_prob_measure}\cite{olli}\cite{Lin}\cite{pey}
 A probability measure is a function $\mu: V \rightarrow [0, 1]$ with 
$\sum_{x \in V}\mu(x)=1$. To define Ricci curvature for each edge of a graph 
in this paper, 
we only consider probability measures $\mu_{x}^{\alpha}$
in the following form, 

\begin{eqnarray}
\mu_{x}^{\alpha}(v)= \left\{ \begin{array}{lll}
			\alpha  & (v=x), \\
			\frac{1-\alpha}{\deg_{G}(x)} & (v \in N_{G}(x)), \\
			0 & (otherwise) .\\
		\end{array} \right.
\end{eqnarray}
Here, $x, v \in V$ and $\alpha \in [0, 1]$.
\end{df}

\begin{df}\cite{olli} \label{def_alpha_ricci}
Let $G=(V, E)$ be a connected graph. 
For any value $\alpha \in [0, 1]$ and $x, y \in V$, we define the Ollivier Ricci curvature 
$\kappa_{\alpha}(x, y)$ as
\begin{equation}
\kappa_{\alpha}(x, y):= 1-\frac{W_{1}(\mu_{x}^{\alpha}, \mu_{y}^{\alpha})}{d(x, y)}.
\label{or}
\end{equation}

\end{df}

The Ricci curvature is definable on any undirected pair of vertices $x$ and $y$.

\begin{df}\cite{Lin} \label{def_ricci}
The Ricci curvature is defined  from the Ollivier Ricci curvature;
\begin{equation}
\kappa (x, y) = \lim_{\alpha \rightarrow 1} \frac{\kappa_{\alpha}(x, y)}{1-\alpha}.
\end{equation}

\end{df}

\subsection{Cayley graph}
This paper deals with the Cayley graph as an undirected graph.

\begin{df}\cite{mag}
Let $\mathcal{G}$ be a group.
Let $S \subset \mathcal{G}$  be a generating set of $\mathcal{G}$ 
such that $e \notin S$ and $S =S^{-1}$.
Here $S=S^{-1}$ means that $s^{-1} \in S$ if $s \in S$,
and  we call this $S$ symmetric.
If a graph satisfies the following conditions, then the graph is called a Cayley graph, 
and the graph is written  by $\Gamma(\mathcal{G}, S)$;
\begin{enumerate}
\item There exist a bijection between the group $\mathcal{G}$ and  
the set of vertices of graph $\Gamma(\mathcal{G}, S)$.
The vertex corresponding to $g$ is labeled as $g$.

\item 

The set of  edges of graph $\Gamma(\mathcal{G}, S)$ is 
defined as $\{ (g, gs) ~ | ~g \in \mathcal{G}, ~s \in S \}$.

\end{enumerate}
\end{df}

For example, refer to Magnus' book, which provides a detailed introduction to the properties of Cayley graphs \cite{mag}. \\

In this paper, we calculate the Ricci curvature of Cayley graphs defined by dihedral groups, general quaternion groups, and cyclic groups.

\section{The Ricci curvature of the Cayley graph of dihedral groups} \label{dihe}
The purpose of this section is to calculate the Ricci curvatures for edges in Cayley graphs for dihedral groups. \\

\ A dihedral group is related to the congruent transformation of regular polygons. 
The dihedral group is generated by two generators $\sigma, \tau$.
$\sigma$ represents the action of rotation of the angle $\frac{\pi}{n}$ around the center of gravity, and $\tau$ represents the transformation of reflection around a certain axis. 

\begin{df}
The dihedral group is defined by $D_n= <\sigma, \tau | \sigma^{n}=\tau^{2}=e, \tau \sigma
 = \sigma^{n-1}\tau >$ , 
where $n \geq 3$.
\end{df}

We investigate Ricci curvatures of  Cayley graphs for the dihedral group of 
the generating set
$S=\{ \tau, \tau^{-1}, \sigma, \sigma^{-1} \} $.
Here, the generating set
S makes two types of edges in Cayley graphs.
One type of the edge set is $A= \{ (g, g\sigma)~ |~ g \in D_n \}$, 
and the other is the edge set $B=\{(g, g\tau)~ |~ g \in D_n \}$. 
Now, we call the edge in set  $A$ and $B$ type $A$ and type $B$, respectively.\\
\bigskip

At first, let us consider Ricci curvatures of the Cayley graph for the dihedral group $D_3$ 
with the generating set $S=\{ \tau, \tau^{-1}, \sigma, \sigma^{-1} \} $.
\begin{prop} 
Let $\Gamma(D_{3}, S)$ be the Cayley graph of the dihedral group $D_{3}$
with $S = \{ \tau, \tau^{-1}, \sigma, \sigma^{-1} \}$.
The Ricci curvature of any type A  edge in $\Gamma(D_{3}, S)$ is $\kappa = 1$. 
The Ricci curvature of any type B edge in $\Gamma(D_{3}, S)$ is $\kappa = \frac{2}{3}$. \\
\end{prop}

\begin{proof}
First, we prove that the Ricci curvature of  the edge $xy $ in Figure \ref{fig.1} 
as a type A  edge is $\kappa = 1$. 
\begin{figure}[htbp]
  \begin{minipage}[c]{0.5\hsize}
    \centering
    \includegraphics[width=8cm]{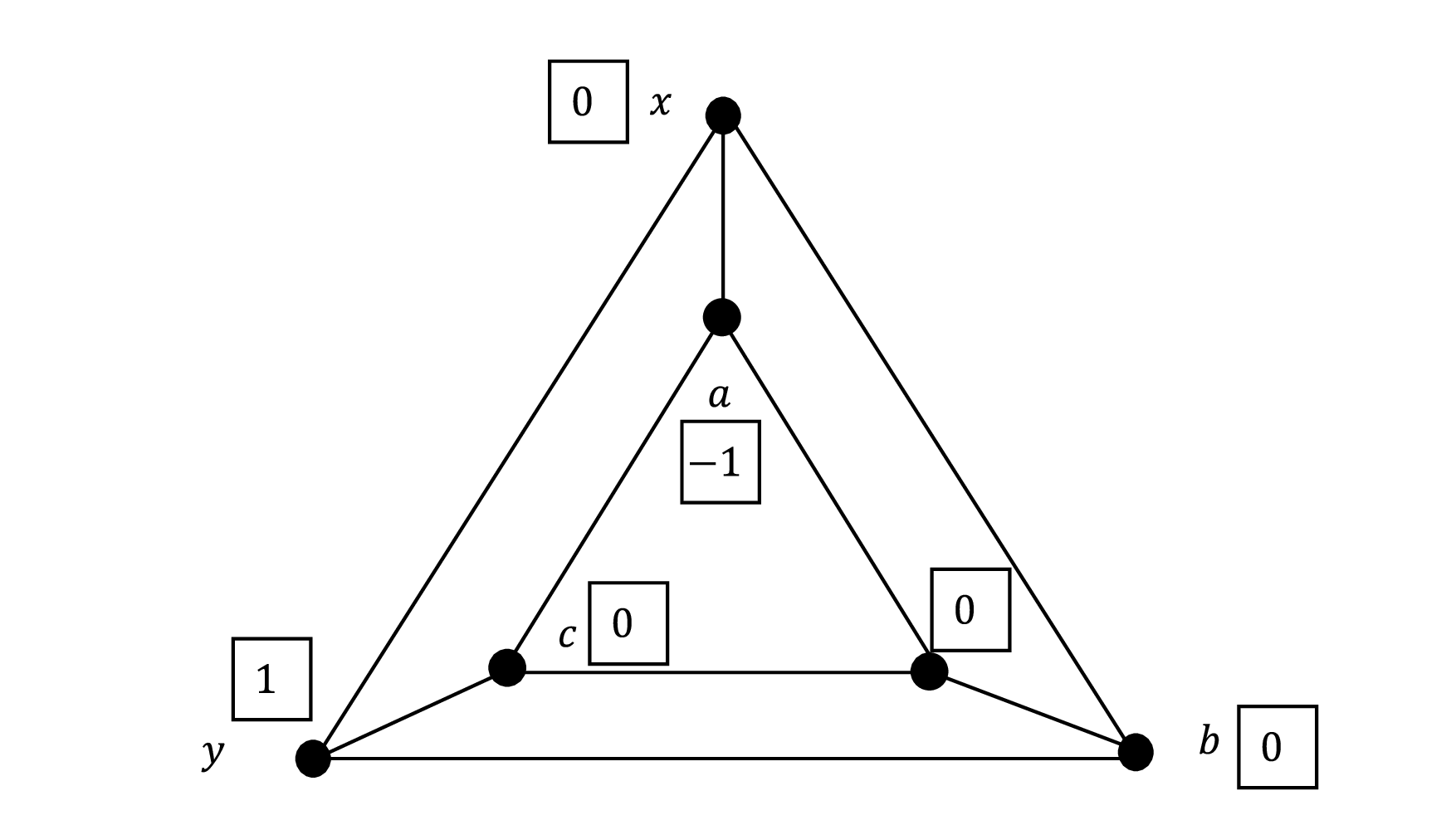}
    \caption{Type A in the Cayley graph of $D_{3}$}
    \label{fig.1}
  \end{minipage}
  \begin{minipage}[c]{0.5\hsize}
    \centering
    \includegraphics[width=8cm]{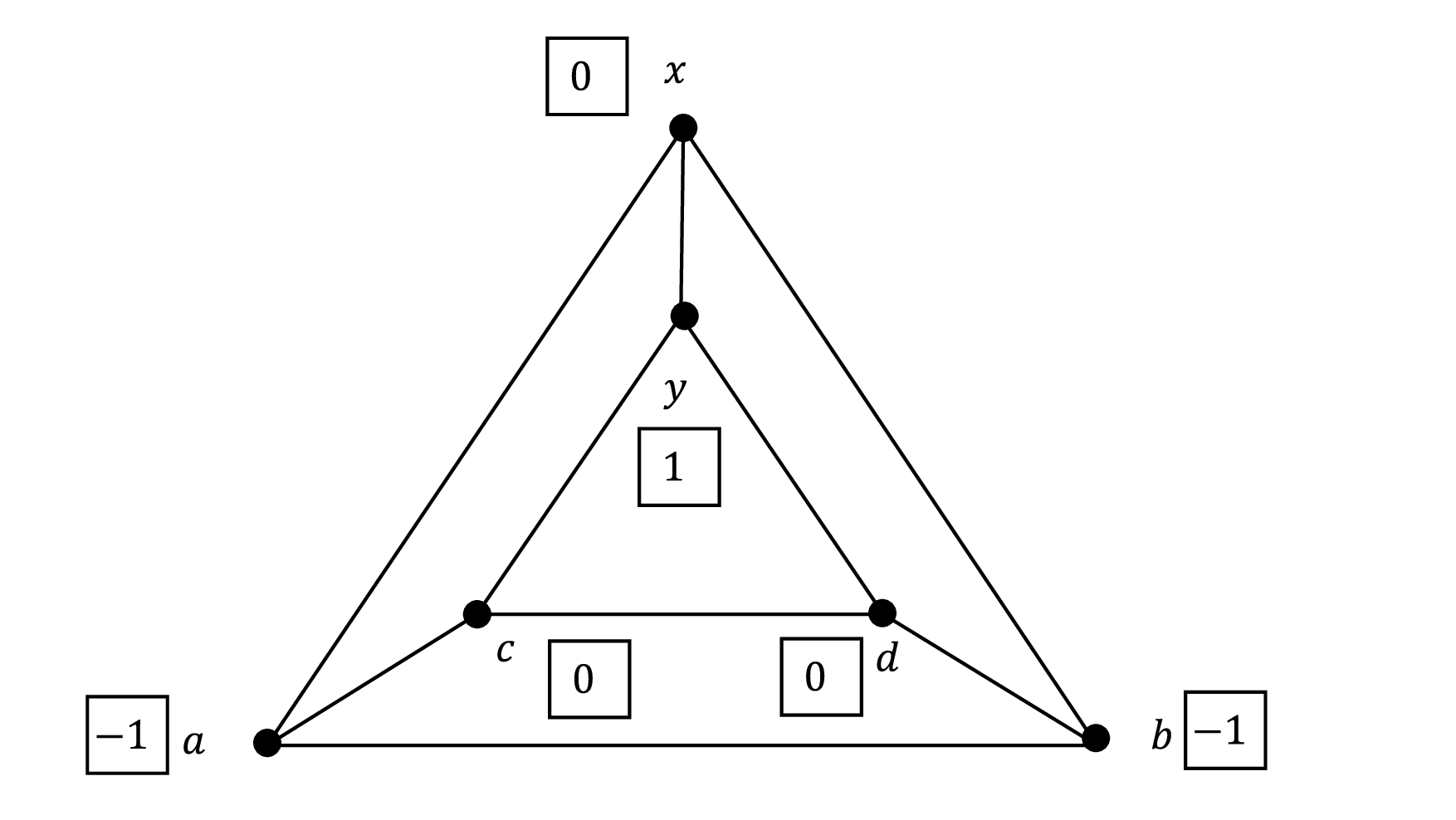}
    \caption{Type B in the Cayley graph of $D_{3}$}
    \label{fig.2}
  \end{minipage}
\end{figure}

We consider the transport cost from 
$\mu^{\alpha}_x$ to $\mu^{\alpha}_y$.
By Definition \ref{def_prob_measure}, the vertex $x$ has probability $\alpha$  for $\mu^{\alpha}_x$.
Each probability of vertex $y, a$ and $b$  
is  $\frac{1-\alpha}{3}$.
On the other hand, vertex $y$ has probability $\alpha$  for $\mu^{\alpha}_y$.
Each probability of vertex $x, b$ and $c$  
is  $\frac{1-\alpha}{3}$.
Satisfying Definition \ref{trans}, we can provide the following  transport cost: 
$x \rightarrow y \ :  1 \times \pi(x, y) = \alpha-\frac{1-\alpha}{3}, \
a \rightarrow c : 1 \times \pi(a, c) = \frac{1-\alpha}{3}$, 
%1 \times \pi(d, h) = \frac{1-\alpha}{3}, \ d \rightarrow h ; 1 \times \pi(b, f) =  \frac{1-\alpha}{3}$, 
and the amount of transport between the other vertices is zero.

From (\ref{W_1_inf}), we can estimate an upper bound of the Wasserstein distance 
for the above transport between probability measures.
We have the following result.
\begin{align}
W_1 \leq  \alpha.
\end{align}
From Definition \ref{def_alpha_ricci},
we have the following result about the Ollivier Ricci curvature.
\begin{align}
\kappa_{\alpha}(x, y) \geq  (1-\alpha).
\end{align}
By Definition \ref{def_ricci}, we have the following result of the lower bound of the Ricci curvature.
\begin{align}
\kappa(x, y) \geq  1.
\label{ineq.1}
\end{align}

Next, using a $1$-Lipschtiz function, we estimate the upper bound of 
the Ricci curvature by Theorem \ref{thm_w_1_lower}.
We define a $1$-Lipschtiz function as Figure $\ref{fig.1}$.
The number in each box beside each vertex is the value of the $1$-Lipschiz function. 
We have the following result from Theorem \ref{thm_w_1_lower} and this $1$- Lipschiz function.
\begin{align}
W_1 \geq  \alpha.
\end{align}
From Definition \ref{def_alpha_ricci},  
we have the following the Ollivier Ricci curvature.
\begin{align}
\kappa_{\alpha}(x, y) \leq  (1-\alpha).
\end{align}
Therefore, we have the upper bound of the Ricci curvature
by Definition \ref{def_ricci},
\begin{align}
\kappa(x, y)  \leq 1.
\label{ineq.2}
\end{align}

By (\ref{ineq.1}) and (\ref{ineq.2}), 
\begin{align}
\kappa(x, y) = 1.
\end{align}
Thus, the curvature of all type A edges is
$\kappa = 1$.\\

Next, we prove that the Ricci curvature of the edge $xy$ in Figure\ref{fig.2}
as a type B edge  is $\kappa = \frac{2}{3}$.
We consider the transport cost from 
$\mu^{\alpha}_x$ to $\mu^{\alpha}_y$.
By Definition \ref{def_prob_measure}, vertex $x$ has probability $\alpha$  for $\mu^{\alpha}_x$.
Each probability of vertex $y, a$ and $b$  
is  $\frac{1-\alpha}{3}$.
On the other hand, vertex $y$ has probability $\alpha$  for $\mu^{\alpha}_y$.
Each probability of vertex $x, c$ and $d$  
is  $\frac{1-\alpha}{3}$.
Satisfying Definition \ref{trans}, we can provide the following  transport cost:
$x \rightarrow y  : 1 \times \pi(x, y) = \alpha-\frac{1-\alpha}{3}, 
a \rightarrow c : 1 \times \pi(a, c) = \frac{1-\alpha}{3}, 
b  \rightarrow d : 1 \times \pi(b, d) = \frac{1-\alpha}{3}$,
and the amount of transport between the other vertices is zero.
From (\ref{W_1_inf}), we can estimate an upper bound of the Wasserstein distance 
for the above transport between probability measures.
We have the following result.
\begin{align}
W_1 \leq  \alpha +\frac{1}{3}(1-\alpha).
\end{align}
From Definition \ref{def_alpha_ricci},
we have the following result about the Ollivier Ricci curvature.
\begin{align}
\kappa_{\alpha}(x, y) \geq  \frac{2}{3}(1-\alpha).
\end{align}
By Definition \ref{def_ricci}, we obtain a lower bound of the Ricci curvature.
\begin{align}
\kappa(x, y) \geq   \frac{2}{3}.
\label{ineq.3}
\end{align}

Next, using a $1$-Lipschtiz function, we estimate an upper bound of 
the Ricci curvature by Theorem \ref{thm_w_1_lower}.
We define a $1$-Lipschtiz function as Figure $\ref{fig.2}$.
The number in each box beside each vertex is the value of the $1$-Lipschiz function.
We have the following result from Theorem \ref{thm_w_1_lower} and this $1$- Lipschiz function.
\begin{align}
W_1 \geq  \alpha+\frac{1}{3}(1-\alpha).
\end{align}
From Definition \ref{def_alpha_ricci},  
we have an upper bound for the Ollivier Ricci curvature
\begin{align}
\kappa_{\alpha}(x, y) \leq  \frac{2}{3}(1-\alpha).
\end{align}
Therefore, we have an upper bound of the Ricci curvature
by Definition \ref{def_ricci},
\begin{align}
\kappa(x, y)  \leq  \frac{2}{3}
\label{ineq.4}
\end{align}

By (\ref{ineq.3}) and (\ref{ineq.4}), 
\begin{align}
\kappa(x, y) =  \frac{2}{3}.
\end{align}
Thus, the curvature of all type B edges is 
$\kappa =  \frac{2}{3}$.

\rightline{\qed}
\end{proof}
\bigskip

Similarly,
let us consider the Ricci curvature of the Cayley graph for dihedral group $D_4$,
$D_5$, $D_6$
with a minimum  generating set $S=\{ \tau, \tau^{-1}, \sigma, \sigma^{-1} \}$.

\begin{prop}\label{D4} 
Let $\Gamma(D_{4}, S)$ be a Cayley graph of dihedral group $D_{4}$
with $S = \{ \tau, \tau^{-1}, \sigma, \sigma^{-1} \}$.
The Ricci curvature $\kappa$ of any edge in $\Gamma (D_{4}, S)$
is $\frac{2}{3}$.
\end{prop}

The proof of this proposition is given in Appendix \ref{prop4proof}.

\begin{prop}\label{D5}
Let $\Gamma(D_{5}, S)$ be a Cayley graph of dihedral group $D_{5}$
with $S = \{ \tau, \tau^{-1}, \sigma, \sigma^{-1} \}$.
The Ricci curvature of any  type $A$ edge in $\Gamma(D_{5}, S)$ is $\kappa = \frac{1}{3}$. 
The Ricci curvature of any  type $B$ edge in $\Gamma(D_{5}, S)$is $\kappa = \frac{2}{3}$. 
\end{prop}

The proof of this proposition is given in Appendix \ref{prop5proof}.

\begin{prop}\label{D6}
Let $\Gamma(D_{6}, S)$ be a Cayley graph of dihedral group $D_{6}$
with $S = \{ \tau, \tau^{-1}, \sigma, \sigma^{-1} \}$
Ricci curvature of the edges  in Cayley graph $\Gamma (D_6, S)$.\\
Ricci curvature of type A   is $\kappa = 0$. \\
Ricci curvature of type B   is $\kappa = \frac{2}{3}$. \\
\end{prop}

The proof of this proposition is given in Appendix \ref{prop6proof}.

\begin{thm}
Let $\Gamma(D_{n}, S)$ be a Cayley graph of dihedral group $D_{n}$
with $S = \{ \tau, \tau^{-1}, \sigma, \sigma^{-1} \}$ for $n \geq 6$. 
The Ricci curvature of any type $A$ edge in $\Gamma(D_{n}, S)$ is $\kappa = 0$. 
The Ricci curvature of any type $B$ edge in $\Gamma(D_{n}, S)$ is $\kappa = \frac{2}{3}$.
\end{thm}

\begin{proof}
First, we prove that  the Ricci curvature of the edge $bc$ in Figure \ref{fig.5}
as a type A edge in $\Gamma(D_{n}, S)$ is $\kappa = 0$.
\begin{figure}[htbp]
\begin{minipage}[c]{0.5\hsize}
   \centering
    \includegraphics[width=6cm]{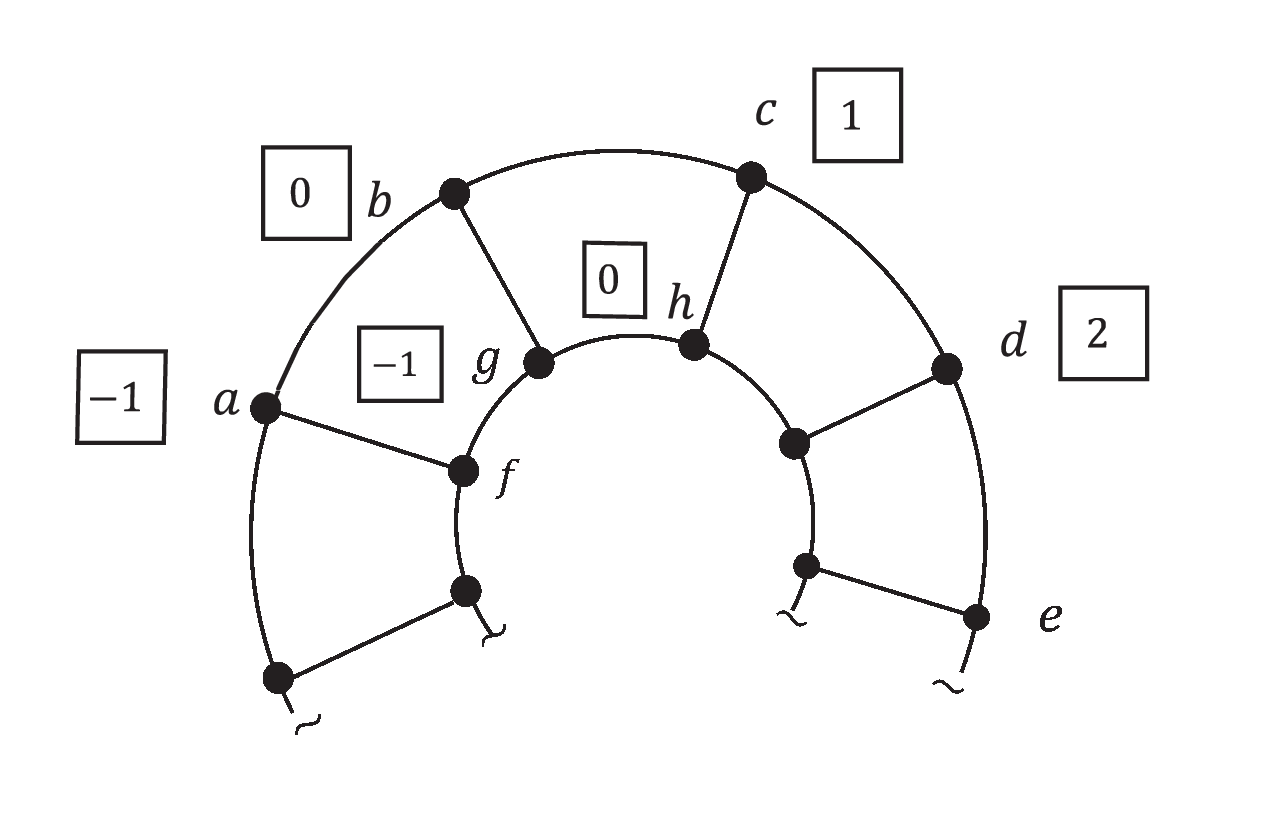}
    \caption{Type A in the Cayley graph of $D_{n}$}
    \label{fig.5}
\end{minipage}
  \begin{minipage}[c]{0.5\hsize}
   \centering
    \includegraphics[width=6cm]{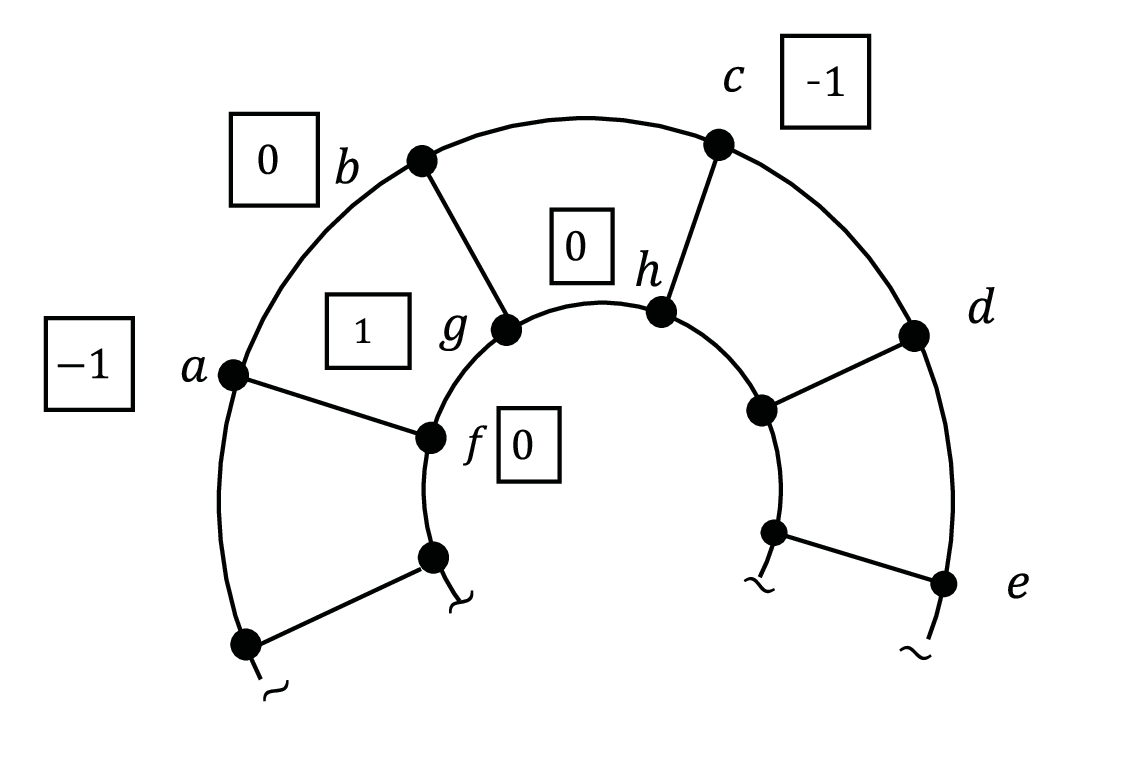}
    \caption{Type B in the Cayley graph of $D_{n}$}
    \label{fig.6}
  \end{minipage}
\end{figure}

Let us consider transport costs.
By Definition \ref{def_prob_measure}, vertex $b$ has probability $\alpha$, and
each vertex $a, c$ and $g$ has probability $\frac{1-\alpha}{3}$ for $\mu^{\alpha}_{b}$.
$\mu^{\alpha}_{c}$ is  determined by the similar way.
Satisfying Definition \ref{trans}, we can provide the following  transport costs:
$a \rightarrow b : \frac{1-\alpha}{3}, \
b \rightarrow c : \alpha, \
 c \rightarrow d :  \frac{1-\alpha}{3}, \
g \rightarrow h : \frac{1-\alpha}{3}$, 
and the amount of transport between the other vertices is zero.
From (\ref{W_1_inf}), we can estimate an upper bound of the Wasserstein distance 
for the above transport between probability measures.
We have the following result.
\begin{align}
W_1 \leq  \alpha + (1-\alpha).
\end{align}
From Definition \ref{def_alpha_ricci},
$\kappa_{\alpha} \geq  0$ about the Ollivier Ricci curvature.
By Definition \ref{def_ricci}, we have the following result of a lower bound of the Ricci curvature.
\begin{align}
\kappa \geq  0.
\label{ineq.13}
\end{align}

Next, using a $1$-Lipschtiz function, we estimate an upper bound of 
the Ricci curvature by Theorem \ref{thm_w_1_lower}.
We define a $1$-Lipschtiz function as Figure $\ref{fig.5}$.
The number in each box beside each vertex is the value of the $1$-Lipschiz function.
The condition $n \geq 6$ is used to define this $1$-Lipschiz function.
We have the following result from Theorem \ref{thm_w_1_lower} and this $1$- Lipschiz function.
\begin{align}
W_1 \geq  \alpha + (1-\alpha).
\end{align}
From Definition \ref{def_alpha_ricci},
$\kappa_{\alpha} \leq 0$ about the Ollivier Ricci curvature.
Therefore, we have the following upper bound of the Ricci curvature.
\begin{align}
\kappa  \leq 0.
\label{ineq.14}
\end{align}
By (\ref{ineq.13}) and (\ref{ineq.14}), the curvature of all type A edges is 
$\kappa = 0$.
\bigskip

Next, we prove that the Ricci curvature of type B is $\kappa =  \frac{2}{3}$.
Now, we calculate the Ricci curvature of the edge $bg$, which is the type$B$,
 in Figure \ref{fig.6}.
We consider transport costs.
By Definition \ref{def_prob_measure}, vertex $b$ has probability $\alpha$,
and each vertex $a, c$ and $g$ has probability $\frac{1-\alpha}{3}$ for $\mu^{\alpha}_{b}$.
$\mu^{\alpha}_{g}$ is determined similarly.
Satisfying Definition \ref{trans}, we can provide the following  transport costs:
$b \rightarrow g : \alpha- \frac{1-\alpha}{3}, \ 
a \rightarrow f :  \frac{1-\alpha}{3}, \
c \rightarrow h : \frac{1-\alpha}{3}$, 
and the amount of transport between the other vertices is zero.
From (\ref{W_1_inf}), we can estimate an upper bound of the Wasserstein distance 
for the above transport between probability measures.
We have the following result.
\begin{align}
W \leq  \alpha + \frac{1-\alpha}{3}.
\end{align}
From Definition \ref{def_alpha_ricci},
$\kappa_{\alpha} \geq  2 \times \frac{1-\alpha}{3}$ about the Ollivier Ricci curvature.
By Definition \ref{def_ricci}, we have the following result of 
the lower bound of the Ricci curvature for type B.
\begin{align}
\kappa \geq  \frac{2}{3}.
\label{ineq.15}
\end{align}

Next, using a $1$-Lipschtiz function, we estimate an upper bound of 
the Ricci curvature by Theorem \ref{thm_w_1_lower}.
We define a $1$-Lipschtiz function as Figure $\ref{fig.6}$.
The number in each box beside each vertex is the value of the $1$-Lipschiz function.
We have the following result from Theorem \ref{thm_w_1_lower} and this $1$- Lipschiz function.

\begin{align}
W_1 \geq  \alpha + \frac{1}{3}(1-\alpha).
\end{align}
From Definition \ref{def_alpha_ricci},  
$\kappa_{\alpha} \leq \frac{2}{3}(1-\alpha)$ about the Ollivier Ricci curvature.
Therefore, we have the following  upper bound of the Ricci curvature
by Definition \ref{def_ricci},
\begin{align}
\kappa  \leq \frac{2}{3}.
\label{ineq.16}
\end{align}
By ($\ref{ineq.15}$) and ($\ref{ineq.16}$), the Ricci curvature of all type B edges is
$\kappa = \frac{2}{3}$.

\rightline{\qed}
\end{proof}

\section{The Ricci curvature of the Cayley graph 
for generalized quaternion groups}\label{quate}
\subsection{generalized quaternion group}

$Q_{8}=\langle i, j , k ~| ~i^2 = j^2 = k^2 =ijk \rangle$ is known as the quaternion group.
 Its generalization is as follows.

\begin{df}\cite{mag}
The generalized quaternion group $Q_{4m}$ is defined by
$$Q_{4m}= \langle \sigma, \tau ~|~ \sigma^{2m}= e, \tau^{2}= \sigma^{m},
 \tau^{-1}\sigma \tau = \sigma^{-1} \rangle, where  \  m > 1.$$
\end{df}

We consider the Cayley graph for the generalized quaternion group 
with the generating set $S=\{ \sigma, \tau, \sigma^{-1}, \tau^{-1} \}$.

We distinguish between the two sets of edges.
One is the edge set $A= \{ (g, g \sigma)~ |~ g \in Q_{4m} \}$, 
and the other is edge set $B=\{(g, g \tau)~ |~ g \in Q_{4m} \}$. 
We call the edge in sets $A$ and $B$ type $A$ and type $B$, respectively.

\subsection{Ricci curvature of Cayley graph for generalized quaternion group $Q_{8}$ 
with the generating set $S = \{ \sigma, \tau, \sigma^{-1}, \tau^{-1} \}$}

\begin{prop}\label{Q4}
For the Cayley graph $\Gamma (Q_8, S)$, 
The Ricci curvature $\kappa$ of any edge in $\Gamma (Q_{8}, S)$
is $\frac{1}{2}$.
\end{prop}

The proof of this proposition is given in Appendix \ref{propQ4proof}.
\bigskip

\begin{prop}\label{Q6}
For the Cayley graph $\Gamma (Q_{12}, S)$, 
The Ricci curvature of any type A edge in $\Gamma (Q_{12}, S)$  
is $\kappa = \frac{1}{4}$, 
and the Ricci curvature of any  type B edge in $\Gamma (Q_{12}, S)$
 is $\kappa = \frac{1}{2}$. 
\end{prop}
 
The proof of this proposition is given in Appendix \ref{propQ6proof}.

The Ricci curvature of Cayley graph for  generalized quaternion group $Q_{4m} (m \geq 4)$
with the generating set $S=\{ \sigma, \tau, \sigma^{-1}, \tau^{-1} \} $ is given as follows:

\begin{thm}
For the Cayley graph $\Gamma (Q_{4m}, S)$ with $m \geq 4$,
the Ricci curvature of any type A edge in $\Gamma (Q_{4m}, S)$ is $\kappa = 0$, 
and the Ricci curvature of any type B edge in $\Gamma (Q_{4m}, S)$ is $\kappa =\frac{1}{2}$. 
\end{thm}

\begin{proof}
\begin{figure}[htbp]
\begin{minipage}[c]{0.5\hsize}
    \centering
    \includegraphics[width=8cm]{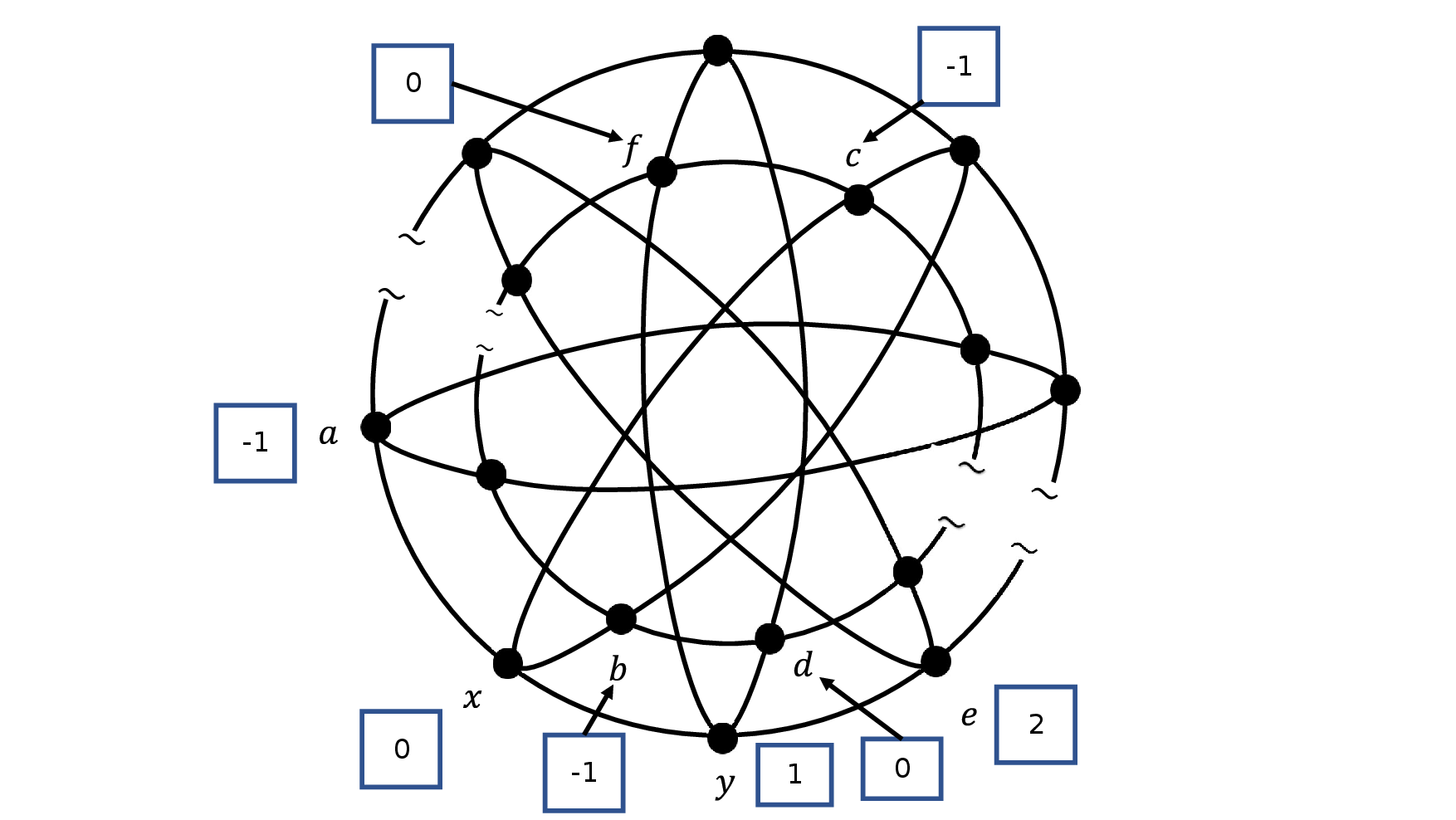}
    \caption{Type A in the Cayley graph of $Q_{4m}$}
    \label{fig.7a}
  \end{minipage}
  \begin{minipage}[c]{0.5\hsize}
    \centering
    \includegraphics[width=8cm]{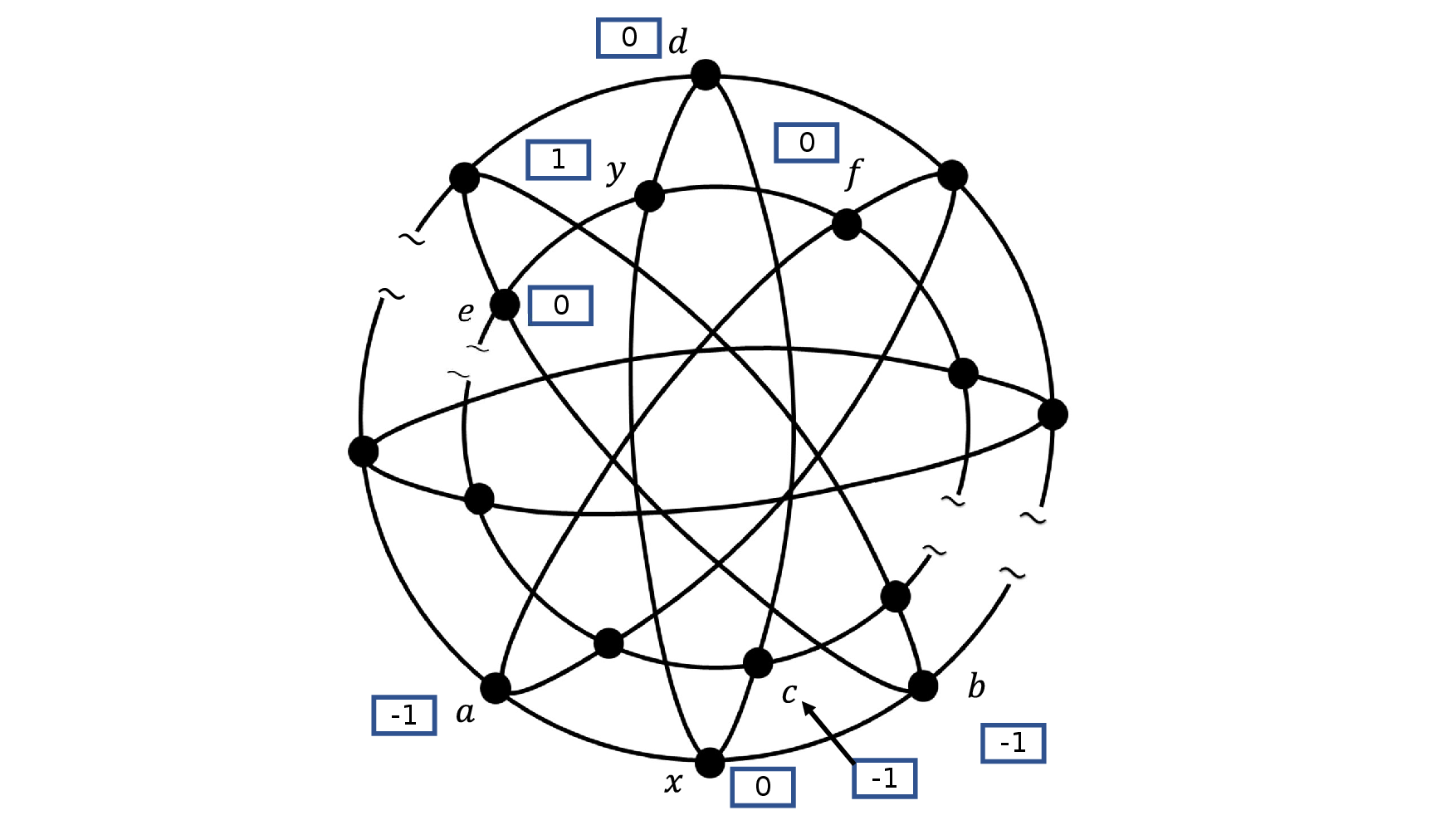}
    \caption{Type B in the Cayley graph of $Q_{4m}$}
    \label{fig.8a}
  \end{minipage}
\end{figure}

First, we prove that the Ricci curvature of type A  is $\kappa = 0$.

Now, we calculate the Ricci curvature of the edge $xy$, which is the type A, i.e. $y=x\sigma$ 
in Figure \ref{fig.7a}.
Satisfying Definition \ref{trans}, we can provide the following  transport cost 
between $\mu^{\alpha}_{x}$ and $\mu^{\alpha}_{y}$ : 
$ a\rightarrow x : \frac{1-\alpha}{4}, \
x \rightarrow  y  :\alpha, \
y \rightarrow e : \frac{1-\alpha}{4}, \
c \rightarrow f: \frac{1-\alpha}{4}, \
b \rightarrow d: \frac{1-\alpha}{4}$, 
and the amount of transport between the other vertices is zero.

From (\ref{W_1_inf}), we can estimate an upper bound of the Wasserstein distance 
for the above transport between probability measures.
We have the following result.
\begin{align}
W_1 \leq  \alpha + (1-\alpha).
\end{align}
From Definition \ref{def_alpha_ricci},
$\kappa_{\alpha} \geq  0$ about the Ollivier Ricci curvature.
By Definition \ref{def_ricci}, we have the following result 
of the lower bound of the Ricci curvature.
\begin{align}
\kappa \geq  0.
\label{ineq.29}
\end{align}

Next, using a $1$-Lipschtiz function, we estimate an upper bound of 
the Ricci curvature by Theorem \ref{thm_w_1_lower}.
We define a $1$-Lipschtiz function as Figure \ref{fig.7a}.
The number in each box beside each vertex is the value of the $1$-Lipschiz function.
We have the following result from Theorem \ref{thm_w_1_lower} and this $1$- Lipschiz function.

\begin{align}
W_1 \geq   \alpha +(1-\alpha).
\end{align}
From Definition \ref{def_alpha_ricci},
$\kappa_{\alpha} \leq 0$ about the Ollivier Ricci curvature.
Therefore, we have the following upper bound of the Ricci curvature.
\begin{align}
\kappa  \leq 0.
\label{ineq.30}
\end{align}

By ($\ref{ineq.29}$) and ($\ref{ineq.30}$),  the curvature of all type A edges is
$\kappa = 0$.
\bigskip

Next, we prove that the Ricci curvature of type B is $\kappa =  0$.
Now, we calculate the Ricci curvature of the edge $xy$, which is the type B as in Figure \ref{fig.8a}.
For $\mu^{\alpha}_{x}$ and $\mu^{\alpha}_{y}$, 
satisfying Definition \ref{trans}, we can provide the following  transport cost:
$ x \rightarrow y :\alpha- \frac{1-\alpha}{4}, \
c \rightarrow  d  :\frac{1-\alpha}{4}, \
a \rightarrow f : \frac{1-\alpha}{4}, \ 
b \rightarrow e: \frac{1-\alpha}{4}$,
and the amount of transport between the other vertices is zero.
From (\ref{W_1_inf}), we can estimate an upper bound of the Wasserstein distance 
for the above transport between probability measures.
We have the following result.
\begin{align}
W_1 \leq  \alpha + \frac{1-\alpha}{2}.
\end{align}
From Definition \ref{def_alpha_ricci},
$\kappa_{\alpha} \geq  \frac{1-\alpha}{2}$ about the Ollivier Ricci curvature.
By Definition \ref{def_ricci}, we have the following result of a lower bound of the Ricci curvature.
\begin{align}
\kappa \geq  \frac{1}{2}.
\label{ineq.31}
\end{align}

Next, using a $1$-Lipschtiz function, we estimate an upper bound of 
the Ricci curvature by Theorem \ref{thm_w_1_lower}.
We define a $1$-Lipschtiz function as Figure $\ref{fig.8a}$.
The number in each box beside each vertex is the value of the $1$-Lipschiz function.
We have the following result from Theorem \ref{thm_w_1_lower} and this $1$- Lipschiz function.

\begin{align}
W_1 \geq  \alpha + \frac{1}{2}(1-\alpha).
\end{align}
From Definition \ref{def_alpha_ricci}, 
$\kappa_{\alpha} \leq \frac{1}{2}(1-\alpha)$ about the Ollivier Ricci curvature.
Therefore, we have the following upper bound of the Ricci curvature.
\begin{align}
\kappa  \leq \frac{1}{2}.
\label{ineq.32}
\end{align}

By (\ref{ineq.31}) and (\ref{ineq.32}), the curvature of all type B edges is 
$\kappa = \frac{1}{2}$.

This is the same result for all other type B edges.
\hspace{\fill}\qed
\end{proof}

\section{The Ricci curvature of the Cayley graph for cyclic groups}\label{cyc}
\subsection{Cyclic group and Cayley graph}
\ 
Every infinite cyclic group is isomorphic to the additive group of $\mathbb{Z}$, the integers. 
Every finite cyclic group of order $n$ is isomorphic to the additive group of $\mathbb{Z}/n\mathbb{Z}$, the integers modulo $n$. 
Every cyclic group is an abelian group meaning that its group operation is commutative, and every finitely generated abelian group is a direct product of cyclic groups.\\

We consider the Cayley graph for the cyclic group $\mathbb{Z}/n\mathbb{Z}$
with a generating set $S_{1, k} =\{ +1, +k, -1, -k \}$, 
where $k$ is a positive integer not equal to $1$.
We distinguish between the two sets of edges.
One is the edges $A= \{ (g, g+1)~ |~ g \in \mathbb{Z}/n\mathbb{Z} \}$, 
the other is edges $B=\{(g, g+k)~ |~ g \in \mathbb{Z}/n\mathbb{Z} \}$. 
We call the edge in sets $A$ and $B$ type $A$ and type $B$, respectively\\

%%%%%%%%%% plus comment 24/1/11
In the following, Ricci curvatures of the type $A$ and type $B$ edges for these Cayley graphs are discussed.
It shall be noted that results for the Ricci curvature of certain Cayley graphs related to $Z_n$ 
have been obtained in \cite{Dagli}. 
However, the Cayley graphs discussed therein are different from those in this paper.

%%%%%%%%%%.

\subsection{The Ricci curvature of the Cayley graph for the cyclic group 
$\Gamma (\mathbb{Z}/n\mathbb{Z}, S_{1})$ with the generating set $S_{1}=\{ +1, -1 \} $} \label{subsection5_2}
Before considering the cases for $S_{1, k}$, 
we make comments about the Ricci curvatures of Cayley graphs for $\Gamma (\mathbb{Z}/n\mathbb{Z})$ with $S_{1}=\{ +1, -1 \}$.
For the case of $\Gamma (\mathbb{Z}/n\mathbb{Z})$
with $S_{1}$,
if $n$ is equal to $1$, then the Cayley graph consists of an isolated vertex. 
At this time, there are no edges. 
For $n=2, 3$ cases, 
we know a useful result.
\begin{thm}\cite{smith}
The complete graph $K_{n}$ has constant the Ricci curvature 
$\kappa = \frac{n}{n-1}$.
\label{complete}
\end{thm}
By Theorem \ref{complete}, we can calculate the following cases.
If $n=2$, the Cayley graph is a complete graph $K_{2}$ with $2$ vertices
having the Ricci curvature $\kappa = 2$.
If $n = 3$, the Cayley graph is a complete graph $K_{3}$ with $3$ vertices 
having the Ricci curvature $\kappa$ is equal to $\frac{3}{2}$.\label{K3}
Also, We know the following results in \cite{Lin}\cite{pey}.
If $n=4$, the Cayley graph is a cycle graph $C_{4}$ with $4$ vertices, 
then the Ricci curvature $\kappa= 1$.
If $n=5$, the Cayley graph is a cycle graph $C_{5}$ with $5$ vertices 
having the Ricci curvature $\kappa = \frac{1}{2}$.
If $n \geq 6$, the Cayley graph is a cycle $C_{n}$ with length $n$
having the Ricci curvature $\kappa =0$.

\subsection{The Ricci curvature of the Cayley graph for the cyclic group $\Gamma (\mathbb{Z}/n\mathbb{Z}, S_{1, 2})$ with the generating set $S_{1, 2}=\{ +1,+2 -1, -2 \} $} 

For $n=3$, $S_{1, 2}=S_{1}$, and then the Cayley graph is given by $K_{3}$ discussed in section 5.2.
Let us consider $\Gamma (\mathbb{Z}/n\mathbb{Z}, S_{1, 2})$
for the cases $n=4, 5$, at first.
For these cases, Theorem \ref{complete}
gives the Ricci curvatures.
If $n = 4$, the Cayley graph is a complete graph $K_{4}$ with $4$ vertices 
having the Ricci curvature $\kappa = \frac{4}{3}$.
If $n = 5$, the Cayley graph is a complete graph $K_{5}$ with $5$ vertices
having the Ricci curvature $\kappa = \frac{5}{4}$.
For $n \geq 6$, we have to distinguish 
between type A and type B edges 
when we calculate the Ricci curvature of their graphs.

%図を書いて説明。

\begin{prop}
The Ricci curvatures of the Cayley graphs for the cyclic group 
$\Gamma (\mathbb{Z}/n\mathbb{Z}, S_{1, 2})$ for $ 6 \leq n \leq 10$
 with the generating set  $S_{1, 2}=\{ +1,+2, -1, -2 \} $
are given in Table\ref{table1}.
\begin{table}[H]
\centering
  \caption{The Ricci curvature of the Cayley graph generated by $S_{1, 2}$.}
  \begin{tabular}{|l||c|r|c|c|c|c|c|c|c|c|c|c|c|c|c|c|c|c|c|}  \hline
   \ \ $n$ & 6&7&8&9&10 \\ \hline
    Type A  &1&1 &$\frac{2}{3}$& $\frac{3}{4}$& $\frac{1}{2}$  \\ \hline
    Type B & 1 & $\frac{3}{4}$ &  $\frac{1}{2}$ & $\frac{1}{4}$& $\frac{1}{4}$    \\ \hline
    \end{tabular}
\label{table1}
\end{table}
\end{prop}
The proof is omitted because it is obtained easily in similar ways
as them in the previous sections.
\bigskip

\begin{prop}\label{s12}
Let $\Gamma (\mathbb{Z}/n\mathbb{Z}, S_{1, 2})$ be a Cayley graph 
of $\mathbb{Z}/n\mathbb{Z}$ for $n \geq 11$ with the generating set $S_{1, 2}$.
The Ricci curvature of any type A edge 
in $\Gamma (\mathbb{Z}/n\mathbb{Z}, S_{1, 2})$ is $\kappa = \frac{1}{2}$. 
The Ricci curvature of any  type B  edge 
in $\Gamma (\mathbb{Z}/n\mathbb{Z}, S_{1, 2})$ is $\kappa =0$. 
\end{prop}

The proof of this proposition is given in Appendix \ref{props12proof}.

\subsection{The Ricci curvature of the Cayley graph for the cyclic group 
$\Gamma (\mathbb{Z}/n\mathbb{Z}, S_{1, 3})$
with the generating set  $S_{1, 3}=\{ +1,+3, -1, -3 \} $} 
By Theorem \ref{complete}, if we can calculate the following case $n=5$, 
the Cayley graph is a complete graph $K_{5}$ with $5$ vertices having 
the Ricci curvature $\kappa=\frac{5}{4}$.
%図を解説する

\begin{prop}
The Ricci curvatures of the Cayley graphs for the cyclic group 
$\Gamma (\mathbb{Z}/n\mathbb{Z}, S_{1, 3})$
 with the generating set  $S_{1, 3}=\{ +1,+3, -1, -3 \} $ for $6 \leq n \leq 15$
are given in Table\ref{table2}.
\begin{table}[H]
\centering
  \caption{The Ricci curvature of the Cayley  graph  generated by $S_{1, 3}$}
  \begin{tabular}{|c||c|c|c|c|c|c|c|c|c|c|c|c|c|c|c|c|c|c|c|c|c|c|c|c|c|c|c|c|c|c|c|c|}  \hline
   \ \ $n$ &6&7&8&9&10&11&12&13&14&15 \\ \hline 
    Type A & $\frac{2}{3}$ & $\frac{3}{4}$ & $\frac{1}{2}$ & $\frac{1}{2}$ &
   $\frac{1}{2}$ &$\frac{1}{2}$&$\frac{1}{2}$&$\frac{1}{2}$&$\frac{1}{2}$&$\frac{1}{2}$  \\ \hline
    Type B &  $\frac{2}{3}$ & 1 &$\frac{1}{2}$ &$\frac{3}{4}$&
						$\frac{1}{2}$ &$\frac{3}{4}$ &$\frac{1}{2}$&$\frac{1}{4}$&
   					0&$\frac{1}{4}$  \\ \hline
    \end{tabular}
\label{table2}
\end{table}
\end{prop}
The proof is obtained from direct calculations.
\bigskip

\begin{prop}\label{s13}
Let $\Gamma (\mathbb{Z}/n\mathbb{Z}, S_{1, 3})  (n \geq 16)$ 
be a Cayley graph with the generating set $S_{1, 3}$.
The Ricci curvature of is any type A edge 
in $\Gamma (\mathbb{Z}/n\mathbb{Z}, S_{1, 3})$ is $\kappa = \frac{1}{2}$. 
The Ricci curvature of any type B edge 
in $\Gamma (\mathbb{Z}/n\mathbb{Z}, S_{1, 3})$ is $\kappa =0$.
\end{prop}
The proof of this proposition is given in Appendix \ref{props13proof}.

\subsection{The Ricci curvature of the Cayley graph for the cyclic group 
$\Gamma (\mathbb{Z}/n\mathbb{Z}, S_{1, 4})$
 with the generating set $S_{1, 4}=\{ +1, +4, -1, -4 \} $} 
For the case $n=5$ and $S_{1, 4}$, 
the Cayley graph is a cycle $C_{5}$ with $5$ vertices having $\kappa = \frac{1}{2}$
 as mentioned in Subsetion \ref{subsection5_2}. 

\begin{prop}
The Ricci curvatures of the Cayley graph for the cyclic group 
$\Gamma (\mathbb{Z}/n\mathbb{Z}, S_{1, 4})$
 of generating set  $S_{1, 4}=\{ +1, +4, -1, -4 \} $ with $6 \leq n \leq 22$
are given in Table \ref{table3}.
\begin{table}[H]
\centering
  \caption{The Ricci curvature of the Cayley  graph that can be generated by $S_{1, 4}$}
  \begin{tabular}{|c|c|c|c|c|c|c|c|c|c|c|c|c|c|c|c|c|c|c|c|c|c|c|c|c|c|c|c|c|c|c|c|c|}  \hline
   \ \  $n$ & 6&7&8&9&10&11&12&13&14&15&16&17&18&19&20&21&22 \\  \hline 
    Type A & $\frac{2}{3}$ & $\frac{3}{4}$& $\frac{1}{4}$ &$\frac{1}{4}$& 
							$\frac{1}{4}$ & $\frac{1}{4}$ &$\frac{1}{4}$ &$\frac{1}{4}$
	 					&$\frac{1}{4}$&$\frac{1}{4}$&$\frac{1}{4}$ &$\frac{1}{4}$&$\frac{1}{4}$
						&$\frac{1}{4}$&$\frac{1}{4}$&$\frac{1}{4}$&$\frac{1}{4}$ \\ \hline
    Type B & $\frac{2}{3}$&  1 & $\frac{1}{4}$&$\frac{3}{4}$ &$\frac{1}{4}$ &
   $\frac{1}{2}$&$\frac{3}{4}$&$\frac{1}{2}$ &$\frac{1}{4}$&$\frac{1}{4}$&$\frac{1}{2}$&$\frac{1}{4}$
    & $\frac{1}{4}$&0&$\frac{1}{4}$&0& $\frac{1}{4}$\\ \hline
    \end{tabular}
\label{table3}
\end{table}
\end{prop}
The proof is obtained from direct calculations.
\begin{prop}\label{s14}
Let $\Gamma (\mathbb{Z}/n\mathbb{Z}, S_{1, 4}) (n \geq 23)$ be 
a Cayley graph with the generating set $S_{1, 4}$.
The Ricci curvature of any type A edge 
in $\Gamma (\mathbb{Z}/n\mathbb{Z}, S_{1, 4})$ is $\kappa = \frac{1}{4}$. 
The Ricci curvature of any type B edge 
in $\Gamma (\mathbb{Z}/n\mathbb{Z}, S_{1, 4})$ is $\kappa =0$. 
\end{prop}
The proof of this proposition is given in Appendix \ref{props14proof}.

\subsection{The Ricci curvature of the Cayley graph for the cyclic group 
$\Gamma (\mathbb{Z}/n\mathbb{Z}, S_{1, 5})$
 with the generating set $S_{1, 5}=\{ +1, +5, -1, -5 \} $} 

\begin{prop}
The Ricci curvature of the edges of the Cayley graph 
for the cyclic group $\Gamma (\mathbb{Z}/n\mathbb{Z}, S_{1, 5}) ~ (7 \le n \le 25)$
 with the generating set  $S_{1, 5}=\{ +1,+5, -1, -5 \}$
are given in Table \ref{table4}.
\begin{table}[h]
\centering
  \caption{The Ricci curvature of the Cayley graph that can be generated by $S_{1, 5}$}
  \begin{tabular}{|c|c|c|c|c|c|c|c|c|c|c|c|c|c|c|c|c|c|c|c|c|c|c|c|c|c|c|c|c|c|c|c|c|}  \hline
   \ \ $n$ &7&8&9&10&11&12&13&14&15&16&17&18&19&20&21&22&23&24&25 \\  \hline 
    Type A & 1& $\frac{1}{2}$ & $\frac{1}{4}$ & $\frac{1}{4}$ & $\frac{1}{2}$ &$\frac{1}{2}$ &$\frac{1}{4}$ &
	 0&0&0&0&0&0&0&0&0&0&0&0 \\ \hline
    Type B & $\frac{3}{4}$&  $\frac{1}{2}$ & $\frac{3}{4}$ &$\frac{3}{4}$ &0&0&$\frac{1}{4}$ &
   $\frac{1}{2}$&$\frac{3}{4}$&$\frac{1}{2}$ &$\frac{1}{4}$&0&$\frac{1}{4}$&$\frac{1}{2}$
    & $\frac{1}{4}$ &0&0&0& $\frac{1}{4}$ \\ \hline
    \end{tabular}
\label{table4}
\end{table}

\end{prop}
The proof is obtained from direct calculations.
For $n \geq 26$, Theorem \ref{thm_z_m} gives the value of Ricci curvature in the following subsection. 

\subsection{The Ricci curvature of the Cayley graph for the cyclic group 
$\Gamma (\mathbb{Z}/n\mathbb{Z}, S_{1, k})$
with the generating set $S_{1, k}=\{ +1, +k, -1, -k \} $}

\begin{thm}\label{thm_z_m}
Let $n$ and $k$ be positive integers satisfying $n > k$.
 Consider the Cayley graph for the cyclic group 
$\Gamma(\mathbb{Z}/n\mathbb{Z}, S_{1, k})$.\\
If $ k \geq  5$, $n \neq 3k-2$, and $n \geq 2k+4$, 
then the Ricci curvature of any type A edge 
in $\Gamma(\mathbb{Z}/n\mathbb{Z}, S_{1, k})$ is $\kappa = 0$.
The Ricci curvature of any type B edge 
in $\Gamma(\mathbb{Z}/n\mathbb{Z}, S_{1, k})$ 
is $\kappa = 0 $ 
if $n$ and $k$ satisfy 
any one of the following conditions :
\begin{enumerate}
\item $k \geq 5$ and  
\begin{align}
3k+3 \leq n \leq 4k-2,
\label{b1}
\end{align} 
\item $ k \geq 3$ and 
\begin{align}
4k+2 \leq n \leq 5k-1, 
\label{b2}
\end{align}
\item $ k \geq 3$ and 
\begin{align}
n \geq 5k+1,
\label{b3}
\end{align}
\item $ k \geq 6$ and 
\begin{align}
2k+4 \leq n \leq 3k-3.
\label{b4}
\end{align}
\end{enumerate}
\end{thm}
%
%%%%%%%%%%%%%%%%%%%%%%%%%%%%%%%%%%%%%%%%%%%%%%%%%%%%
\begin{proof}

\begin{figure}[htbp]
\begin{minipage}[c]{0.5\hsize}
    \centering
    \includegraphics[width=8cm]{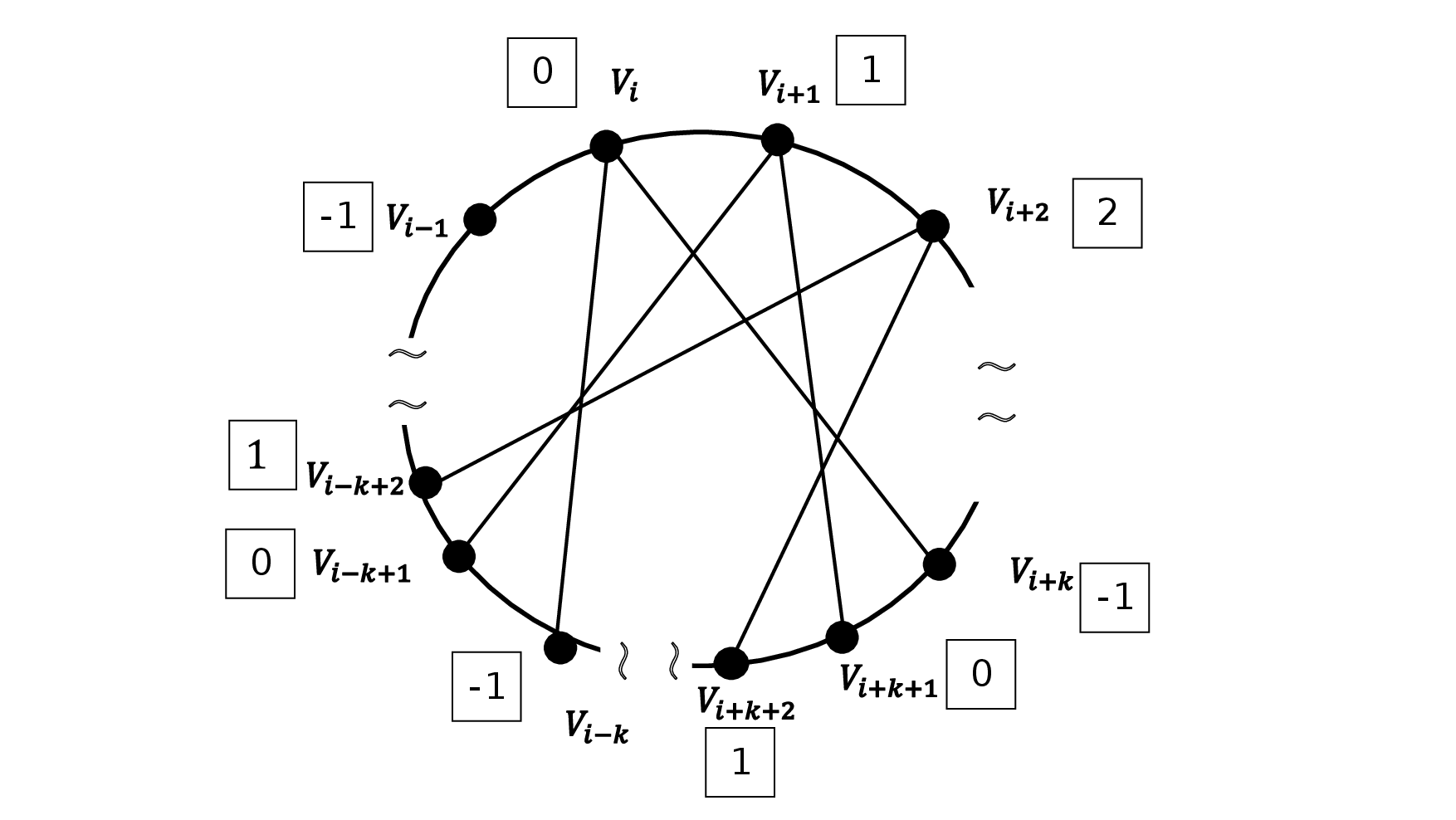}
    \caption{Type A in Cayley graphs of $\mathbb{Z}/n\mathbb{Z}$}
    \label{fig.9}
  \end{minipage}
  \begin{minipage}[c]{0.5\hsize}
    \centering
    \includegraphics[width=8cm]{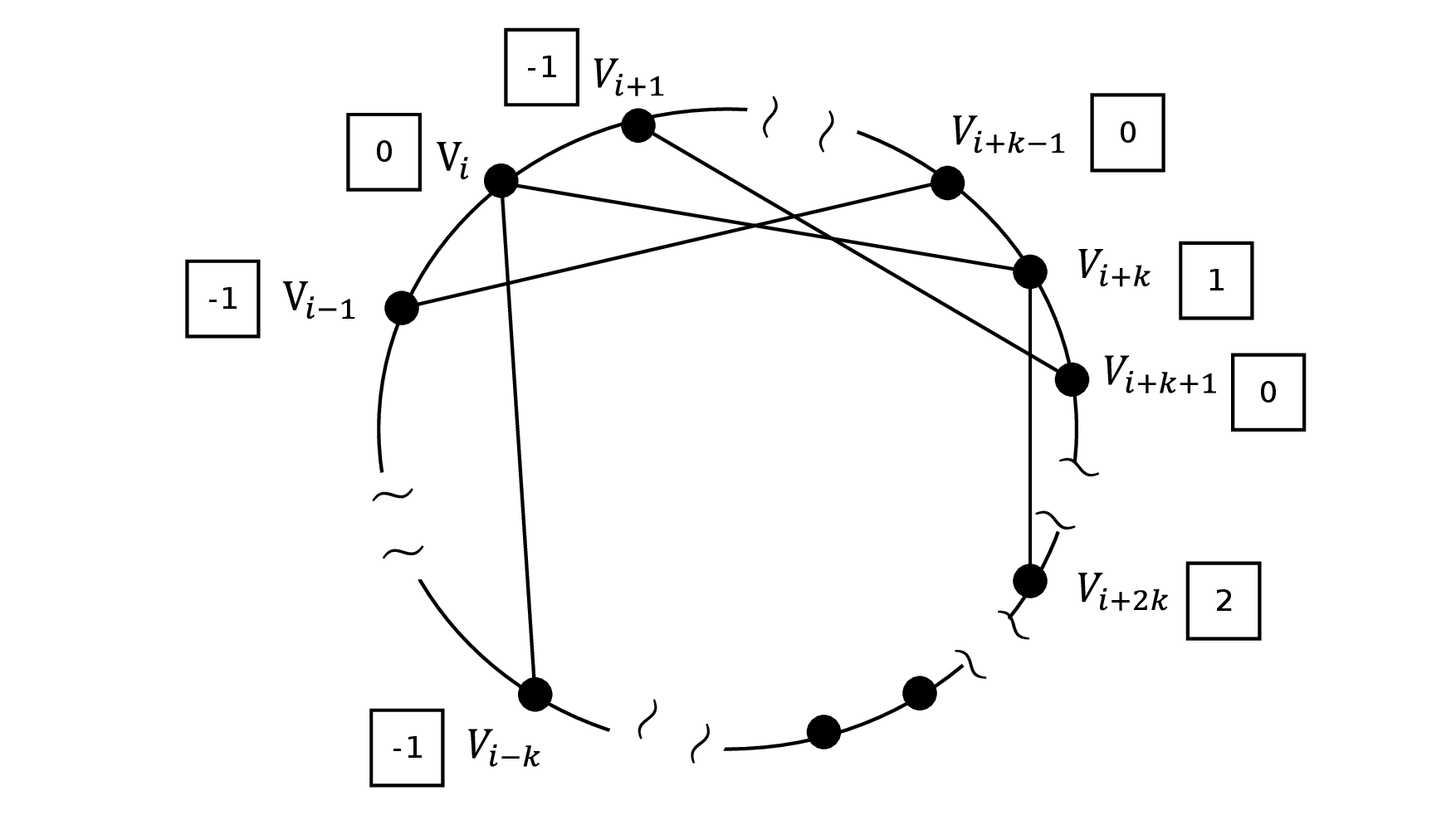}
    \caption{Type B in Cayley graphs of $\mathbb{Z}/n\mathbb{Z}$}
    \label{fig.10}
  \end{minipage}
\end{figure}
First, we calculate the curvature of edge $V_{i}V_{i+1}$ as type A in Figure \ref{fig.9}.
For $\mu^{\alpha}_{V_i}$ and  $\mu^{\alpha}_{V_{i+1}}$, 
we calculate cost $(\pi )$ under the following transport plan :
$ V_ {i-k}\rightarrow V_{i-k+1} : \frac{1-\alpha}{4}, \
V_{i+k} \rightarrow V_{i+k+1} : \frac{1-\alpha}{4}, \
V_{i-1} \rightarrow V_{i} : \frac{1-\alpha}{4}, \
V_{i} \rightarrow V_{i+1} : \alpha, 
V_{i+1} \rightarrow V_{i+2} : \frac{1-\alpha}{4} $, 
and the amount of transport between the other vertices is zero in Figure \ref{fig.9}.
From (\ref{W_1_inf}), we can estimate an upper bound of the Wasserstein distance 
for the above transport between probability measures.
We have the following result.
\begin{align}
W_1 \leq  \alpha + (1-\alpha).
\end{align}
From Definition \ref{def_alpha_ricci}, $\kappa_{\alpha} \geq 0$.
By Definition \ref{def_ricci}, we have the following result of the lower bound of the Ricci curvature.
\begin{align}
\kappa \geq  0.
\label{ineq.46}
\end{align}

Now, we find sufficient conditions under 
which the $1$-Lipschitz function in Figure \ref{fig.9} can be defined.
Let $P_{V_{i+2 }, V_{i+k}}$ be a path that consists of type A edges and
some vertices : 
 $P_{V_{i+2}, V_{i+k}} := V_{i+2}V_{i+3} \cdots V_{i+k-1}V_{i+k}$.
 $|P_{V_{i+2 }, V_{i+k}} |$ denotes the length of the path. Then,
\begin{align}
|P_{V_{i+2 }, V_{i+k}} | =(i+k)-(i+2) \geq 3 \Leftrightarrow k \geq 5. \label{a1} 
\end{align}
Also, from Figure \ref{fig.9},
\begin{align}
d(V_{i-1}, V_{i-k+2}) \geq 2 \Leftrightarrow (i-1) -(i-k+2) \geq 2
\Leftrightarrow k \geq 5
 \label{a11} 
\end{align}
Next, consider the distance between $V_{i-k}, V_{i+k+2}$ 
in Figure \ref{fig.9}.
Let  $P_{V_{i-k}, V_{i+k+1}}$ be a path that consists of type A edges and
some vertices : 
$P_{V_{i-k}, V_{i+k+2}} := V_{i-k}V_{i-k-1} \cdots V_{i+k+3}V_{i+k+2}$.
Then,
\begin{align}
|P_{V_{i-k}, V_{i+k+2}}| =(i-k)-(i+k+2-n)   \geq 2
\Leftrightarrow n \geq 2k+4.
\label{a2} 
\end{align}
By $d(V_{i-k}, V_{i+k+2}) \geq 2$, 
they do not connect by $\pm k$ (by a type B edge), so
\begin{align}
d(V_{i-k}, V_{i+k+2}) =(i-k)-(i+k+2-n)   \neq k
\Leftrightarrow n \neq 3k-2.
\label{a21} 
\end{align}

From the above, 
we have found all the sufficient conditions (\ref{a1}), (\ref{a2}) and (\ref{a21})
for  the $1$-Lipschitz function defined in Figure \ref{fig.9}. In short, the sufficient conditions are
\begin{align}
k \geq 5,  \quad n \neq 3k-2, \quad \mbox{ and }  \quad n \geq 2k+4. \label{AA}
\end{align}
These sufficient conditions (\ref{AA}) are condtions 
for type A  in Theorem \ref{thm_z_m}.
In other words, when the conditions of Theorem \ref{thm_z_m}  for type A  are satisfied, 
the upper bound of the Ricci curvature can be calculated using the Lipschitz function in Figure \ref{fig.9}.
So, using the $1$-Lipschtiz function in Figure \ref{fig.9} we estimate an upper bound of 
Ricci curvature by Theorem \ref{thm_w_1_lower}.
We define a $1$-Lipschtiz function as Figure $\ref{fig.9}$.
The number in each box beside each vertex is the value of the $1$-Lipschiz function.
We have the following result from Theorem \ref{thm_w_1_lower} and this $1$- Lipschiz function.

\begin{align}
W_1 \geq  \alpha +(1-\alpha).
\end{align}
From Definition \ref{def_alpha_ricci} 
we have the following the Ollivier Ricci curvature.
\begin{align}
\kappa_{\alpha}  \leq 0.
\end{align}

Therefore, we have the following  upper bound of the Ricci curvature
by Definition \ref{def_ricci},
\begin{align}
\kappa  \leq  0.
\label{ineq.47}
\end{align}

By (\ref{ineq.46}) and (\ref{ineq.47}), 
the curvature of all type A edges is zero.
\bigskip

Next, we investigate the curvature of edge $V_{i}V_{i+k}$ as type B.
\begin{comment}
From the conditions $ k \geq  5$ and $n \geq 2k+2$,
the Cayley graph is described as
 Figure \ref{fig.10}.
Especially, from $ k \geq  5$,
we find that the 
path $( V_{i+2}, V_{i+3}, \cdots , V_{k-1},  V_{i+k})$
consists of at least four points.
\end{comment}
At first, we find sufficient conditions under 
which the $1$-Lipschitz function in Figure \ref{fig.10} can be defined.

(i) Let $P_{V_{i+1 }, V_{i+k-1}}$ be a path that consists of type A edges and
some vertices: \\
 $P_{V_{i+1 }, V_{i+k-1}} := V_{i+1}V_{i+2} \cdots V_{i+k-2}V_{i+k-1}$.
Then,
\begin{align}
|P_{V_{i+1 }, V_{i+k-1}} |=(i+k-1)-(i+1) \geq 1 \Leftrightarrow k \geq 3 . \label{c1}
\end{align}

(ii) 
 Let $P_{V_{i+k+1 }, V_{i+2k}}$ be a path that consists of type A edges and
some vertices : 
 $P_{V_{i+k+1}, V_{i+2k}} := V_{i+k+1}V_{i+k+2} \cdots V_{i+2k-1}V_{i+2k}$.
 Then,
\begin{align}
|P_{V_{i+k+1 }, V_{i+2k}} | =(i+2k)-(i+k+1) \geq 2 \Leftrightarrow k \geq 3. \label{c2} 
\end{align}

(iii) Let $P_{V_{i-1 }, V_{i-k}}$ be a path that consists of type A edges and
some vertices : 
 $P_{V_{i-1}, V_{i-k}} := V_{i-1}V_{i-2} \cdots V_{i-k+1}V_{i-k}$.
 $|P_{V_{i+k+1 }, V_{i+2k}} |$ denotes the length of the path. Then,
\begin{align}
|P_{V_{i-1 }, V_{i-k}} | =(i-k)-(i+k+1) \geq 0 \Leftrightarrow k \geq 1. \label{c3} 
\end{align}

(iv) Here, we show that $d(V_{i-k}, V_{i+2k})=3$ 
in Figure \ref{fig.10}
when any of the conditions
(\ref{b1}), (\ref{b2}), or (\ref{b3}) in Theorem \ref{thm_z_m}
is satisfied.
In other words, we prove that any of (\ref{b1}), (\ref{b2}), or (\ref{b3})
 is sufficient condition for $d(V_{i-k}, V{i+2k})=3$, 
in the following.\\

The condition $3k+3 \leq n \leq 4k-2$ of 1 in Theorem \ref{thm_z_m} is equivalent to
\begin{align}
%3k+3 \leq n \leq 4k-2
%\Leftrightarrow n \geq 3K+3, n \leq 4k-2 \\
%%\Leftrightarrow n-3k \geq 3, n-4k \leq 2 \\
%\Leftrightarrow
 (i-k)-(i+2k-n) \geq 3, \quad (i+2k-n)-(i-2k) \geq 2.
\label{equation59}
\end{align}
%Note, however, that 
%\begin{align}
%V_{i}=V_{i+n} 
%\end{align}
%holds.
%(For any given $i$)\\

Consider the following paths:
\begin{align}
%P_{1}&:= V_{i-k} V_{i-k-1} \cdots V_{i+2k} \cdots V_{i-2k+1} V_{i-2k},\\
P_{1}&:= V_{i-k} V_{i-k-1} \cdots V_{i-2k+1}V_{i+2k},\\
P_{2}&:= V_{i+2k} V_{i+2k-1}\cdots V_{i-2k+1}V_{i-2k}V_{i-k},
\end{align}
(see Figure\ref{fig.004}).
From (\ref{equation59}), we find that 
\begin{align}
|P_{1}| \geq 3, \quad \mbox{and} \quad
%\label{path2}
|P_{2}| \geq 3.
\label{path1}
\end{align}
% Note, however, that 
%\begin{align}
%V_{i+2k}=V_{i+2k-n} 
%\end{align}
%holds.
%Here, we define the number of edges between vertex $V_{i}$ and $V_{j}$ as $|P|.$
\begin{figure}[H]
\centering
\includegraphics[width=8cm]{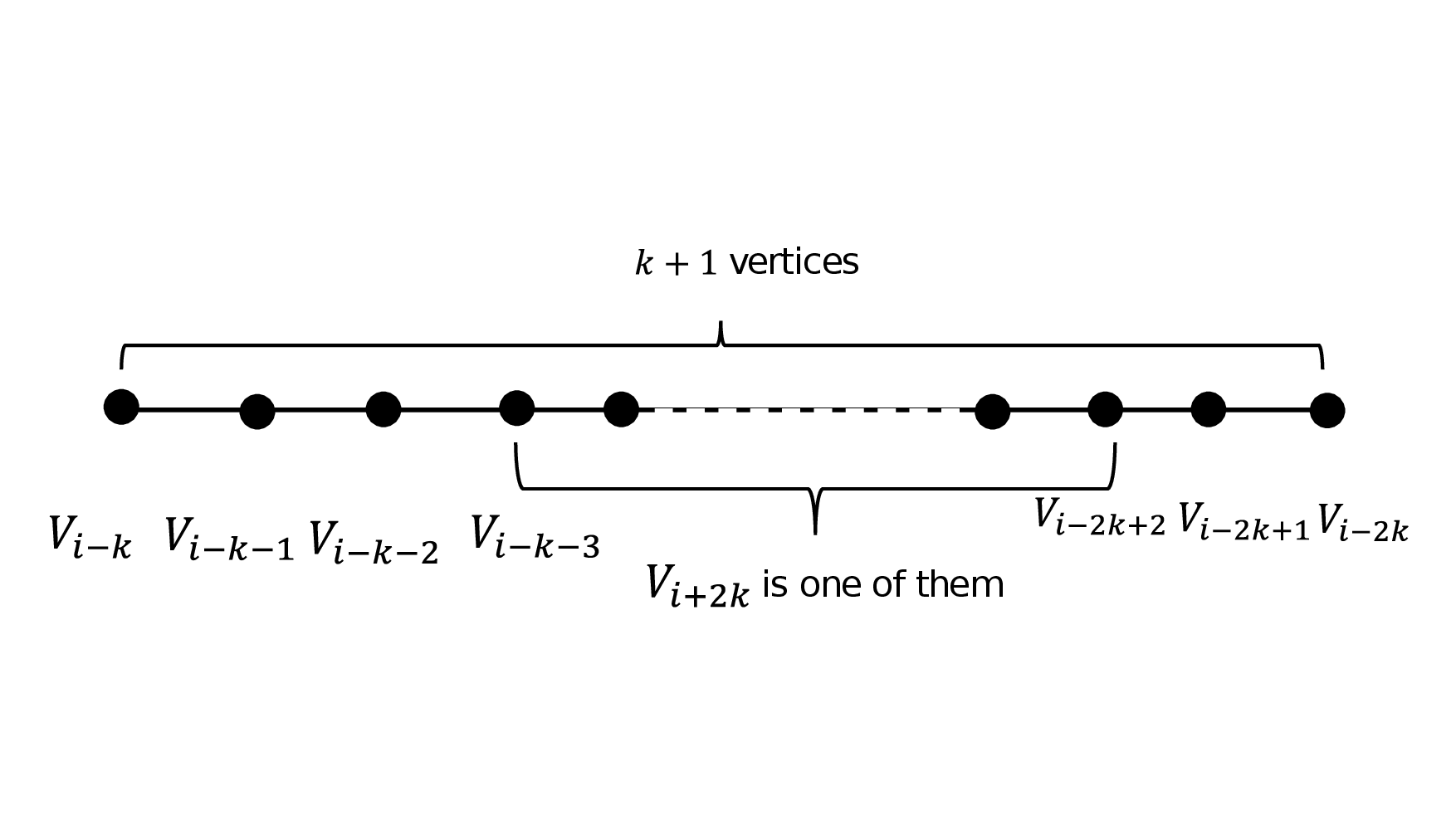} 
\caption{condition $1$}
\label{fig.004}
\end{figure}

From (\ref{path1}) %and (\ref{path3}), 
\begin{align}
d(V_{i-1}, V_{i+2k})=3.
\end{align}
The condition $4k+2 \leq n \leq 5k-2$ of 2 
in Theorem \ref{thm_z_m} is equivalent to
\begin{align}
(i-2k)-(i+2k-n) \geq 2, \quad (i+2k-n)-(i-3k) \geq 1.
\label{equation64}
\end{align}

Consider the following paths:
\begin{align}
P_{3}& := V_{i-k} V_{i-2k}V_{i-2k-1 } \cdots V_{i+2k},
\label{path4} \\
P_{4} &:= V_{i-k}V_{i-2k}V_{i-3k}V_{i-3k+1} \cdots V_{i+2k},
\label{path5}
\end{align}
%Here, we define the number of edges between vertex $V_{i}$ and $V_{j}$ as $|P|.$
(see Figure\ref{fig.005}). 
From (\ref{equation64}), we find that 
\begin{align}
|P_{3}| \geq 3, \quad 
%\label{path4}
|P_{4}| \geq 3.
\label{pathA}
\end{align}
\begin{figure}[H]
\centering
\includegraphics[width=8cm]{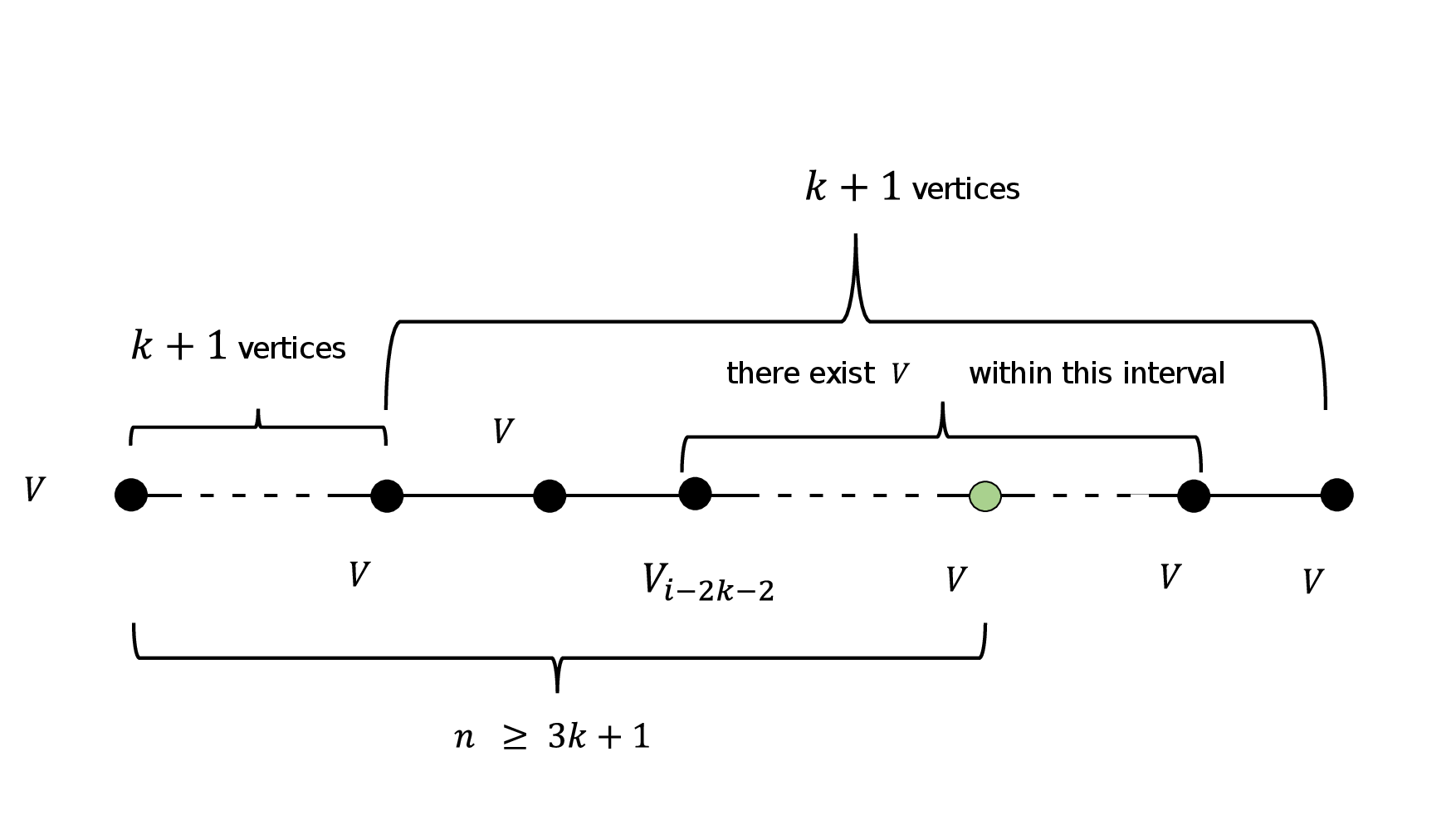} 
\caption{condition $2$}
\label{fig.005}
\end{figure}

%Note, however, that 
%\begin{align}
%V_{i+2k}=V_{i+2k-n} 
%\end{align}
%holds.
From (\ref{pathA}), 
\begin{align}
d(V_{i-1}, V_{i+2k})=3.
\end{align}

The condition $n \geq 5k+1$ of 3 in Theorem \ref{thm_z_m} is equivalent to
\begin{align}
%\Leftrightarrow n-5k \geq 1 \\
%\Leftrightarrow
 (i-3k)-(i+2k-n) \geq 1.
\label{equation69}
\end{align}
\begin{figure}[H]
\centering
\includegraphics[width=8cm]{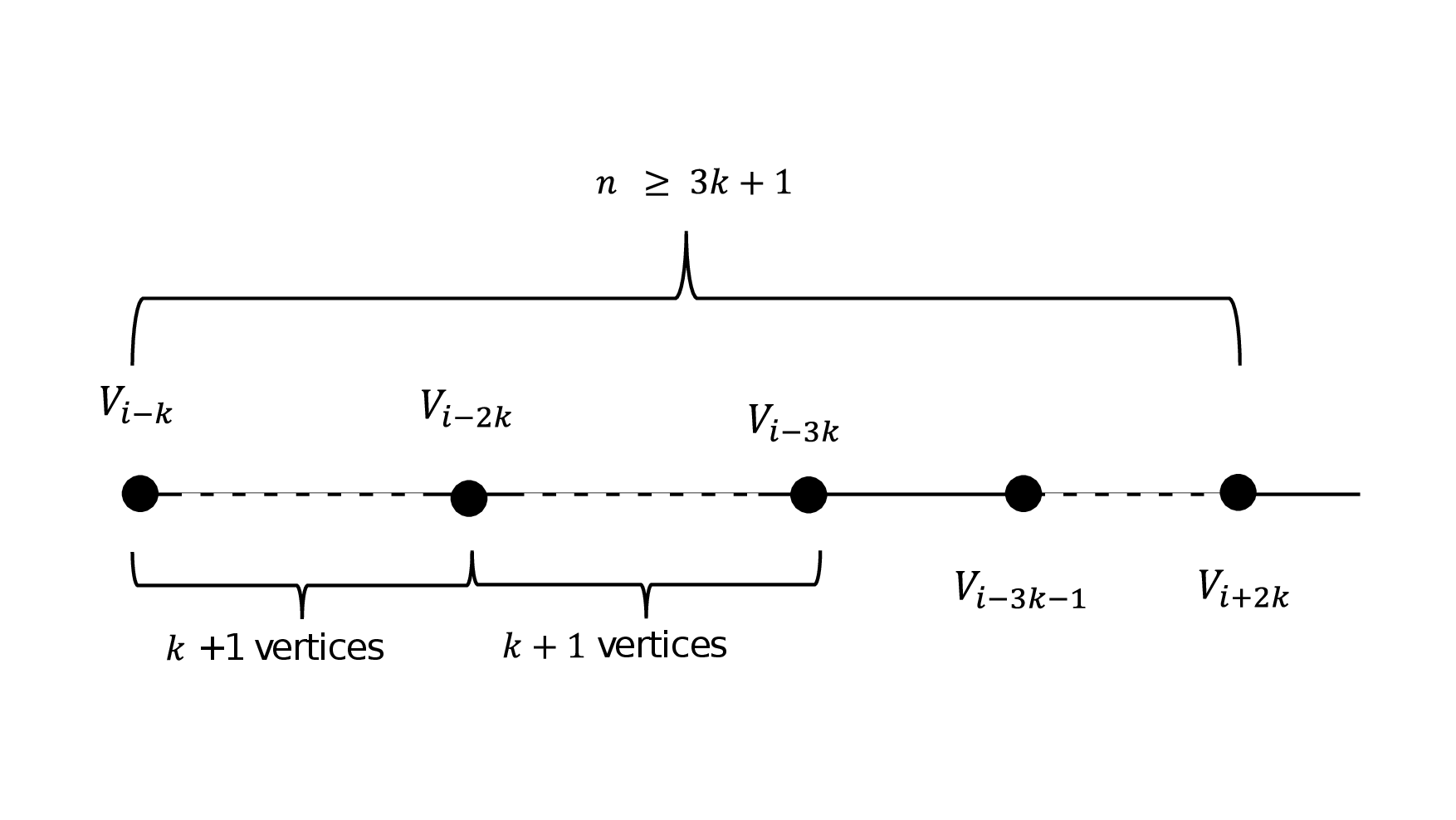} 
\caption{condition $3$}
\label{fig.BB}
\end{figure}
Let us consider the path: 
\begin{align}
P_{5} := V_{i-k} V_{i-2k}V_{i-3k}V_{i-3k-1 } \cdots V_{i+2k},
\label{path6}
\end{align}
(see Figure\ref{fig.BB}). 
From (\ref{equation69}), we find that 
\begin{align}
|P_{5}| \geq 3.
\label{path7}
\end{align}
From (\ref{path7}), 
\begin{align}
d(V_{i-1}, V_{i+2k})=3.
\end{align}
Thus it was shown that any one of (\ref{b1}), (\ref{b2}), or (\ref{b3})
in Theorem \ref{thm_z_m} derives 
$
d(V_{i-k}, V_{i+2k})=3.
$

From the above, 
we have found all the sufficient conditions for  the $1$-Lipschitz function defined in Figure \ref{fig.10}
are (\ref{c1}), (\ref{c2}), (\ref{c3}) 
and  any of (\ref{b1}), (\ref{b2}), or (\ref{b3}):
\begin{align}
k \geq 3 \quad \mbox{ and   any of } (\ref{b1}), (\ref{b2}), \mbox{or} \ (\ref{b3}) . \label{CC}
\end{align}

In other words, when one of the conditions 1 through 3 of Theorem \ref{thm_z_m}  is satisfied, 
the upper bound of Ricci curvature can be calculated using the $1$- Lipschitz function
 in Figure \ref{fig.10}.

%%%%%%%%%%%%%%%%%%%%%%%%%%%%%%%%%%%%%%%
\begin{comment}
From condition 1: $k \geq 5$ and $3k+3 \leq 4k-2$ 
By $k \geq 5 >3$, that is, $k \geq 5$ hold.
and By $3k+3 \leq 4k-2$,
$n \leq 3k+3$ hold.
Therefore, \ref{CC} satisfies condition 1.

From condition 2 : $k \geq 3$ and $4k+2 \leq 5k-1$ , By $3k+3 \leq 4k+2$, 
\begin{align}
n \geq 4k+2 \Leftrightarrow n \geq 3k +(k+2) \Leftrightarrow n \geq 3k+5 > 3k+2
\end{align}
hold.
Therefore, \ref{CC} satisfies condition 2.

From condition 3 : $n \geq 5k+1$, 
By $n \leq 5k+1$, 
\begin{align}
n \geq 5k+1 \Leftrightarrow n \geq 3k +(2k+1) 
\Leftrightarrow n \geq 3k+7 > 3k+3
\end{align}
hold.\\
Therefore, \ref{CC} satisfies condition 3.
From the above, conditions 1 through 3 
satisfy sufficient conditions\ref{CC}, 
so the $1$-Lipschitfunctions in the Figure \ref{fig.25} can be applied.

Similarly.
we find sufficient conditions under 
which the $1$-Lipschitz function in Figure\ref{fig.25} can be defined.
\end{comment}
%%%%%%%%%%%%%%%%%%%%%%%%%%%%%%%%%%%%%%%%%%%%%%%%%%%%%%%

\begin{figure}[H]
\centering
\includegraphics[width=8cm]{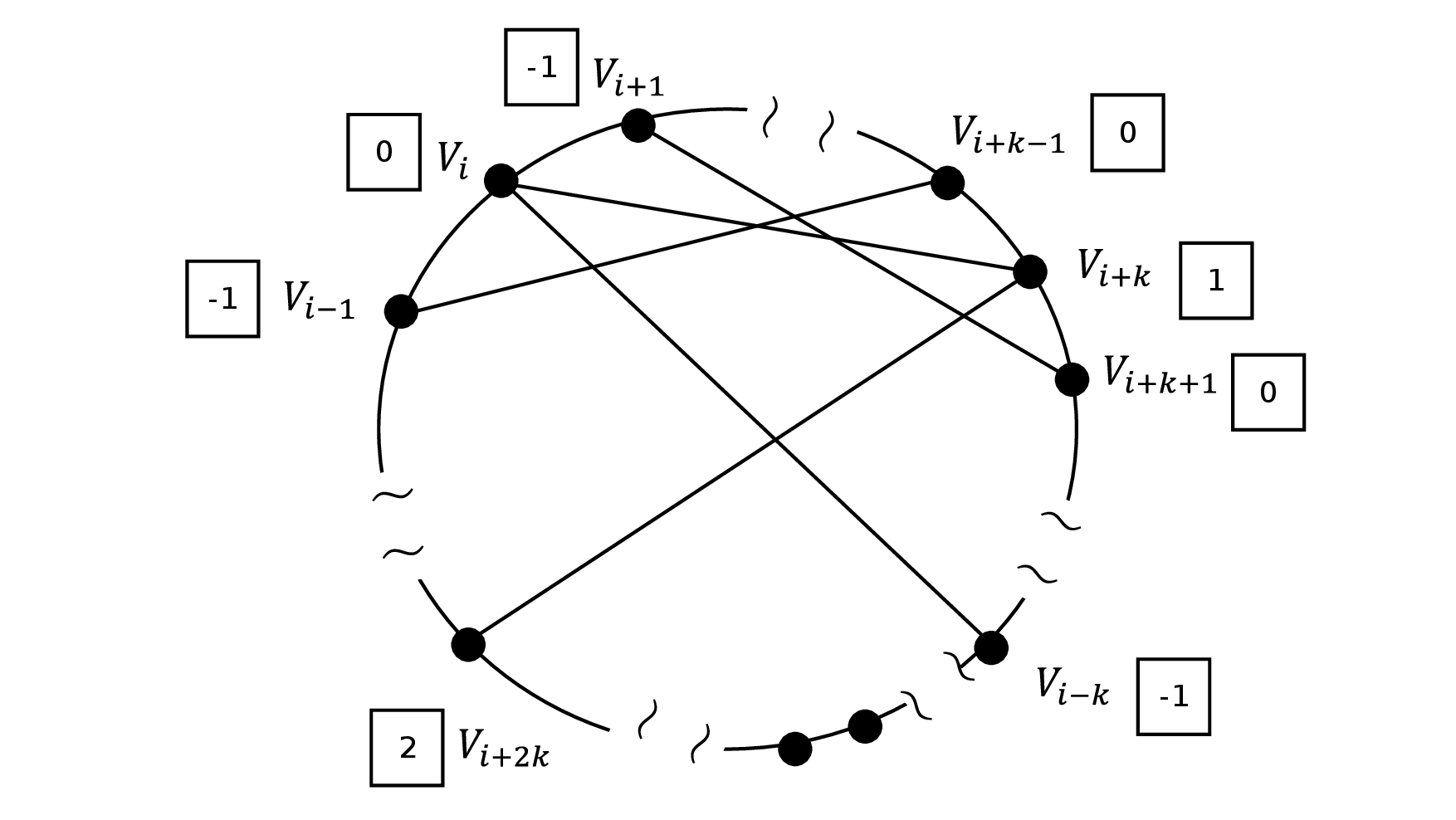} 
\caption{Condition 4 in type B}
\label{fig.222}
\end{figure}
Also, we find sufficient conditions under 
which the $1$-Lipschitz function in Figure\ref{fig.222} can be defined.\\

(v) 
Let $P_{V_{i+k-1 }, V_{i+1}}$ be a path that consists of type A edges and
some vertices: \\
 $P_{V_{i+k-1 }, V_{i+1}} := V_{i+k-1}V_{i+k-2} \cdots V_{i+2}V_{i+1}$.
 $|P_{V_{i+k-1 }, V_{i+1}} |$ denotes the length of the path. Then,
\begin{align}
|P_{V_{i+k-1 }, V_{i+1}} |=(i+k-1)-(i+1) \geq 1 \Leftrightarrow k \geq 3 . \label{c5}
\end{align}
This condition implies $d(V_{i+1}, V_{i+k-1}) \geq 1$.

(vi) Let $P_{V_{i-k}, V_{i+k+1}}$ be a path that consists of type A edges and
some vertices: \\
 $P_{V_{i-k}, V_{i+k+1}} := V_{i-k}V_{i-k-1} \cdots V_{i+k+2}V_{i+k+1}$.
Then,
\begin{align}
|P_{V_{i-k}, V_{i+k+1}} |=(i-k)-(i+k+1-n) \geq 1 \Leftrightarrow n \geq 2k+2 . \label{c6}
\end{align}
Here we use $V_{i-k+1}=V_{i-k+1-n}.$
This condition implies $d(V_{i+k+1}, V_{i-k}) \geq 1$.

(vii) Let $P_{V_{i-1}, V_{i+2k}}$ be a path that consists of type A edges and
some vertices: \\
 $P_{V_{i-1}, V_{i+2k}} := V_{i-1}V_{i-2} \cdots V_{i+2k+1}V_{i+2k}$.
 Then,
\begin{align}
|P_{V_{i-1}, V_{i+2k}} |=(i-1)-(i+2k-n) \geq 3 \Leftrightarrow n \geq 2k+4 . \label{c7}
\end{align}
Here we use $V_{i+2k}=V_{i+2k-n}.$
This condition implies $d(V_{i+2k}, V_{i-1}) \geq 3$.

(viii) Let $P_{V_{i+2k}, V_{i-k}}$ be a path that consists of type A edges and
some vertices: \\
 $P_{V_{i+2k}, V_{i-k}} := V_{i+2k}V_{i+2k-1} \cdots V_{i-k+1}V_{i-k}$.
Then,
\begin{align}
|P_{V_{i+2k}, V_{i-k}} |=(i+2k-n)-(i-k) \geq 3 \Leftrightarrow n \leq 3k-3 . \label{c8}
\end{align}
Here we use $V_{i+2k}=V_{i+2k-n}.$
This is one of the sufficient conditions.

From the above, 
we have found all the sufficient conditions (\ref{c5}), (\ref{c6}), (\ref{c7}) and  (\ref{c8}) 
for  the $1$-Lipschitz function defined in Figure \ref{fig.222}. In short, 
the sufficient conditions are
\begin{align}
k \geq 3 \quad \mbox{ and }  \quad  2k+4 \leq n \leq 3k-3. \label{CCC}
\end{align}
However for $k=3, 4, 5$ there is no $n$ that satisfies (\ref{CCC}).
Thus, we found that the sufficient condition to define 
the $1$-Lipschtiz function in Figure \ref{fig.222}: \\ 
\begin{align}
k \geq 6 \quad \mbox{ and }  \quad  2k+4 \leq n \leq 3k-3. \label{CCCC}
\end{align}

In other words, when condition 4  of Theorem \ref{thm_z_m}  is satisfied, 
the upper bound of Ricci curvature can be calculated 
using the $1$-Lipschitz function in Figure \ref{fig.222}.
\\

We want to adopt the following transportation cost
between $\mu^{\alpha}_{V_{i}}$ and $\mu^{\alpha}_{V_{i+k}}$
for any case 1-4:
\begin{align} 
&V_ {i} \rightarrow V_{i+k} : \alpha-\frac{1-\alpha}{4},
V_{i-1} \rightarrow V_{i+k-1} : \frac{1-\alpha}{4},
V_{i+1} \rightarrow V_{i+k+1} : \frac{1-\alpha}{4},
\label{trans1}\\
&V_{i-k} \rightarrow V_{i+2k} : \frac{1-\alpha}{4} \times 3. 
\label{trans2} \\
&\mbox{All other transport masses are zero.}\label{trans0}
\end{align}
In (\ref{trans1}) above, 
since 
\begin{align}
d(V_{i}, V_{i+k})=d(V_{i+1}, V_{i+k+1})=d(V_{i-1}, V_{i+k-1})=1,
\end{align}
 transport mass in (\ref{trans1}) is proper. 
On the other hand, it is unclear if transport (\ref{trans2}) is proper.
Since the path : $V_{i-k}V_{i}V_{i+k}V_{i+2k}$ exists 
in Figure \ref{fig.10} and  Figure \ref{fig.222}, 
$d(V_{i-k}, V_{i+2k}) \leq 3$ holds.

For cases 1, 2, or 3, as we already saw in 
(iv) $d(V_{i-k}, V_{i+2k})=3$, (\ref{trans2}) is correct.
For the case 4, we have to check $d(V_{i-k}, V_{i+2k})=3$ 
to use (\ref{trans2}).

Let us consider $d(V_{i-k}, V_{i+2k})$ in Figure \ref{fig.222}.
The condition $2k+4 \leq n \leq 3k-3$ of 4 
in Theorem \ref{thm_z_m} is equivalent to
\begin{align}
%2k+4 \leq n \leq 3k-3.\\
%\Leftrightarrow n-2k \geq 4, \\
%3k-n \geq 3. \\
%\Leftrightarrow 
i-(i+2k-n)+1 \geq 5,  \quad
(i+2k-n)-(i-k) \geq 3. 
\label{DD}
\end{align}

Let us consider the following paths:
\begin{align}
P_{6}& := V_{i-k} V_{i-k+1}V_{i-k+2}
 \cdots V_{i+2k-1}V_{i+2k}, \qquad |P_{6}| \geq3. \\
P_{7} &:= V_{i+2k}V_{i+2k+1}
 \cdots V_{i-1}V_{i}V_{i-k}.
\label{path8}
\end{align}
\begin{figure}[H]
\centering
\includegraphics[width=8cm]{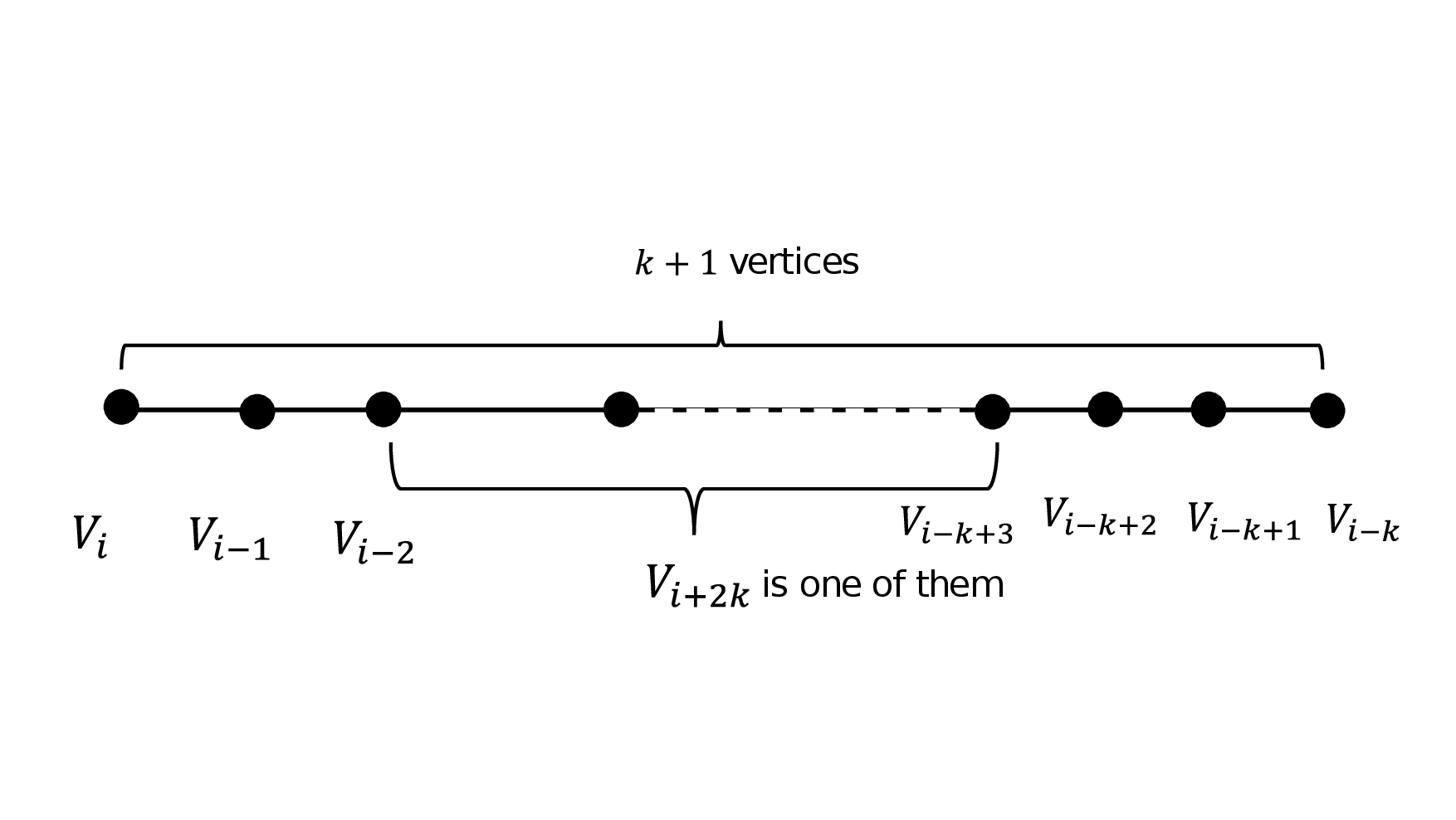} 
\caption{condition $4$}
\label{fig.gg}
\end{figure}
\noindent
(See Figure\ref{fig.gg}). 
 Note that $P_7$ consists not only of type A edges
but also type B edge $V_{i}V_{i-k}$.\\
From (\ref{DD}) we find that
\begin{align}
|P_{6}| \geq 3, \quad \mbox{and} \quad
|P_{7}| \geq 5.
\label{path9}
\end{align}
From (\ref{path9}),  we can conclude
\begin{align}
d(V_{i-k}, V_{i+2k})=3.
\end{align}
Thus it was shown that (\ref{b4}) in Theorem \ref{thm_z_m} 
derives $d(V_{i-k}, V_{i+2k})=3$ in Figure \ref{fig.222}.\\

Therefore, for any one of the conditions 1, 2, 3, or 4 in the Theorem \ref{thm_z_m}, 
the transportation plan (\ref{trans1}), (\ref{trans2}) and (\ref{trans0}) is consistent.
\bigskip

By Definition \ref{trans}, 
we can estimate Wasserstein distance for the transport between probability measures.

The transport cost is given by
(\ref{trans1}), (\ref{trans2}), and (\ref{trans0}).
From (\ref{W_1_inf}), we can estimate an upper bound of the Wasserstein distance.

\begin{align}
W_1 \leq  \alpha + (1-\alpha).
\end{align}
From Definition \ref{def_alpha_ricci},
we have $
\kappa_{\alpha} \geq 0$.
By Definition \ref{def_ricci}, we have the following result of the lower bound of the Ricci curvature.
\begin{align}
\kappa \geq  0.
\label{ineq.48}
\end{align}

We introduced a $1$-Lipschtiz function as Figure \ref{fig.10} or \ref{fig.222}.
The number in each box beside each vertex is the value of the $1$-Lipschiz function.
We have the following result from Theorem $1$ and this $1$- Lipschiz function.

\begin{align}
W_1 \geq  \alpha + (1-\alpha).
\end{align}
From Definition \ref{def_alpha_ricci}, 
we have 
$\kappa_{\alpha}  \leq 0$.
Therefore, we have the following upper bound of the Ricci curvature
by Definition \ref{def_ricci},
\begin{align}
\kappa  \leq  0.
\label{ineq.49}
\end{align}

By (\ref{ineq.48}) and (\ref{ineq.49}), the curvature of all type B edges is $\kappa =  0$.
\qed

\end{proof}

\appendix

\section{proof of propositions}

\subsection{Proof of Proposition \ref{D4}}\label{prop4proof}
\begin{proof}
Cayley graph of $D_{4}$ with $S$ is a regular tetrahedral Figure \ref{fig.3a}.
In the current case, there is no distinction between type A and type B edges.
First, we calculate the Ricci curvature of the edge $ae$ as the type A or B.
\begin{figure}[H]
\centering
\includegraphics[width=8cm]{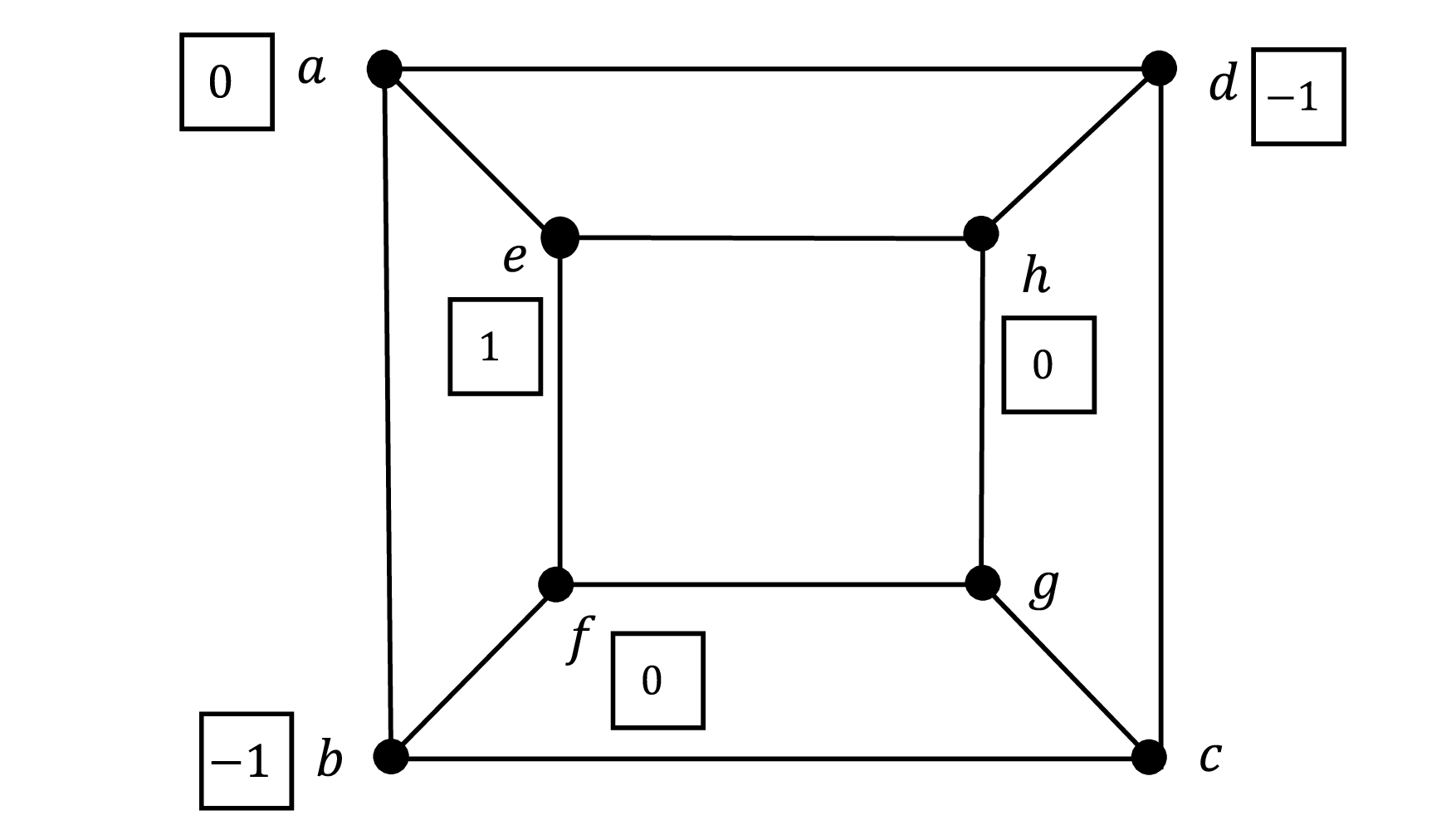} 
\caption{The Cayley graph of $D_{4}$}
\label{fig.3a}
\end{figure}

We consider the transport cost from 
$\mu^{\alpha}_a$ to $\mu^{\alpha}_e$.
By Definition \ref{def_prob_measure}, vertex $a$ has probability $\alpha$  for $\mu^{\alpha}_a$.
Each probability of vertex $b, d$ and $e$  
is  $\frac{1-\alpha}{3}$.
On the other hand, vertex $e$ has probability $\alpha$,
and $a, f$ and $h$ have $\frac{1-\alpha}{3}$ for $\mu^{\alpha}_{e}$.
Satisfying Definition \ref{trans}, we can provide the following  transport cost:
$a \rightarrow e : \ \pi(a, e) = \alpha-\frac{1-\alpha}{3}, \
b \rightarrow f : \pi(b, f) = \frac{1-\alpha}{3}, \
\ d \rightarrow h : \pi(d, h) = \frac{1-\alpha}{3}$, 
and the amount of transport between the other vertices is zero.
From (\ref{W_1_inf}), we can estimate an upper bound of the Wasserstein distance 
for the above transport between probability measures:
\begin{align}
W_1 \leq  \alpha +\frac{1-\alpha}{3}.
\end{align}
From Definition \ref{def_alpha_ricci},
we have the following result about the Ollivier Ricci curvature.
\begin{align}
\kappa_{\alpha}(a, e) \geq \frac{2}{3} (1-\alpha).
\end{align}
By Definition \ref{def_ricci}, we have the following result of the lower bound of the Ricci curvature.
\begin{align}
\kappa(a, e) \geq  \frac{2}{3}.
\label{ineq.5}
\end{align}

Next, using a $1$-Lipschtiz function, we estimate an upper bound of 
the Ricci curvature by Theorem \ref{thm_w_1_lower}.
We define a $1$-Lipschtiz function as Figure $\ref{fig.3a}$.
The number in each box beside each vertex is the value of the $1$-Lipschiz function.
We have the following result from Theorem \ref{thm_w_1_lower} and this $1$- Lipschiz function.

\begin{align}
W_1 \geq  \alpha +\frac{1-\alpha}{3}.
\end{align}
From Definition \ref{def_alpha_ricci}, 
the Ollivier Ricci curvature value is bounded as
\begin{align}
\kappa_{\alpha} (a, e)\leq \frac{2}{3}(1-\alpha).
\end{align}
Therefore, we have the following  upper bound of the Ricci curvature
by Definition \ref{def_ricci},
\begin{align}
\kappa(a, e)  \leq \frac{2}{3}.
\label{ineq.6}
\end{align}

By (\ref{ineq.5}) and (\ref{ineq.6}), 
\begin{align}
\kappa(a, e) = \frac{2}{3}.
\end{align}
Thus, the curvature of all edges is
$\kappa = \frac{2}{3}$.

\rightline{\qed}
\end{proof}

\subsection{Proof of Proposition \ref{D5}}\label{prop5proof}
\begin{proof}
\begin{figure}[htbp]
\begin{minipage}[c]{0.5\hsize}
    \centering
    \includegraphics[width=8cm]{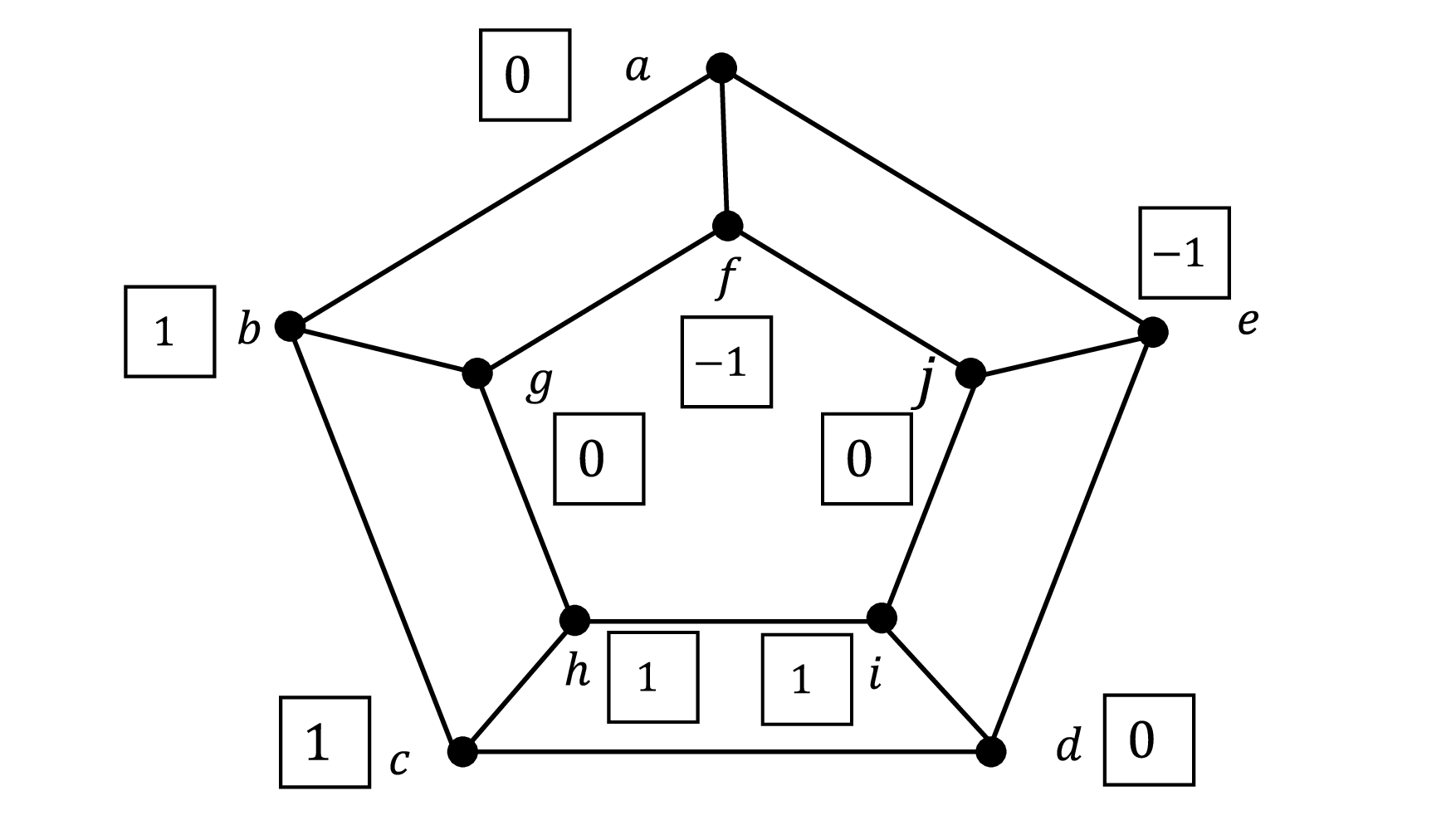}
    \caption{Type A in the Cayley graph of $D_{5}$}
    \label{fig.5a}
  \end{minipage}
  \begin{minipage}[c]{0.5\hsize}
    \centering
    \includegraphics[width=8cm]{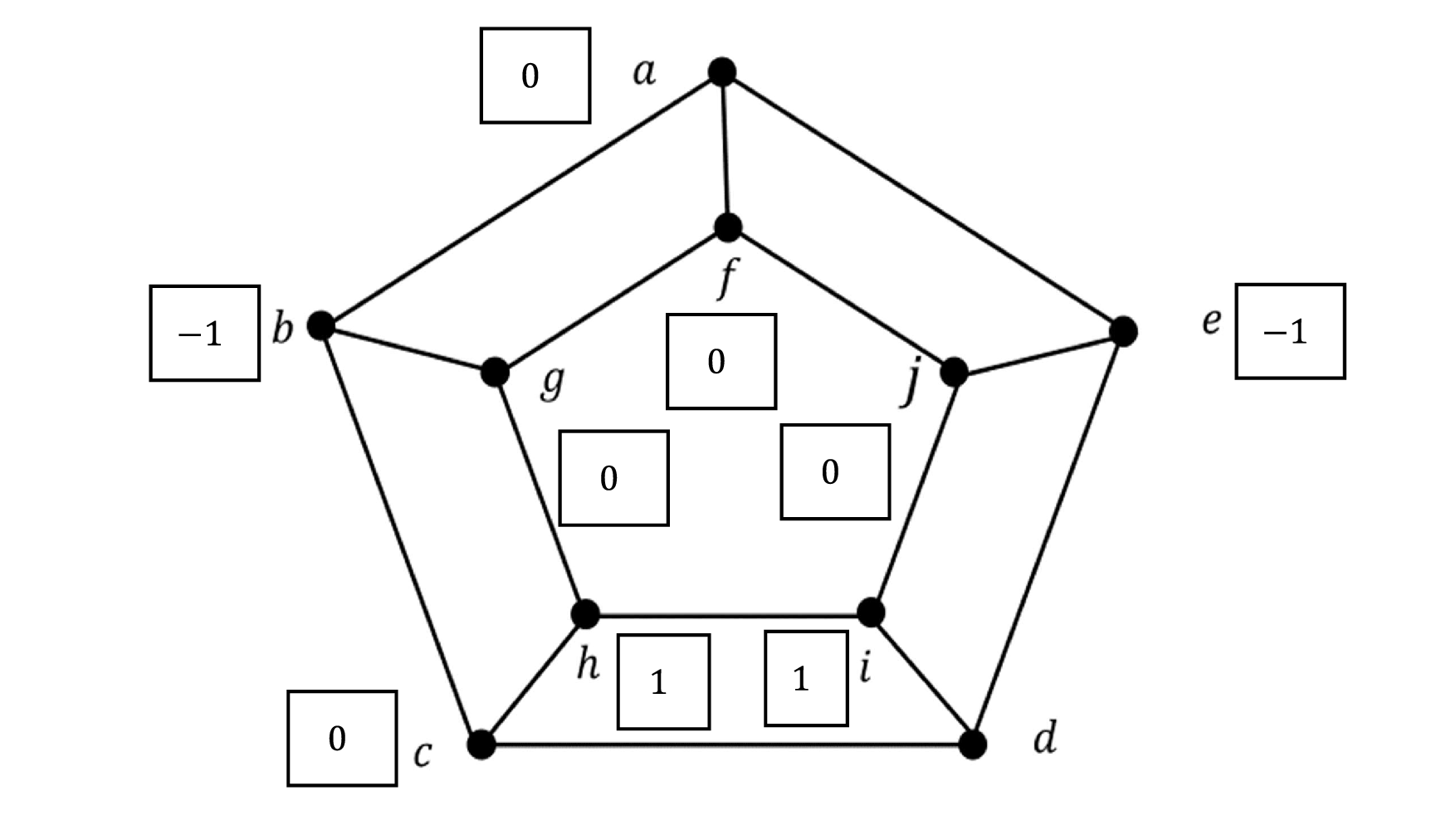}
    \caption{Type B in the Cayley graphs of $D_{5}$}
    \label{fig.5b}
  \end{minipage}
\end{figure}
\

At first, we calculate the Ricci curvature of the edge $ab$ as type $A$ in Figure \ref{fig.5a}.
By Definition \ref{def_prob_measure}, vertex $a$ has probability $\alpha$ for $\mu^{\alpha}_{a}$, 
on the other hand,  vertex $b$, $f$ and $e$ each has probability $\frac{1-\alpha}{3}$.
$\mu^{\alpha}_{b}$ is defined similarly. 
Satisfying Definition \ref{trans}, we provide the following  transport cost:
$a \rightarrow b  : \alpha-\frac{1-\alpha}{3}, \
f \rightarrow g   : \frac{1-\alpha}{3}, \
e \rightarrow c  :  \frac{1-\alpha}{3} \times 2$, \
and the amount of transport between the other vertices is zero.
From (\ref{W_1_inf}), we can estimate an upper bound of the Wasserstein distance 
for the above transport between probability measures.
We have the following result.
$$W_1 \leq  \alpha + \frac{2}{3}(1-\alpha)$$
From Definition \ref{def_alpha_ricci},
we have
$\kappa_{\alpha} \geq  \frac{1-\alpha}{3}$.
By Definition \ref{def_ricci}, we have the following result of the lower bound of the Ricci curvature.
\begin{align}
\kappa \geq  \frac{1}{3}.
\label{ineq.7}
\end{align}

Next, using a $1$-Lipschtiz function, we estimate an upper bound of 
the Ricci curvature by Theorem \ref{thm_w_1_lower}.
We define a $1$-Lipschtiz function as Figure $\ref{fig.5a}$.
The number in each box beside each vertex is the value of the $1$-Lipschiz function.
We have the following result from Theorem \ref{thm_w_1_lower} and this $1$- Lipschiz function.

\begin{align}
W_1 \geq  \alpha +  \frac{2}{3}(1-\alpha).
\end{align}
From Definition \ref{def_alpha_ricci},
we have $\kappa_a \leq  \frac{1}{3}(1-\alpha)$ for the Ollivier Ricci curvature value.
Therefore, we have the following  upper bound of the Ricci curvature
by Definition \ref{def_ricci},
\begin{align}
\kappa  \leq \frac{1}{3}.
\label{ineq.8}
\end{align}

By (\ref{ineq.7}) and (\ref{ineq.8}),  the curvature of all type A edges is
$\kappa = \frac{1}{3}$.
\bigskip

Next, we calculate the Ricci curvature of the edge $af$  as type $B$ in Figure \ref{fig.5b}.
We consider the transport cost.
By Definition \ref{def_prob_measure}, vertex $a$ has probability $\alpha$, and vertex $b, f$ and $e$ each has probability $\frac{1-\alpha}{3}$
 in $\mu^{\alpha}_{a}$.
$\mu^{\alpha}_{f}$ is defined similarly.
Satisfying Definition \ref{trans}, we can provide the following  transport cost:
$a \rightarrow f : \alpha-\frac{1-\alpha}{3}, \
b \rightarrow g  : \frac{1-\alpha}{3}, \
e\rightarrow j : \frac{1-\alpha}{3}$, \
and the amount of transport between the other vertices is zero.
From (\ref{W_1_inf}), we can estimate an upper bound of the Wasserstein distance 
for the above transport between probability measures.
\begin{align}
W_1 \leq  \alpha + \frac{(1-\alpha)}{3}.
\end{align}
From Definition \ref{def_alpha_ricci},
we have 
$\kappa_{\alpha} \geq  2 \times \frac{(1-\alpha)}{3}$.
By Definition \ref{def_ricci}, we have the following result of the lower bound of the Ricci curvature.
\begin{align}
\kappa \geq  \frac{2}{3}.
\label{ineq.9}
\end{align}

Next, using a $1$-Lipschtiz function, we estimate an upper bound of 
the Ricci curvature by Theorem \ref{thm_w_1_lower}.
We define a $1$-Lipschtiz function as Figure $\ref{fig.5b}$.
The number in each box beside each vertex is the value of the $1$-Lipschiz function.
We have the following result from Theorem \ref{thm_w_1_lower} and this $1$- Lipschiz function.

\begin{align}
W_1 \geq  \alpha + \frac{(1-\alpha)}{3}.
\end{align}
From Definition \ref{def_alpha_ricci},  
we have 
$\kappa_{\alpha} \leq 2 \times \frac{1-\alpha}{3}$.
Therefore, we have the following  upper bound of the Ricci curvature
by Definition \ref{def_ricci},
\begin{align}
\kappa  \leq \frac{2}{3}.
\label{ineq.10}
\end{align}

By (\ref{ineq.9}) and (\ref{ineq.10}), the curvature of all type B edges is
$\kappa = \frac{2}{3}$.\rightline{\qed}
\rightline{\qed}
\end{proof}

\subsection{Proof of Proposition \ref{D6}}\label{prop6proof}
\begin{proof}
First, we prove that the Ricci curvature of type A   is $\kappa = 0$.
\begin{figure}[htbp]
\centering
\includegraphics[width=10cm]{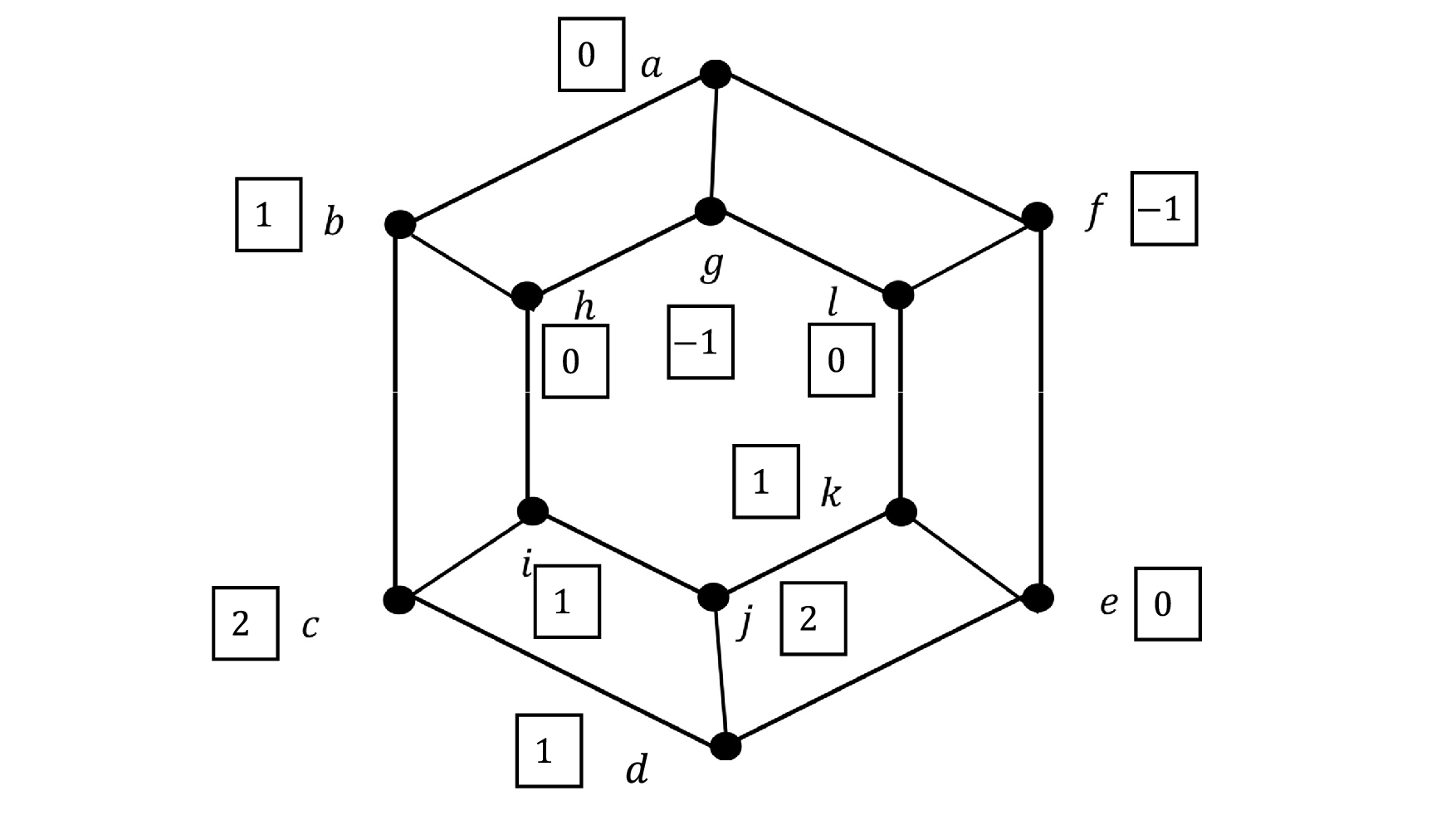} 
\caption{Cayley graph of $D_{6}$ with $S$}
\label{fig.4}
\end{figure}

Now, we calculate. Ricci curvature of the edge $ab$ in Figure 5

We are considering a transport plan.
By Dare cons ering $4$, vertex $a$ has probability $\alpha$. 
on the other hand,  vertex $b, f$ and $g$ each has probability $\frac{1-\alpha}{3}$.
Satisfying Definition $2$, we can provide the following  transport plan:\\
$f \rightarrow a ; \frac{1-\alpha}{3}, \
a \rightarrow b  ; \alpha, \
b \rightarrow c ;  \frac{1-\alpha}{3}, 
g \rightarrow h ; \frac{1-\alpha}{3}$, 
but the amount of transport between the other vertices is zero.
The following results can be calculated from these transport plans.
\begin{align}
W_1 \leq  \alpha + (1-\alpha).
\end{align}
Next, by Definition $5$, we get the following result of the  Ollivier  Ricci curvature value.
\begin{align}
\kappa_{\alpha} \geq  0.
\end{align}
By Definition $7$, we get the following lower bound of the Ricci curvature.
\begin{align}
\kappa \geq  0.
\label{ineq.11}
\end{align}

We define a $1-$Lipschtiz function as figure $5$.
The function in each beside each vertex is the value of the $1$-Lipschiz function.
We have the following result from Theorem $1$ and this $1-$ Lipschiz function.
\begin{align}
W_1 \geq  \alpha + (1-\alpha).
\end{align}
We use Definition$5$ and $7$, and we get the following the Ollivier Ricci curvature value.
\begin{align}
\kappa_a \leq 0.
\end{align}
Therefore, we get the following upper bound of the Ricci curvature.
\begin{align}
\kappa  \leq 0.
\label{ineq.12}
\end{align}

By ($\ref{ineq.11})$ and ($\ref{ineq.12}$), the curvature of all edges is as follows.
\begin{align}
\kappa = 0.
\end{align}

Next, we prove that the Ricci curvature of type B   is $\kappa =  \frac{2}{3}$.

\begin{figure}[htbp]
\centering
\includegraphics[width=10cm]{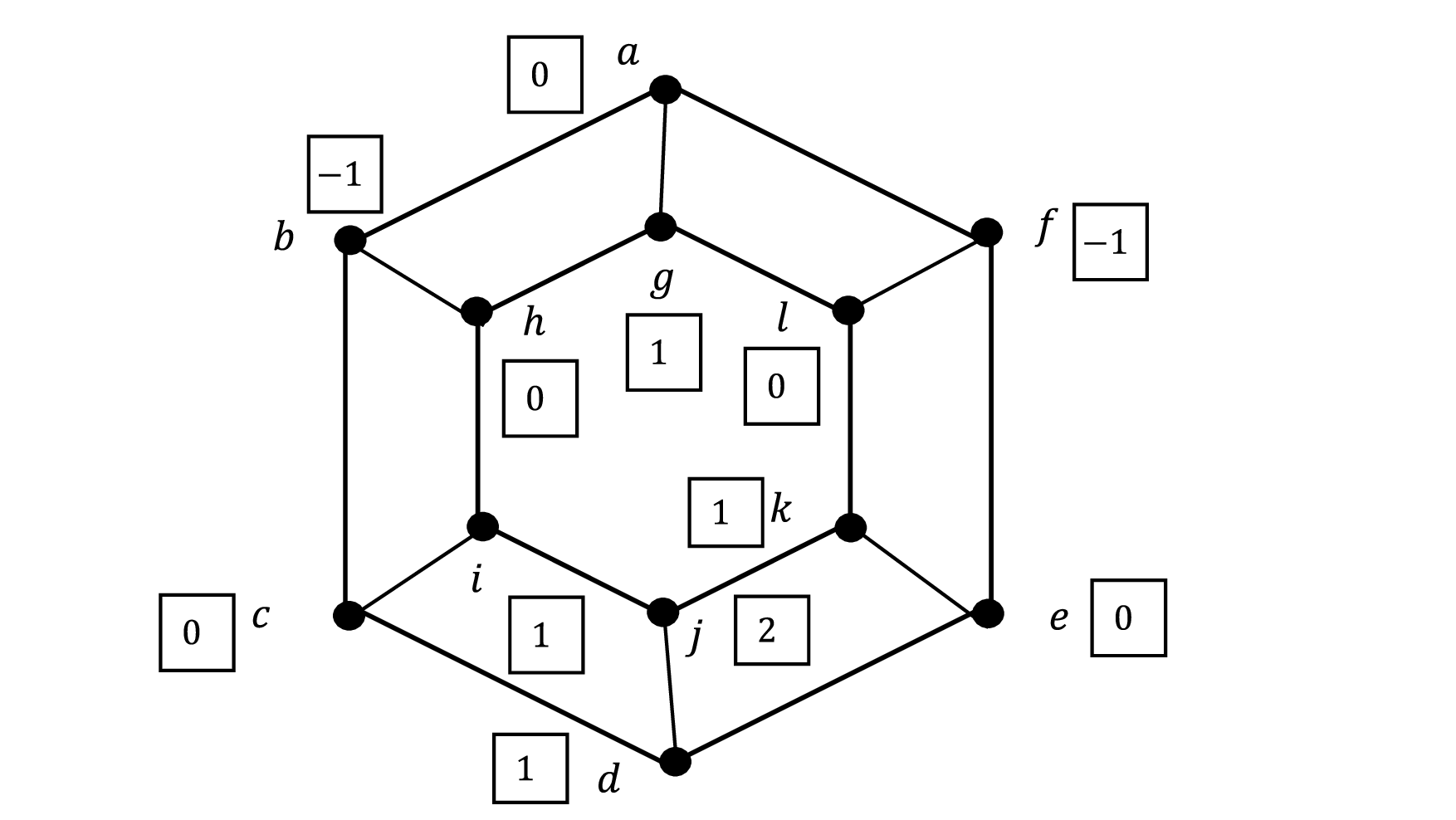} 
\caption{Cayley graphs of the dihedral group of regular Icosahedrons}
\label{fig6}
\end{figure}

Now, we calculate. Ricci curvature of the edge $ab$ in Figure 6
We are considering a transport plan.
By Definition $4$, vertex $a$ has probability $\alpha$. 
on the other hand,  vertex $b, f$ and $g$ each has probability $\frac{1-\alpha}{3}$.
Satisfying Definition $2$, we can provide the following  transport plan:\\
$a \rightarrow f ; \alpha- \frac{1-\alpha}{3}, \
b \rightarrow h ;  \frac{1-\alpha}{3}, \
f \rightarrow l ; \frac{1-\alpha}{3}$, 
but the amount of transport between the other vertices is zero.
The following results can be calculated from these optimal transport plans.
\begin{align}
W_1 \leq  \alpha + \frac{1-\alpha}{3}.
\end{align}
Next, by Definition $5$, we get the following result of the  Ollivier  Ricci curvature value.
\begin{align}
\kappa_{\alpha} \geq  2 \times \frac{1-\alpha}{3}.
\end{align}
By Definition $7$, we get the following lower bound of the Ricci curvature.
\begin{align}
\kappa \geq  \frac{2}{3}.
\end{align}

We define a $1$-Lipschtiz function as figure $6$.
The number in each beside each vertex is the value of the $1$-Lipschiz function.
We have the following result from Theorem $1$ and this $1-$ Lipschiz function.
\begin{align}
W _1 \geq  \alpha + \frac{1}{3}(1-\alpha).
\end{align}
We use Definition$5$ and $7$, and we get the following the Ollivier Ricci curvature value.
\begin{align}
\kappa_{\alpha} \leq \frac{2}{3}(1-\alpha).
\end{align}
Therefore, we get the following upper bound of the Ricci curvature.
\begin{align}
\kappa  \leq \frac{2}{3}.
\end{align}
By (49) and (52), the curvature of all edges is as follows:
\begin{align}
\kappa = \frac{2}{3}.
\end{align}
\end{proof}

\subsection{Proof of Proposition \ref{Q4}}\label{propQ4proof}.
\begin{proof}
\begin{figure}[htbp]
\begin{minipage}[c]{0.5\hsize}
    \centering
    \includegraphics[width=8cm]{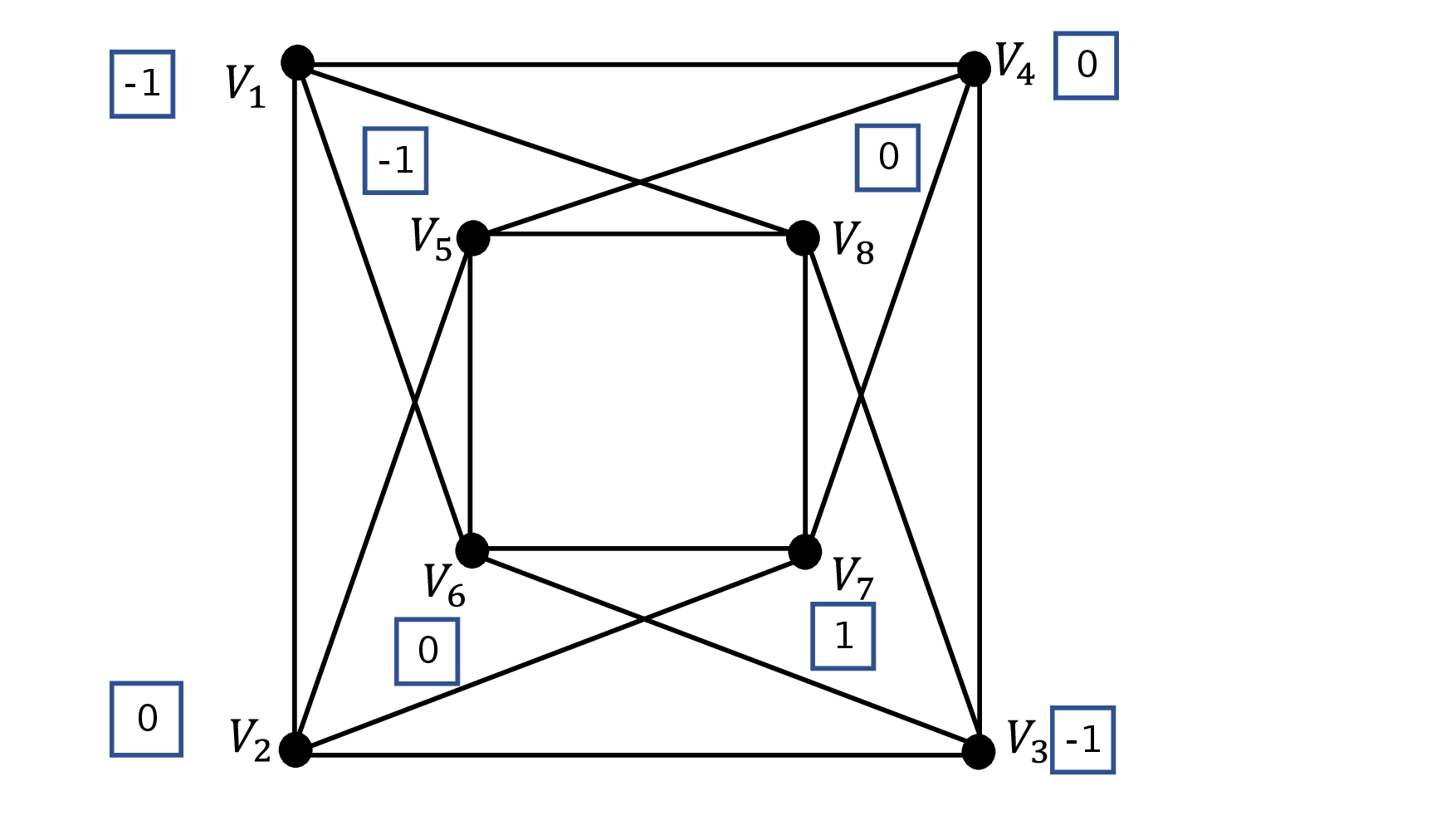}
    \caption{Type A in the Cayley graph of $Q_{8}$}
    \label{fig.q4a}
  \end{minipage}
  \begin{minipage}[c]{0.5\hsize}
    \centering
    \includegraphics[width=8cm]{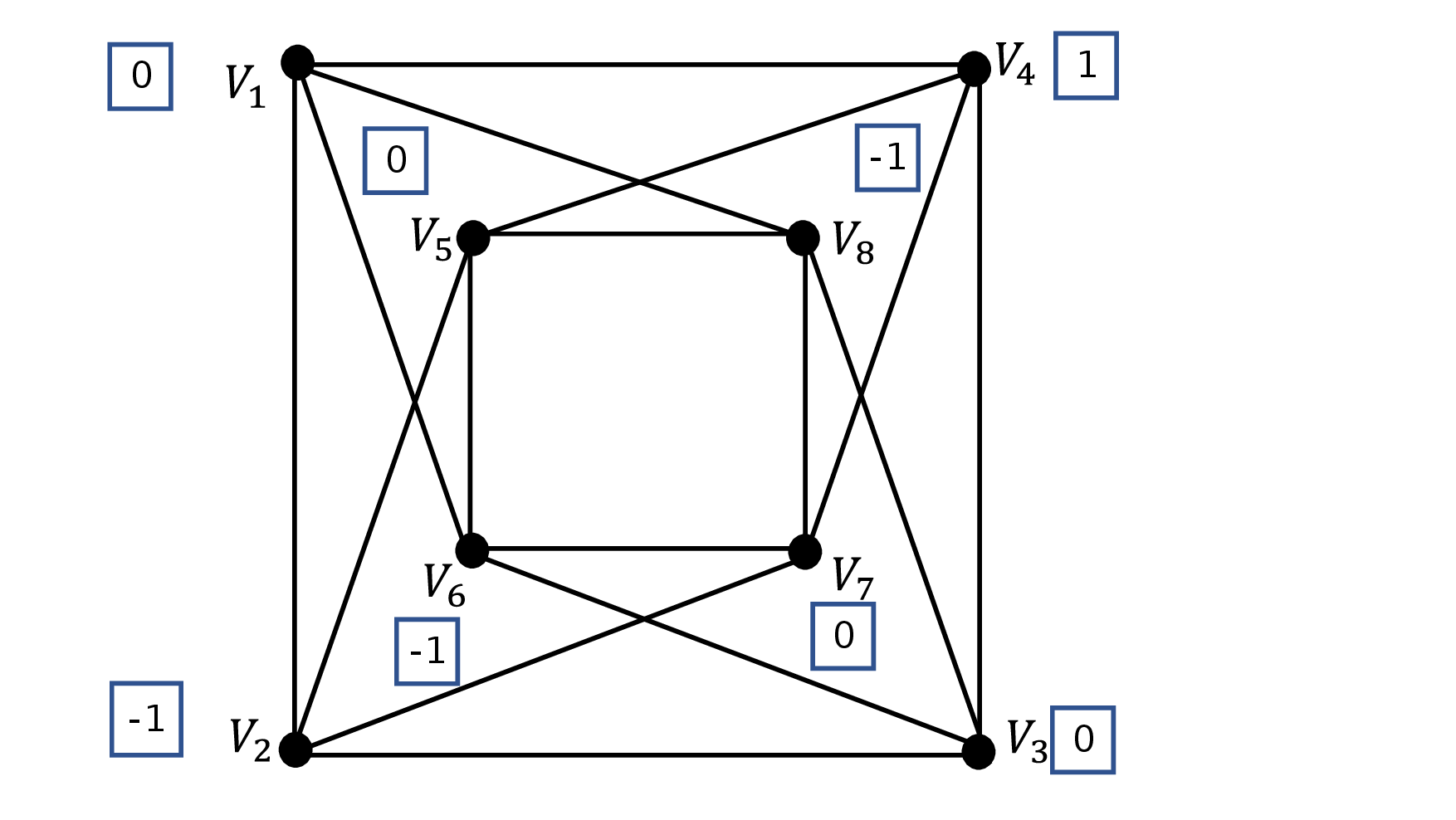}
    \caption{Type B in the Cayley graphs of $Q_{8}$}
    \label{fig.q4b}
  \end{minipage}
\end{figure}

First, we prove that the Ricci curvature of type A is $\kappa = \frac{1}{2}$.
We calculate the Ricci curvature of the edge $V_{6}V_{7}$ in Figure \ref{fig.q4a} as a type A.
By Definition \ref{def_prob_measure}, vertex $V_{6}$ has probability $\alpha$ for $\mu^{\alpha}_{V_{6}}$, 
on the other hand,  vertex $V_{1}, V_{5}, V_{7}$ and $V_{3}$ 
each has probability $\frac{1-\alpha}{4}$. $\mu^{\alpha}_{V_{7}}$ is determined similarly.
Satisfying Definition \ref{trans}, we can provide the following  transport cost:
$V_6 \rightarrow V_7 :  \alpha - \frac{1-\alpha}{4}, \
V_1 \rightarrow V_4 : \frac{1-\alpha}{4}, \
 V_5\rightarrow V_8 :  \frac{1-\alpha}{4}, \
V_3  \rightarrow V_2 : \frac{1-\alpha}{4}$, 
and the amount of transport between the other vertices is zero.
From (\ref{W_1_inf}), we can estimate an upper bound of the Wasserstein distance 
for the above transport between probability measures.
We have the following result.
\begin{align}
W_1 \leq  \alpha +\frac{1}{2} (1-\alpha).
\end{align}
From Definition \ref{def_alpha_ricci},
we have
$\kappa_{\alpha} \geq  \frac{1-\alpha}{2}$.
By Definition \ref{def_ricci}, 
we have the following result of the lower bound of the Ricci curvature.
\begin{align}
\kappa \geq  \frac{1}{2}.
\label{ineq.17}
\end{align}

Next, using a $1$-Lipschtiz function, we estimate an upper bound of 
the Ricci curvature by Theorem \ref{thm_w_1_lower}.
We define a $1$-Lipschtiz function as Figure $\ref{fig.q4a}$.
The number in each box beside each vertex is the value of the $1$-Lipschiz function.
We have the following result from Theorem \ref{thm_w_1_lower} and this $1$- Lipschiz function.

\begin{align}
W _1 \geq  \alpha + \frac{1-\alpha}{2}.
\end{align}
From Definition \ref{def_alpha_ricci}, 
we have 
$\kappa_{\alpha} \leq \frac{1-\alpha}{2}$.
Therefore, we have the following  upper bound of the Ricci curvature
by Definition \ref{def_ricci},
\begin{align}
\kappa  \leq \frac{1}{2}.
\label{ineq.18}
\end{align}

By (\ref{ineq.17}) and (\ref{ineq.18}), the curvature of all type A edges is
$\kappa = \frac{1}{2}$.
\bigskip

We prove that the Ricci curvature of type B   is $\kappa =  \frac{1}{2}$.
We calculate the Ricci curvature of the edge $V_{4}V_{7}$ in Figure \ref{fig.q4b} 
as a type B edge.
By Definition \ref{def_prob_measure}, vertex $V_7$ has probability $\alpha$, and
vertex $V_2, V_4, V_6$ and $V_8$ each has probability $\frac{1-\alpha}{4}$ in $\mu^{\alpha}_{V_{7}}$.
$\mu^{\alpha}_{V_{4}}$ is defined similarly.
Satisfying Definition \ref{trans}, we can provide the following  transport cost:
$V_7 \rightarrow V_4 : \alpha- \frac{1-\alpha}{4}, 
 V_2 \rightarrow V_1 :  \frac{1-\alpha}{4}, \
V_6 \rightarrow V_5 : \frac{1-\alpha}{4}, \
V_8 \rightarrow V_7 : \frac{1-\alpha}{4}$, 
and the amount of transport between the other vertices is zero.
From (\ref{W_1_inf}), we can estimate an upper bound of the Wasserstein distance 
for the above transport between probability measures.
We have the following result.
\begin{align}
W_1 \leq  \alpha + \frac{1-\alpha}{2}.
\end{align}
From Definition \ref{def_alpha_ricci},
we have 
$\kappa_{\alpha} \geq  2 \times \frac{1-\alpha}{2}$.
By Definition \ref{def_ricci}, we have the following result of the lower bound of the Ricci curvature.
\begin{align}
\kappa \geq  \frac{1}{2}
\label{ineq.19}
\end{align}

Next, using a $1$-Lipschtiz function, we estimate an upper bound of 
the Ricci curvature by Theorem \ref{thm_w_1_lower}.
We define a $1$-Lipschtiz function as Figure $\ref{fig.q4b}$.
The number in each box beside each vertex is the value of the $1$-Lipschiz function.
We have the following result from Theorem \ref{thm_w_1_lower} and this $1$- Lipschiz function. 

\begin{align}
W_1 \geq  \alpha + \frac{1}{2}(1-\alpha).
\end{align}
From Definition \ref{def_alpha_ricci}, 
we have 
$\kappa_{\alpha} \leq \frac{1}{2}(1-\alpha)$.
Therefore, we have the following  upper bound of the Ricci curvature
by Definition \ref{def_ricci},
\begin{align}
\kappa  \leq \frac{1}{2}.
\label{ineq.20}
\end{align}
By (\ref{ineq.19}) and (\ref{ineq.20}), the curvature of all type B edges is
$\kappa = \frac{1}{2}$.\rightline{\qed}
\rightline{\qed}
\end{proof}

\subsection{Proof of Proposition \ref{Q6}}\label{propQ6proof}.
\begin{proof}

\begin{figure}[htbp]
\begin{minipage}[c]{0.5\hsize}
    \centering
    \includegraphics[width=8cm]{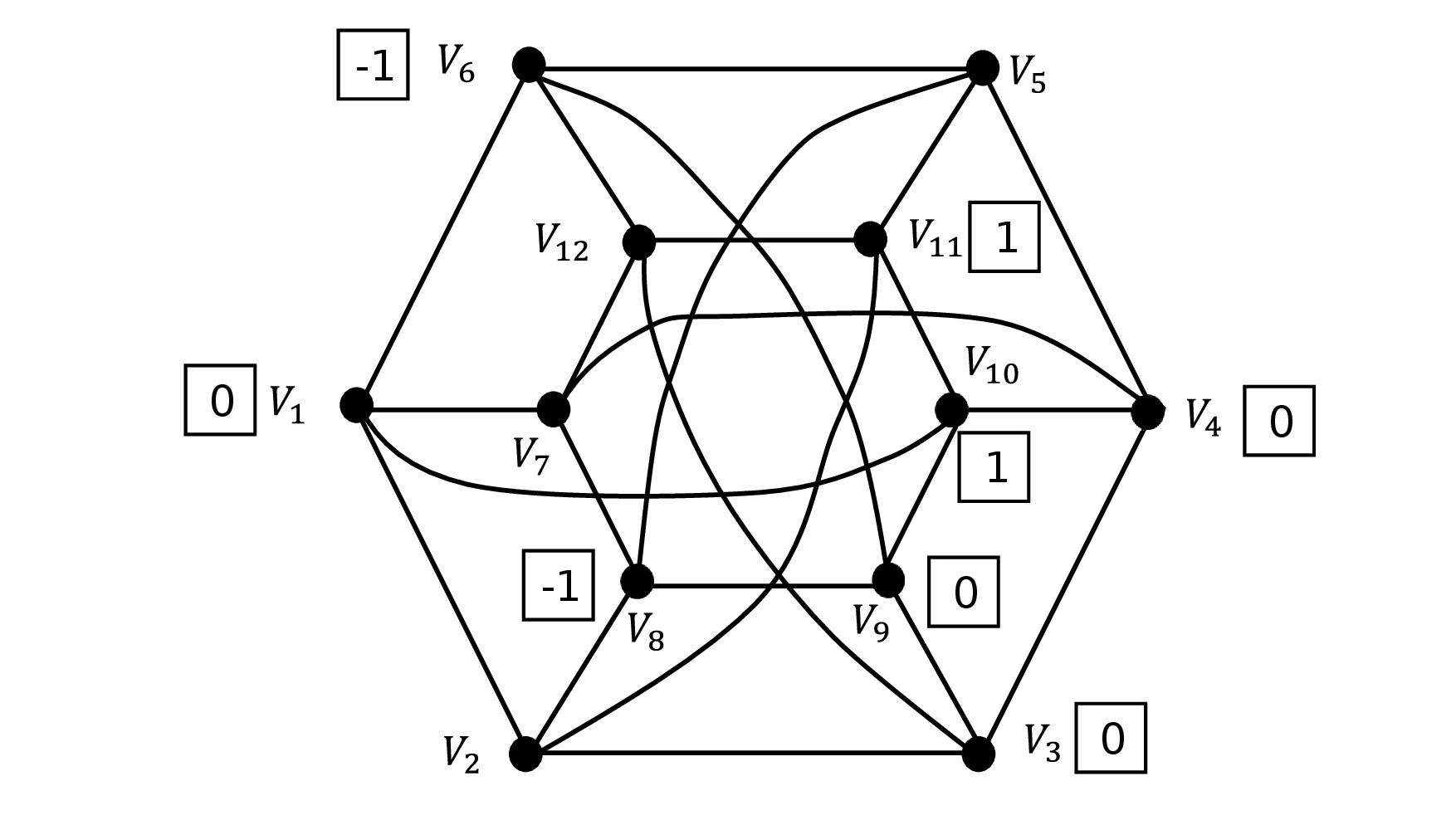}
    \caption{Type A in the Cayley graph of $Q_{12}$}
    \label{fig.q6a}
  \end{minipage}
  \begin{minipage}[c]{0.5\hsize}
    \centering
    \includegraphics[width=8cm]{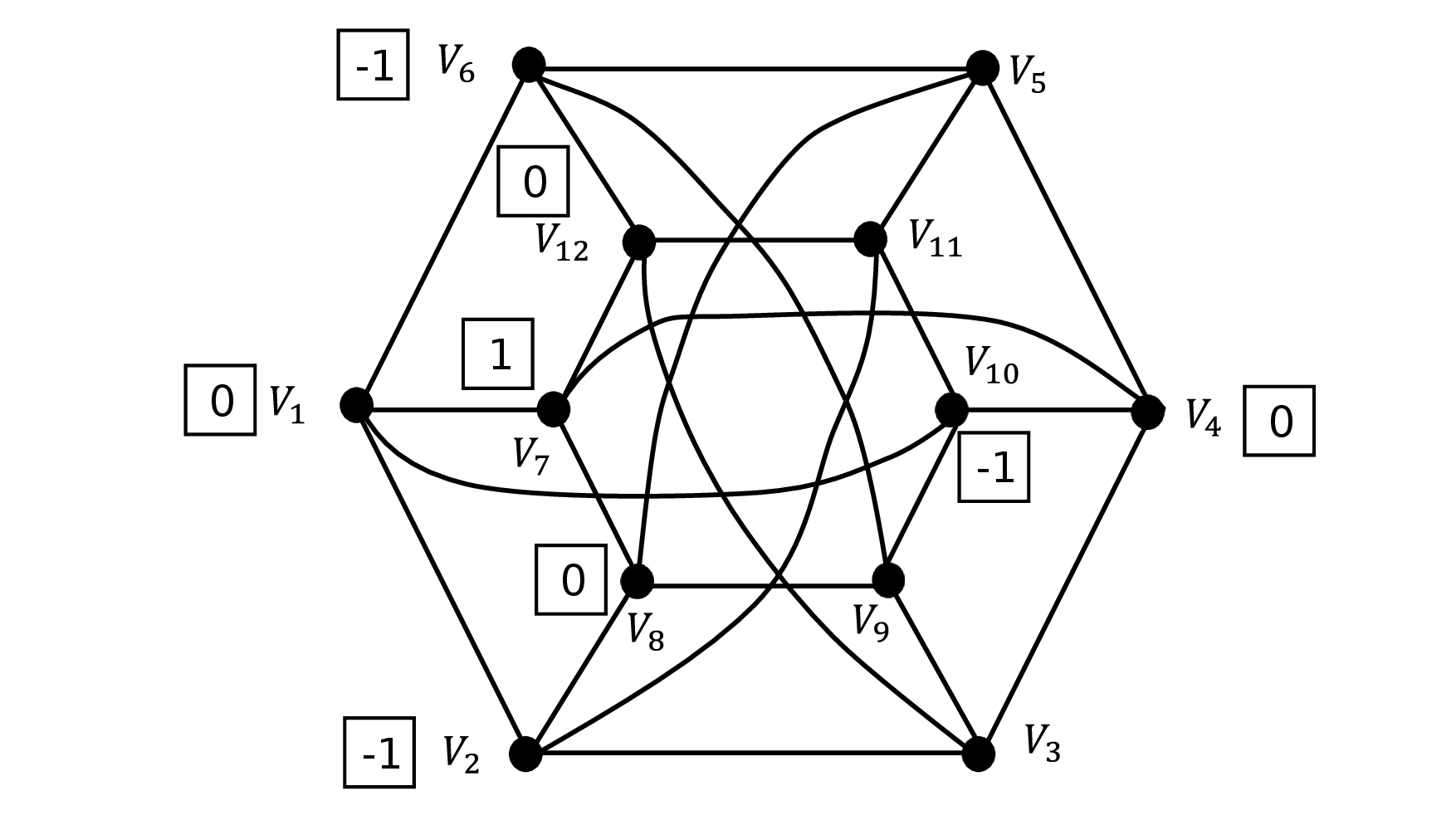}
    \caption{Type B in the Cayley graphs of $Q_{12}$}
    \label{fig.q6b}
  \end{minipage}
\end{figure}

First, we prove that the Ricci curvature of type A  is $\kappa = \frac{1}{2}$.
We calculate the Ricci curvature of type A as $V_{9}V_{10}$ in Figure \ref{fig.q6a}.
Satisfying Definition \ref{trans}, we can provide the following  transport cost
 from $\mu^{\alpha}_{V_{9}}$ to $\mu^{\alpha}_{V_{10}}$:
$V_9 \rightarrow V_{10} : \alpha-\frac{1-\alpha}{4}, \
V_ 6\rightarrow V_1 : \frac{1-\alpha}{4}, \
V_3 \rightarrow V_4 : \frac{1-\alpha}{4}, \
V_8 \rightarrow V_{11} : \frac{1-\alpha}{4} \times 2$, 
and the amount of transport between the other vertices is zero.
From (\ref{W_1_inf}), we can estimate an upper bound of the Wasserstein distance 
for the above transport between probability measures.
We have the following result.
\begin{align}
W_1 \leq  \alpha +\frac{3}{4} (1-\alpha).
\end{align}
From Definition \ref{def_alpha_ricci},
we have 
$\kappa_{\alpha} \geq  \frac{1-\alpha}{4}$.
By Definition \ref{def_ricci}, we have the following result of the lower bound of the Ricci curvature.
\begin{align}
\kappa \geq  \frac{1}{4}.
\label{ineq.21}
\end{align}

Next, using a $1$-Lipschtiz function, we estimate an upper bound of the Ricci curvature 
by Theorem \ref{thm_w_1_lower}.
We define a $1$-Lipschtiz function as Figure \ref{fig.q6a}.
The number in each box beside each vertex is the value of the $1$-Lipschiz function.
We have the following result from Theorem \ref{thm_w_1_lower} and this $1$- Lipschiz function.

\begin{align}
W_1 \geq  \alpha + \frac{3}{4}(1-\alpha).
\end{align}
From Definition \ref{def_alpha_ricci}, 
we have 
$\kappa_{\alpha} \leq \frac{1}{4}(1-\alpha)$.
Therefore, we have the following upper bound of the Ricci curvature
by Definition \ref{def_ricci},
\begin{align}
\kappa  \leq \frac{1}{4}.
\label{ineq.22}
\end{align}

By (\ref{ineq.21}) and (\ref{ineq.22}), the curvature of all type A edges is
$\kappa = \frac{1}{4}$.
This is the same result for all other type A edges.
\bigskip

Next, we prove that the Ricci curvature of type B is $\kappa =  \frac{1}{2}$.
Now, we calculate the Ricci curvature of $V_1V_7$ in Figure \ref{fig.q6b} as a type B edge.
Satisfying Definition \ref{trans}, we can provide the following  transport cost 
from $\mu^{\alpha}_{V_{1}}$ to $\mu^{\alpha}_{V_{7}}$:
$V_1\rightarrow V_7 :\alpha-\frac{1-\alpha}{4}, \  
V_ 2\rightarrow V_{12} : \frac{1-\alpha}{4}, \
V_6 \rightarrow V_8 : \frac{1-\alpha}{4}, \
V_{10} \rightarrow V_4: \frac{1-\alpha}{4}$, 
and the amount of transport between the other vertices is zero.
From (\ref{W_1_inf}), we can estimate an upper bound of the Wasserstein distance 
for the above transport between probability measures.
We have the following result.
\begin{align}
W_1 \leq  \alpha + \frac{1-\alpha}{2}.
\end{align}
From Definition \ref{def_alpha_ricci},
we have 
$\kappa_{\alpha} \geq  \frac{1-\alpha}{2}$.
By Definition \ref{def_ricci}, we have the following result of the lower bound of the Ricci curvature.
\begin{align}
\kappa \geq  \frac{1}{2}.
\label{ineq.23}
\end{align}

Next, using a $1$-Lipschtiz function, we estimate an upper bound of the Ricci curvature by Theorem \ref{thm_w_1_lower}.
We define a $1$-Lipschtiz function as Figure \ref{fig.q6b}.
The number in each box beside each vertex is the value of the $1$-Lipschiz function.
We have the following result from Theorem \ref{thm_w_1_lower} and this $1$- Lipschiz function.

\begin{align}
W_1 \geq  \alpha + \frac{1}{2}(1-\alpha).
\end{align}
From Definition \ref{def_alpha_ricci},  
we have 
$\kappa_{\alpha} \leq \frac{1}{2}(1-\alpha)$.
Therefore, we have the following upper bound of the Ricci curvature
by Definition \ref{def_ricci},
\begin{align}
\kappa  \leq \frac{1}{2}.
\label{ineq.24}
\end{align}
By (\ref{ineq.23}) and (\ref{ineq.24}), the curvature of all type B edges is
$\kappa = \frac{1}{2}$.
This is the same result for all other type B edges.
\hspace{\fill}{\qed}

\end{proof}

\subsection{Proof of Proposition \ref{s12}}\label{props12proof}
\begin{proof}
\begin{figure}[H]
  \centering
    \includegraphics[width=8cm]{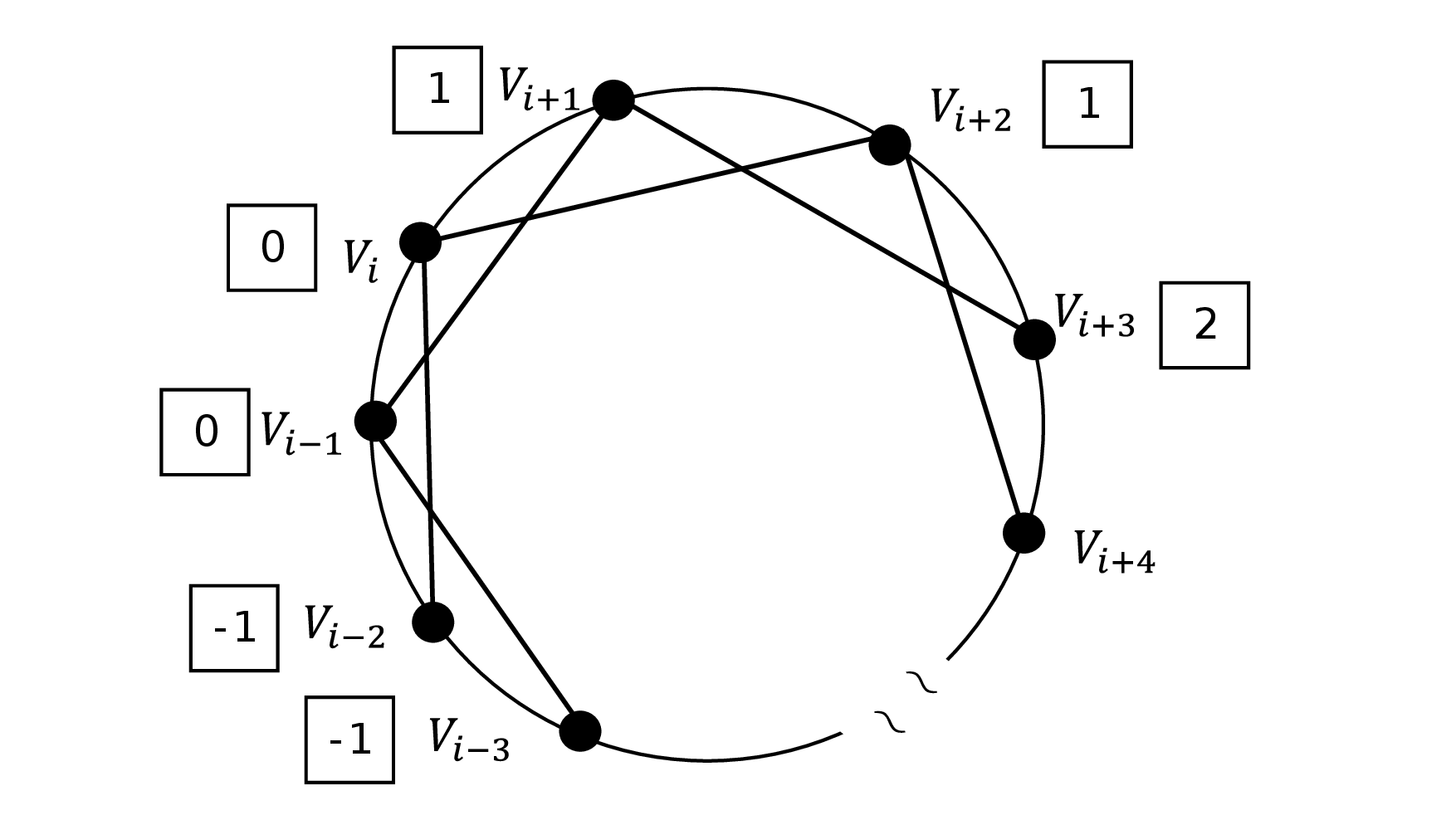}
    \caption{Cayley graphs of $\mathbb{Z}/n\mathbb{Z}$}
    \label{fig.12b}
 
\end{figure}
We prove that the curvature of the edge type A as $V_{i}V_{i+1}$ in Figure \ref{fig.12b}
is $\frac{1}{2}$.
Satisfying Definition \ref{trans}, we can provide the following  transport cost:
$ V_{i-2} \rightarrow V_{i} : \frac{1-\alpha}{4}, \
V_{i} \rightarrow V_{i+1} :\alpha,  \
V_{i+1} \rightarrow V_{i+3} : \frac{1-\alpha}{4}$,
and the amount of transport between the other vertices is zero.
From (\ref{W_1_inf}), we can estimate an upper bound of the Wasserstein distance 
for the above transport between probability measures.
We have the following result.
\begin{align}
W_1 \leq  \alpha +\frac{1}{2} (1-\alpha).
\end{align}
From Definition \ref{def_alpha_ricci},
we have 
$\kappa_{\alpha} \geq  \frac{1}{2} (1-\alpha)$.
By Definition \ref{def_ricci}, we have the following result of the lower bound of the Ricci curvature.
\begin{align}
\kappa \geq  \frac{1}{2}.
\label{ineq.33}
\end{align}

We define a $1$-Lipschtiz function as Figure \ref{fig.12b}.
The number in each box beside each vertex is the value of the $1$-Lipschiz function.
To define this $1$-Lipschtiz function, we use the condition $n \geq 11$.
We have the following result from Theorem $1$ and this $1$- Lipschiz function.

\begin{align}
W_1\geq   \alpha +\frac{1}{2}(1-\alpha).
\end{align}
From Definition \ref{def_alpha_ricci},  
we have
$\kappa_{\alpha} \leq \frac{1}{2} (1-\alpha)$.
Therefore, we have the following upper bound of the Ricci curvature
by Definition \ref{def_ricci},
\begin{align}
\kappa  \leq \frac{1}{2}.
\label{ineq.34}
\end{align}

By (\ref{ineq.33}) and (\ref{ineq.34}), the curvature of all type A edges is
$\kappa = \frac{1}{2}$.
\bigskip

Next, we prove that the Ricci curvature of the edge $V_{i-1}V_{i+1}$ as a type B  is $\kappa =  0$.
Satisfying Definition \ref{trans}, we can provide the following  transport cost:
$ V_{i-1} \rightarrow  V_{i+1} :\alpha- \frac{1-\alpha}{4}, V_{i-2} \rightarrow V_{i+2} : \frac{1-\alpha}{4} \times 2, V_{i-3} \rightarrow V_{i+3} : \frac{1-\alpha}{4} \times 3$, 
and the amount of transport between the other vertices is zero.
Here, we use the fact that $d(V_{i-3}, V_{i+3}) = 3$.
This fact is derived from $n \geq 11$ as follows.
Let us define the path in Figure \ref{fig.12b}
$P_{s12b} := V_{i+3}V_{i+4} \cdots V_{i-4}V_{i-3}$.
Here, $|P_{s12b}| =n-(i-3)-(i+3)=n-6$.
If $n \geq 11$, $ |P_{s12b}| \geq 2+2+1$,
this means that there is no shortcut through the vertices in this path.
Then we obtain
$d(V_{i-3}, V_{i+3})= 3$.

%$d(V_{i-3}, V_{i+3})  \geq 3  \Leftrightarrow
%n-(i+3)-(i-3) \geq 2+2+1 \Leftrightarrow n \geq %11$

%For this to exist, $n \geq 11$ is required.
%This is because if there exists a shortcut from %$V_{i+3}$ to $V_{i-3}$ 
%in a clockwise direction, then $n-6 \geq 4 \Leftrightarrow n \geq 10$ holds. 
%Therefore, for $d(V_{i-3}, V_{i+3}) = 3 $ to exist, 
%it is necessary that $n \geq 10$ does not hold.

From (\ref{W_1_inf}), we can estimate an upper bound of the Wasserstein distance 
for the above transport between probability measures.
We have the following result.
\begin{align}
W_1 \leq  \alpha + (1-\alpha).
\end{align}
From Definition \ref{def_alpha_ricci},
we have 
$\kappa_{\alpha} \geq  0$.
By Definition \ref{def_ricci}, we have the following result of a lower bound of the Ricci curvature.
\begin{align}
\kappa \geq  0.
\label{ineq.35}
\end{align}

Next, using a $1$-Lipschtiz function, we estimate an upper bound of 
the Ricci curvature by Theorem \ref{thm_w_1_lower}.
We define a $1$-Lipschtiz function as Figure $\ref{fig.12b}$.
Here, we use $d(V_{i-3}, V_{i+3}) = 3$, again.
The number in each box beside each vertex is the value of the $1$-Lipschiz function.
We have the following result from Theorem \ref{thm_w_1_lower} and this $1$- Lipschiz function.

\begin{align}
W_1 \geq  \alpha + (1-\alpha).
\end{align}
From Definition \ref{def_alpha_ricci},  
we have
$\kappa_{\alpha} \leq 0$.
Therefore, we have the following  upper bound of the Ricci curvature
by Definition \ref{def_ricci},
\begin{align}
\kappa  \leq 0.
\label{ineq.36}
\end{align}

By (\ref{ineq.35}) and (\ref{ineq.36}), the curvature of all type B edges is
$\kappa = 0$.

\rightline{\qed}
This is the same result for all other edges.
\end{proof}

\subsection{Proof of Proposition \ref{s13}}\label{props13proof}

\begin{proof}
\begin{figure}[H]
\begin{minipage}[c]{0.5\hsize}
    \centering
    \includegraphics[width=8cm]{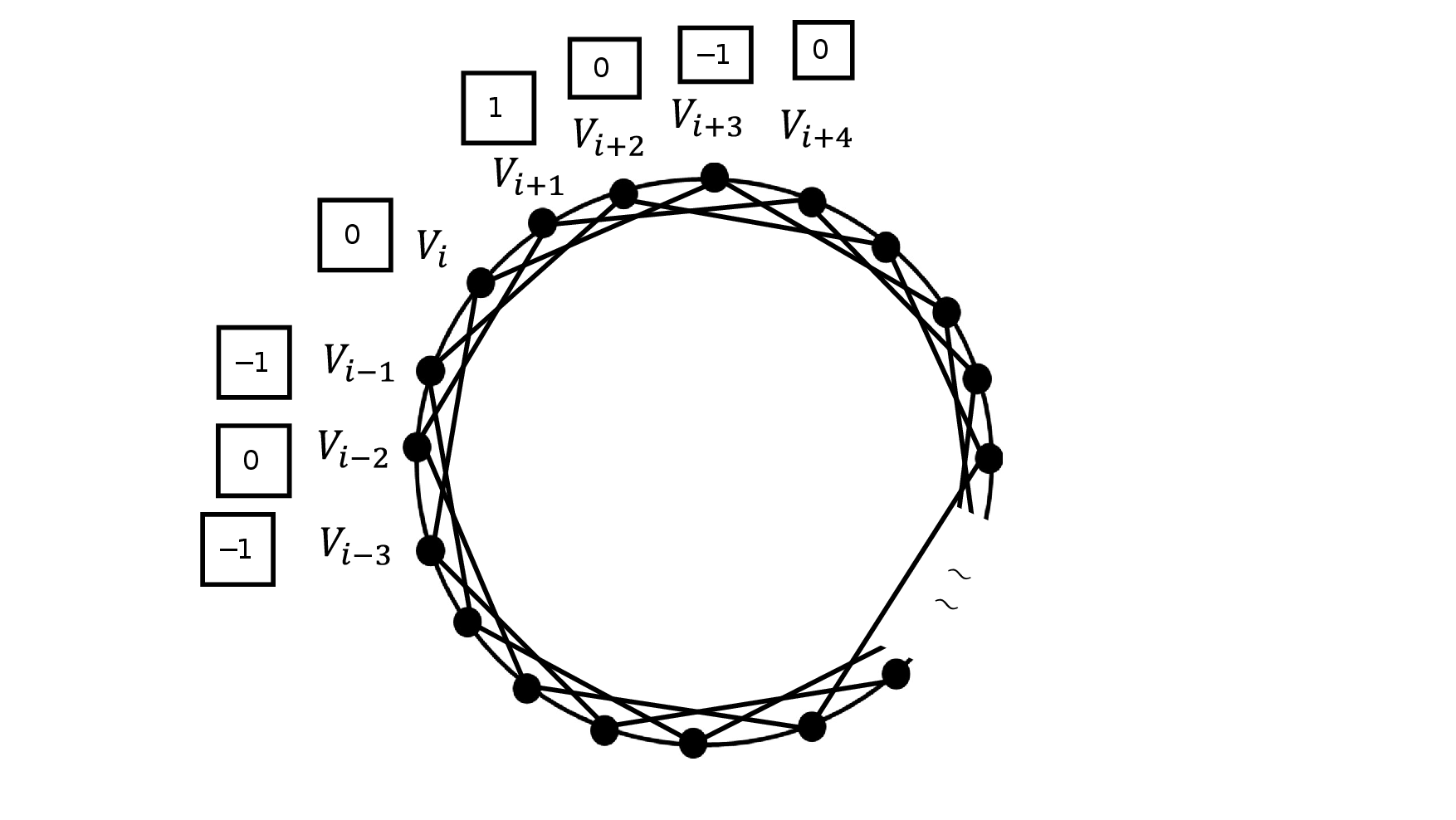}
    \caption{Type A in Cayley graphs of $\mathbb{Z}/n\mathbb{Z}$}
    \label{fig.13a}
  \end{minipage}
  \begin{minipage}[c]{0.5\hsize}
    \centering
    \includegraphics[width=8cm]{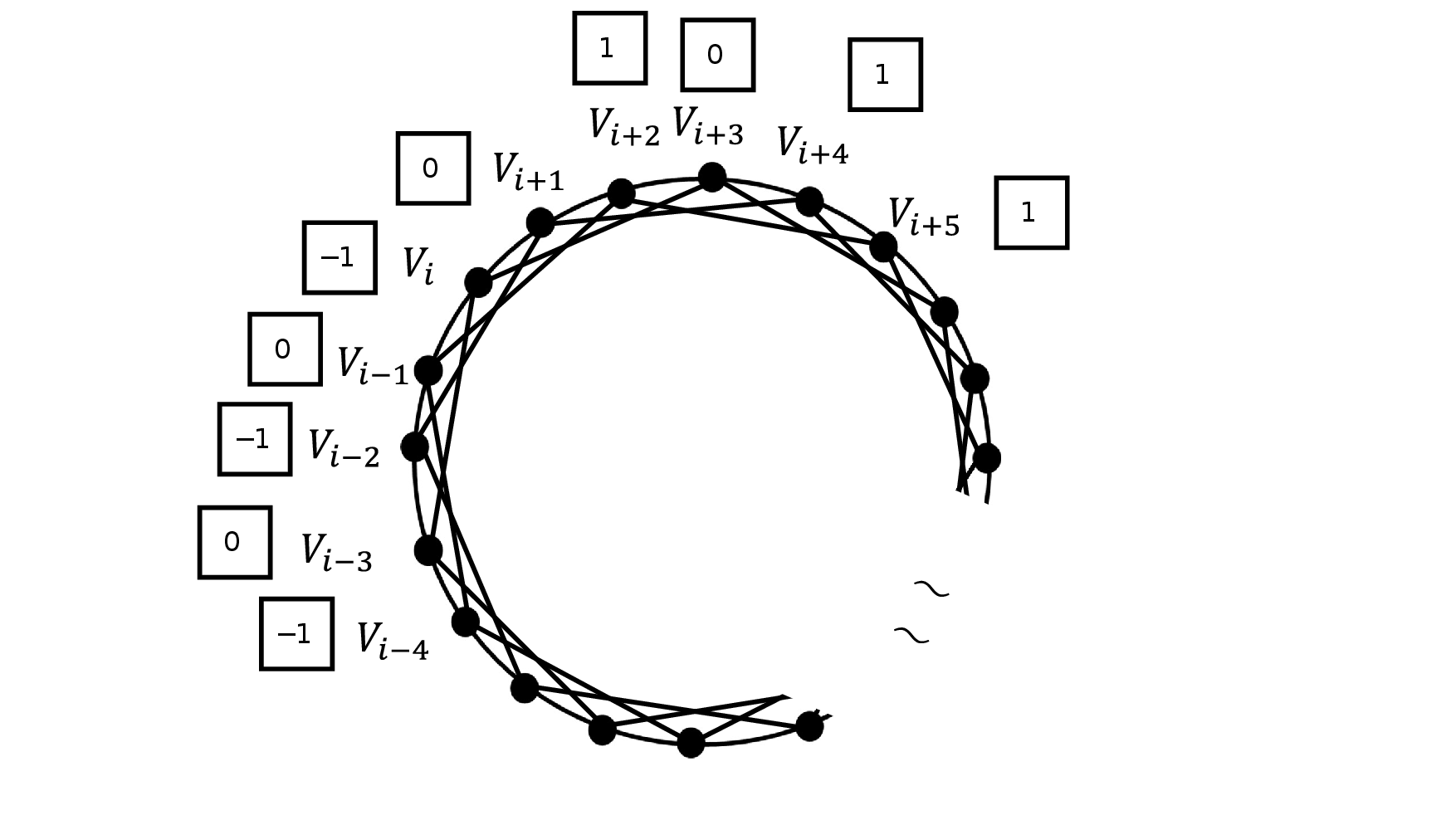}
    \caption{Type B in Cayley graphs of $\mathbb{Z}/n\mathbb{Z}$}
    \label{fig.13b}
  \end{minipage}
\end{figure}
We prove that the Ricci curvature of the edge $V_{i}V_{i+1}$ 
as a type A  edge in Figure \ref{fig.13a}.
We calculate the transport cost defined by
$ V_{i} \rightarrow V_{i+1} : \alpha- \frac{1-\alpha}{4},\ 
V_{i-1} \rightarrow V_{i+2} :\frac{1-\alpha}{4}, 
V_{i-3} \rightarrow V_{i-2} : \frac{1-\alpha}{4}, 
V_{i+3} \rightarrow V_{i+4} : \frac{1-\alpha}{4} $, 
and the amount of transport between the other vertices is zero.
From (\ref{W_1_inf}), we can estimate an upper bound of the Wasserstein distance 
for the above transport between probability measures.
We have the following result.
\begin{align}
W_1 \leq  \alpha + \frac{1}{2}(1-\alpha).
\end{align}
From Definition \ref{def_alpha_ricci},
we have 
$\kappa_{\alpha} \geq \frac{1}{2}(1-\alpha)$.
By Definition \ref{def_ricci}, we have the following result of the lower bound of the Ricci curvature.
\begin{align}
\kappa \geq  \frac{1}{2}.
\label{ineq.37}
\end{align}

We define a $1$-Lipschtiz function as figure $\ref{fig.13a}$.
Next, using a $1$-Lipschtiz function, we estimate an upper bound of the Ricci curvature 
by Theorem \ref{thm_w_1_lower}.
We define a $1$-Lipschtiz function as Figure $\ref{fig.13b}$.
The number in each box beside each vertex is the value of the $1$-Lipschiz function.
We have the following result from Theorem \ref{thm_w_1_lower} and this $1$- Lipschiz function.

\begin{align}
W_1 \geq  \alpha +\frac{1}{2}(1-\alpha).
\end{align}
From Definition \ref{def_alpha_ricci},  
we have 
$\kappa_{\alpha}  \leq \frac{1}{2}$.
Therefore, we have the following upper bound of the curvature
by Definition \ref{def_ricci},
\begin{align}
\kappa  \leq  \frac{1}{2}.
\label{ineq.38}
\end{align}

By (\ref{ineq.37}) and (\ref{ineq.38}), the curvature of all type A edges is
$\kappa =  \frac{1}{2}$.
\bigskip

Next, we investigate the curvature of the edge $V_{i-1}V_{i+2}$ as type $B$ in Figure \ref{fig.13b}.
We calculate the transport cost by 
$V_{i-1} \rightarrow V_{i+2} :\alpha- \frac{1-\alpha}{4}, \
V_{i-2} \rightarrow V_{i+1} : \frac{1-\alpha}{4 }, \
V_{i}  \rightarrow V_{i+3} : \frac{1-\alpha}{4}, \
V_{i-4} \rightarrow V_{ i+5}: \frac{1-\alpha}{4} \times 3$, 
and the amount of transport between the other vertices is zero.
Here, we use the fact that $d(V_{i-4}, V_{i+5}) = 3$.
This fact is derived from $n \geq 16$ as follows.
Let us define the path in Figure \ref{fig.13b}
$P_{s13b} := V_{i+6}V_{i+7} \cdots V_{i-5}V_{i-4}$.
From $n \geq 16$, $|P_{s13b}| =n-(i+5)-(i-4) \geq 3+3+1$.
This means that there is no shortcut through the vertices in this path, and we obtain
$d(V_{i-4}, V_{i+5})=3$.

From (\ref{W_1_inf}), we can estimate an upper bound of the Wasserstein distance 
for the above transport between probability measures.
We have the following result.
\begin{align}
W_1 \leq  \alpha + (1-\alpha).
\end{align}
From Definition \ref{def_alpha_ricci},
we have 
$\kappa_{\alpha }\geq 0$.
By Definition \ref{def_ricci}, 
we have the following result of a lower bound of the Ricci curvature.
\begin{align}
\kappa \geq  0.
\label{inq001}
\end{align}

We define a $1$-Lipschtiz function as figure $\ref{fig.13b}$.
The number in each box beside each vertex is the value of the $1$-Lipschiz function.
From Theorem $1$ and this $1$- Lipschiz function, 
we have the following result.
\begin{align}
W_1 \geq  0.
\end{align}
From Definition \ref{def_alpha_ricci}, 
we have 
$\kappa_{\alpha}  \leq  0$.
Therefore, we have the following upper bound of the Ricci curvature
by Definition \ref{def_ricci},
\begin{align}
\kappa  \leq  0.
\label{inq002}
\end{align}

By (\ref{inq001}) and (\ref{inq002}),  Thus, the curvature of all type B edges is
$\kappa =  0$.
\hspace{\fill}{\qed}
\end{proof}

\subsection{Proof of Proposition \ref{s14}}\label{props14proof}
\begin{proof}
\begin{figure}[htbp]
\begin{minipage}[c]{0.5\hsize}
    \centering
    \includegraphics[width=8cm]{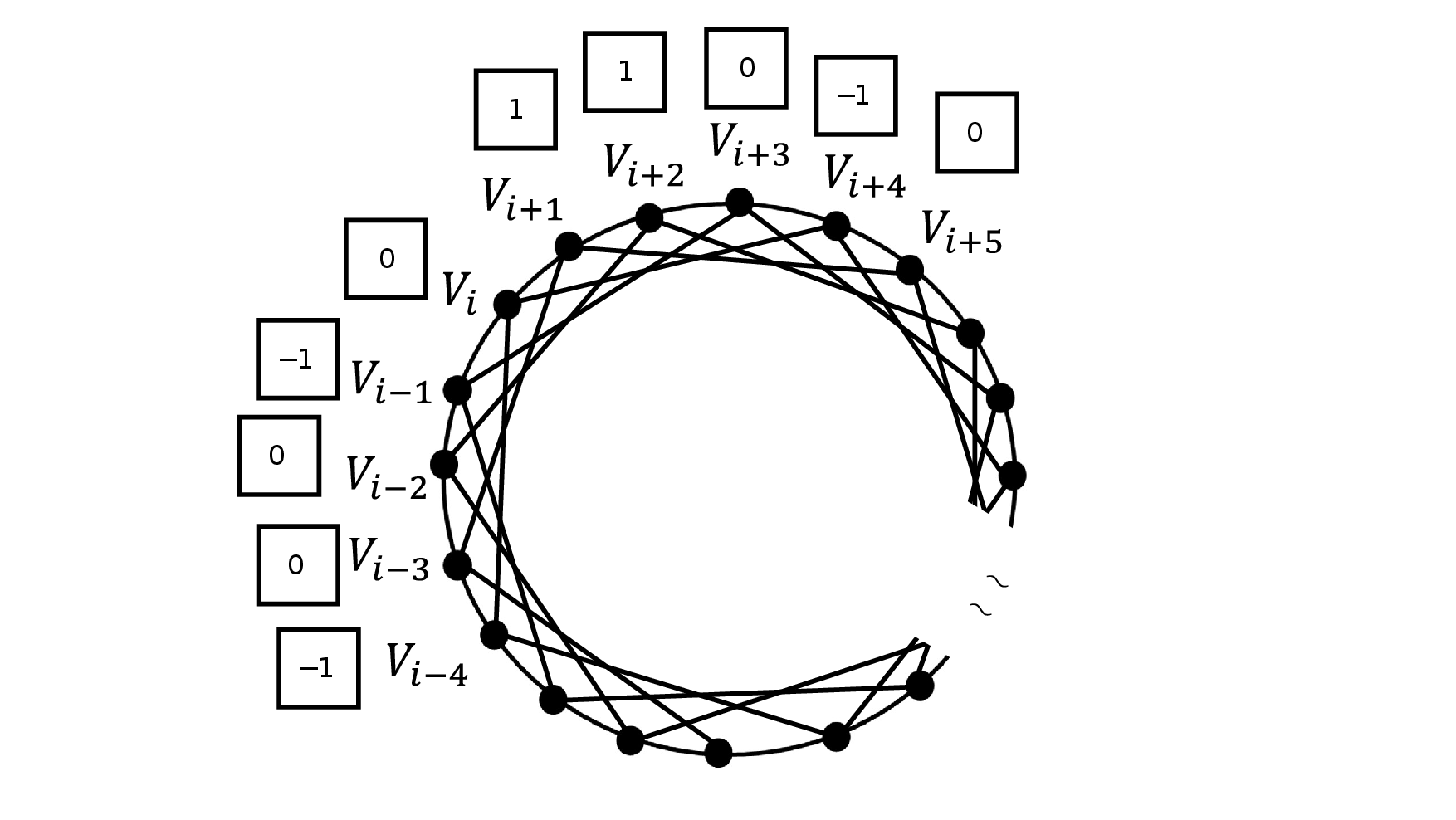}
    \caption{Type A in Cayley graphs of $\mathbb{Z}/n\mathbb{Z}$}
    \label{fig.14a}
  \end{minipage}
  \begin{minipage}[c]{0.5\hsize}
    \centering
    \includegraphics[width=8cm]{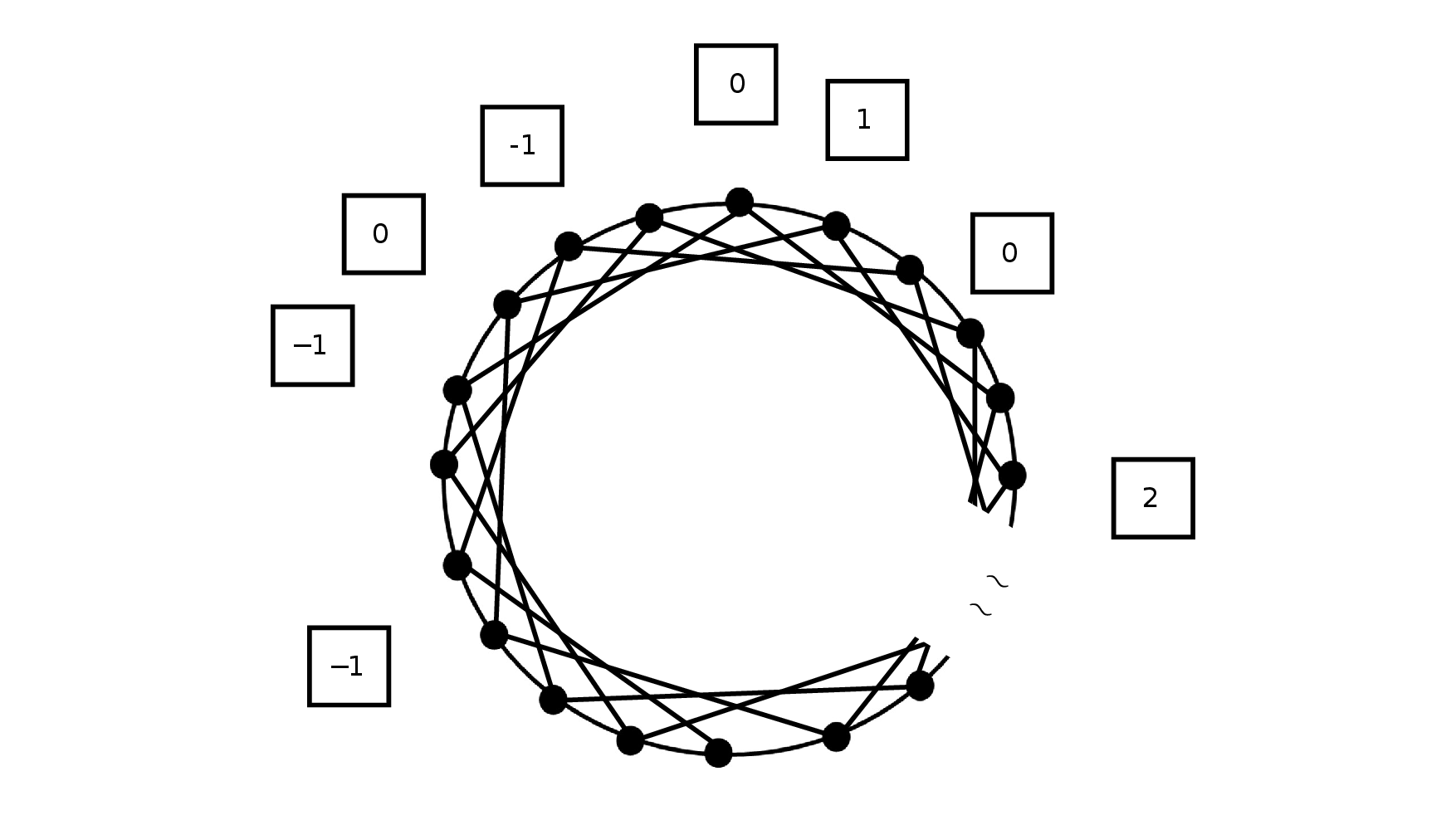}
    \caption{Type B in Cayley graphs of $\mathbb{Z}/n\mathbb{Z}$}
    \label{fig.14b}
  \end{minipage}
\end{figure}

\begin{comment}
First, we find sufficient conditions $n \geq 23$ under 
which transport cost $V_{i-4} \rightarrow V_{i+8}$ 
can be defined in Figure \ref{fig.14b}.
By considering 
$d(V_{i-4}, V_{i+8})  \geq 3  \Leftrightarrow
n-(i+8)-(i-4) \geq 4+4+3 \Leftrightarrow n \geq 23$ in Figure \ref{fig.14b}, 
For this to exist, $n \geq 23$ is required.
This is because if there exists a shortcut from $V_{i-4}$ to $V_{i+8}$ 
in a clockwise direction in Figure \ref{fig.14b}.
There exists the path $V_{i-4}V_{i-5} \cdots  V_{i+9}V_{i+8}$ 
then $n-12 \geq  8  \Leftrightarrow n \geq 20$ holds. 
Therefore, for $d(V_{i-4}, V_{i+8}) = 3 $ to exist, 
it is necessary that $n \geq 20$ does not hold.
\end{comment}
We investigate the curvature of  $V_{i}V_{i+1}$ as a type $A$ edge.
We calculate the transport cost by 
$V_{i} \rightarrow V_{i+1} :\alpha- \frac{1-\alpha}{4}, \
V_{i-4} \rightarrow V_{i-3} : \frac{1-\alpha}{4 }, \
V_{i+4} \rightarrow V_{i+5}: \frac{1-\alpha}{4}, \
V_{i-1} \rightarrow V_{i+2} : \frac{1-\alpha}{4} \times 2$, 
and the amount of transport between the other vertices is zero.

From (\ref{W_1_inf}), we can estimate an upper bound of the Wasserstein distance 
for the above transport between probability measures.
We have the following result.
\begin{align}
W_1 \leq  \alpha + \frac{3}{4}(1-\alpha).
\end{align}
From Definition \ref{def_alpha_ricci},
we have 
$\kappa_{\alpha} \geq  \frac{1}{4}(1-\alpha)$.
By Definition \ref{def_ricci}, we have the following  result of the lower bound of the Ricci curvature.
\begin{align}
\kappa \geq \frac{1}{4}.
\label{ineq.39}
\end{align}

Next, using a $1$-Lipschtiz function, we estimate an upper bound of 
the Ricci curvature by Theorem \ref{thm_w_1_lower}.
We define a $1$-Lipschtiz function as Figure $\ref{fig.14a}$.
The number in each box beside each vertex is the value of the $1$-Lipschiz function.
We have the following result from Theorem \ref{thm_w_1_lower} and this $1$- Lipschiz function.

\begin{align}
W_1 \geq    \alpha + \frac{3}{4}(1-\alpha).
\end{align}
From Definition \ref{def_alpha_ricci}, 
we have 
$\kappa_{\alpha}  \leq   \frac{1}{4}$.
Therefore, we have the following upper bound of the Ricci curvature
by Definition \ref{def_ricci},
\begin{align}
\kappa  \leq   \frac{1}{4}.
\label{ineq.40}
\end{align}

By (\ref{ineq.39}) and (\ref{ineq.40}), the curvature of all type A edges is
$\kappa = \frac{1}{4}$. 
\bigskip

Next, we investigate the curvature of $V_{i}V_{i+4}$ 
as a type $B$ edge.
 We calculate the transport cost by 
$V_{i} \rightarrow V_{i+4} :\alpha- \frac{1-\alpha}{4}, \
V_{i-1} \rightarrow V_{i+3} : \frac{1-\alpha}{4}, \
V_{i+1} \rightarrow V_{i+5} : \frac{1-\alpha}{4}, \
V_{i-4} \rightarrow V_{i+8} : \frac{1-\alpha}{4} \times 3$, 
and the amount of transport between the other vertices is zero.
Here, we use the fact that $d(V_{i-4}, V_{i+8}) = 3$.
This fact is derived from $n \geq 23$ as follows.
Let us define the path in Figure \ref{fig.14b}
$P_{s14b} := V_{i-4}V_{i-5} \cdots V_{i+7}V_{i+8}$.
Then, $|P_{s12b}| =n-(i+8)-(i-5) \geq 4+4+3$, and we get
$d(V_{i-4}, V_{i+5})=3$.

From (\ref{W_1_inf}), we can estimate an upper bound of the Wasserstein distance 
for the above transport between probability measures.
We have the following result.
\begin{align}
W_1 \leq  \alpha + (1-\alpha).
\end{align}
From Definition \ref{def_alpha_ricci},
we have 
$\kappa_{\alpha} \geq 0$.
By Definition \ref{def_ricci}, we have the following result of the lower bound of the Ricci curvature.
\begin{align}
\kappa \geq  0.
\label{ineq.41}
\end{align}

Next, using a $1$-Lipschtiz function, we estimate an upper bound of 
the Ricci curvature by Theorem \ref{thm_w_1_lower}.
We define a $1$-Lipschtiz function as Figure $\ref{fig.14b}$.
The number in each box beside each vertex is the value of the $1$-Lipschiz function.
We have the following result from Theorem \ref{thm_w_1_lower} and this $1$- Lipschiz function.

\begin{align}
W_1 \geq  \alpha +(1-\alpha).
\end{align}
From Definition \ref{def_alpha_ricci},  
we have 
$\kappa_{\alpha}  \leq 0$.
Therefore, we have the following upper bound of the Ricci curvature
by Definition \ref{def_ricci},
\begin{align}
\kappa  \leq  0.
\label{ineq.41}
\end{align}

By (\ref{ineq.40}) and (\ref{ineq.41}),  the curvature of all type B edges is
$\kappa =  0$.\rightline{\qed}
\rightline{\qed}
\end{proof}

\section*{Acknowledgements}

A.S. was supported by JSPS KAKENHI Grant Number 21K03258.


\begin{thebibliography}{99}

\bibitem{Dagli}
M.~ Dagli, O.~ Olmez, and J.~ D. ~H. ~Smith, {\it Ricci curvature, circulants, and extended matching conditions }, Bulletin of the Korean Mathematical Society, {\bf 56}, 1, 201–217 (2019).

\bibitem{hua}
W. ~Hua He, J. ~Luo, C.~ Yang, W. ~Yuan, and H.~ C. ~ Zhang,
{\it Ricci-flat Graphs with Girth Four},
{\em Acta Mathematica Sinica, English Series}. {\bf 37}, 1679–1691 (2021).

\bibitem{Lin}
 Y.~ Lin,~ L.~ Lu, and S.~ Yau, {\it Ricci curvature of graphs}, 
Tohoku Mathematical Journal, {\bf 63},4, 605-627 (2011).

\bibitem{mag}
W.~Magnus and I.~Grossman,{\it Groups and Their Graphs}, 
Anneli Lax New Mathematical Library, (1975)
\bibitem{sako&mizukai}
I.~Mizukai and A.~Sako, {\it Ollivier Ricci curvature 
of Cayley graphs for dihedral groups, generalized quaternion groups, and cyclic groups},  preprint, arXiv:2210.00860 (2022)
\bibitem{olli}
Y.~Ollivier, {\it Ricci curvature of Markov chains on metric spaces}, 
Journal of Functional Analysis, {\bf 256},3, 810–864 (2009).
\bibitem{pey}
N.~Peyerimoff, {\it Curvature notation on graphs Leeds summer school}, 
https://www.maths.dur.ac.uk/users/norbert.peyerimhoff/peyerimhoff-lecture-notes.pdf, (2019)
\bibitem{dies}
D.~Reinhard, {\it Graph Theory (3rd ed. ed.)}, Graduate Texts in Mathematics, vol. 173
\bibitem{smith}
J.~D.~H.~Smith, {\it Ricci curvature, circulants, and a matching condition}, 
{\em Discret Math}.{\bf 329}, 88-98 (2014).
\end{thebibliography}
\end{document}